\documentclass[11pt]{article}
\usepackage{amssymb}
\usepackage{graphicx}
\usepackage{setspace}
  \usepackage{paralist}
  \usepackage{longtable}
   \usepackage{multirow}
    \usepackage{rotating}
\usepackage{fancyhdr}

\usepackage{amsmath, amsthm, amssymb}
\usepackage{authblk}
\usepackage{graphicx}
\usepackage{float}
\usepackage{hyperref}
\usepackage[margin=1in]{geometry}
\usepackage{comment}
\usepackage{color}
\usepackage{mathrsfs}
\allowdisplaybreaks[4]
\numberwithin{equation}{section}
\date{}

\newtheorem{theorem}{Theorem}[section]
\newtheorem{proposition}[theorem]{Proposition}
\newtheorem{lemma}[theorem]{Lemma}

\newtheorem{definition}{Definition}[section]

\newtheorem{remark}{Remark}[section]
\theoremstyle{definition}

\newcommand{\be}{\begin{equation}}
\newcommand\ee{\end{equation}}
\newcommand\bes{\begin{eqnarray}}
\newcommand\ees{\end{eqnarray}}
\newcommand\bess{\begin{eqnarray*}}
\newcommand\eess{\end{eqnarray*}}

\title{Global strong solution to the inviscid liquid-gas two-phase flow model in $L^p$ framework}

\author{Zhigang Wu\thanks{Corresponding to:\ zhigangwu@hotmail.com},\ \ Mengqian Liu,\ \ Juanzi Cai;\\
Department of Mathematics, Donghua University, Shanghai, P. R. China.}

\begin{document}

\maketitle
\renewcommand{\thefootnote}{\fnsymbol{footnote}}

\maketitle
{{\bf  Abstract:}} This paper is dedicated to the study of the inviscid liquid-gas two-phase flow model in $\mathbb{R}^d\ (d\geq1)$. We establish the global existence of strong solutions to this system with small initial data in hybrid
Besov spaces based on general $L^p$-norms. Additionally, we obtain the decay estimates of solutions rely on the constructed Lyapunov functional.



\section{Introduction}
We are concerned with the following inviscid liquid-gas two-phase flow system in $\mathbb{R}^d$ $(d\geq1)$:
\begin{equation}\label{1.1}
\left\{\begin {array}{ll}
m_t+{\rm div}(m\mathbf{u})=0,\\
n_t+{\rm div}(n\mathbf{u})=0,\\
(m\mathbf{u})_t+{\rm div}(m\mathbf{u}\otimes \mathbf{u})+\nabla P(m,n)=-\alpha m\mathbf u,
\end{array}\right.
\end{equation}
where the unknowns $m(x,t),n(x,t), \mathbf{u}(x,t)=(u_1,\cdots,u_d)$ are the liquid mass and gas mass, the mixed velocity of the liquid and gas, respectively. The function $P=P(m,n)$ denotes the common pressure for both phases and the constant $\alpha>0$ models friction.

To clearly introduce the complicated pressure function $P(m,n)$ in (\ref{1.1}), we shall mention general isothermal compressible two-phase flow model given in \cite{Friis}, which takes the form of
\begin{equation}\label{1.2.0}
\left\{\begin{array}{ll}
(\alpha_l\rho_l)_t+{\rm div}(\alpha_l\rho_l\mathbf{u}_l)=0,\\[2mm]
(\alpha_g\rho_g)_t+{\rm div}(\alpha_g\rho_g\mathbf{u}_g)=0,\\[2mm]
(\alpha_l\rho_l\mathbf{u}_l)_t+{\rm div}(\alpha_l\rho_l\mathbf{u}_l\otimes \mathbf{u}_l)+\alpha_l\nabla P+\Delta P\nabla\alpha_l=Q_l+M_l,\\[2mm]
(\alpha_g\rho_g\mathbf{u}_g)_t+{\rm div}(\alpha_g\rho_g\mathbf{u}_g\otimes \mathbf{u}_g)+\alpha_g\nabla P+\Delta P\nabla\alpha_g=Q_g+M_g,
\end{array}\right.
\end{equation}
where the unknowns $\alpha_l,\alpha_g\in[0,1]$ denote the liquid and gas volume fractions, satisfying the fundamental relation $\alpha_l+\alpha_g=1$, the unknown variables $\rho_l$ and $\rho_g$ denote the liquid and gas densities, satisfying the equations of states $\rho_l=\rho_{l,0}+\frac{P-P_{l,0}}{a_l^2},\ \ \rho_g=\frac{P}{a_g^2}$ with the sonic speeds of the liquid $a_l$ and the sonic speeds of the gas $a_g$. Here $P_{l,0}$ and $\rho_{l,0}$ are the reference pressure and density given as constants.  The interface correction term $\Delta P=P-P^i$ with the pressure at the liquid-gas interface $P^i$, and the external forces $Q_l,\ Q_g$ (friction and gravity) are defined as $Q_l=-\alpha_l\rho_l\mathbf{u}_l$, $Q_g=-\alpha_g\rho_g\mathbf{u}_g$, and $M_l,\ M_g$ ($M_l+M_g=0$) represent interfacial forces modeling interactions between the two phases. By adding the two momentum equations in \eqref{1.2.0}, we have the following:
\begin{equation}\label{1.2.1}
\left\{\begin{array}{ll}
(\alpha_l\rho_l)_t+{\rm div}(\alpha_l\rho_l\mathbf{u}_l)=0,\\[2mm]
(\alpha_g\rho_g)_t+{\rm div}(\alpha_g\rho_g\mathbf{u}_g)=0,\\[2mm]
(\alpha_l\rho_l\mathbf{u}_l+\alpha_g\rho_g\mathbf{u}_g)_t+{\rm div}(\alpha_l\rho_l\mathbf{u}_l\otimes \mathbf{u}_l+\alpha_g\rho_g\mathbf{u}_g\otimes \mathbf{u}_g)+\nabla P(\alpha_l\rho_l,\alpha_g\rho_g)=Q_l+Q_g.
\end{array}\right.
\end{equation}
Consider a no-slip flow with $\mathbf{u}_l=\mathbf{u}_g=\mathbf{u}$ and let $m=\alpha_l\rho_l,\ n=\alpha_g\rho_g$ as in \cite{Zhang1}. Additionally, we can neglect the gas phase in the mixture momentum equation due to the fact that the liquid phase is much heavier than the gas phase, typically to the order $\rho_l/\rho_g\thicksim 10^3$. Thus, one easily finds that \eqref{1.2.1} is exactly \eqref{1.1}. For more information about the above models, we can refer to previous studies \cite{Brennen,Evje1,Ishii} and references therein.

Now, we can go back to state the pressure $P$ in (\ref{1.1}) as follows:
\begin{equation}\label{1.2.2}
P(m,n)=C_0\big(-b_1(m,n)+\sqrt{b_1^2(m,n)+b_2(m,n)}\ \big),\ m=\alpha_l\rho_l,\ n=\alpha_g\rho_g,
\end{equation}
where $C_0=\frac{a_l^2}{2}$, $b_1(m,n)=k_0-m-a_0n$ and $b_2(m,n)=4k_0a_0n$ with $k_0=\rho_{l,0}-\frac{P_{l,0}}{a_l^2}$ and $a_0=\frac{a_g^2}{a_l^2}$.
Suppose that the initial data of (\ref{1.1}) satisfy
\begin{equation}\label{1.2.4}
(m,n,\mathbf{u})(x,t)|_{t=0}=(m_0(x),n_0(x),\mathbf{u}_0(x))\rightarrow(m_\infty,n_\infty,\mathbf{0}),\ as\ |x|\rightarrow\infty,
\end{equation}
where $n_\infty\geq0$ and $m_\infty>(1-{\rm sgn}n_\infty)k_0\geq0$ are given constants.

As far as we known, mathematical models for mixtures of two
phases have been studied for quite a few years, especially for the viscous liquid-gas two-phase flow model as follows:
\begin{equation}\label{1.2.3}
\left\{\begin {array}{ll}
m_t+{\rm div}(m\mathbf{u})=0,\\
n_t+{\rm div}(n\mathbf{u})=0,\\
(m\mathbf{u})_t+{\rm div}(m\mathbf{u}\otimes \mathbf{u})+\nabla P(m,n)=\mu\Delta \mathbf{u}+(\lambda+\mu)\nabla{\rm div}\mathbf{u}.
\end{array}\right.
\end{equation}
For the one-dimensional case of \eqref{1.2.3}, the existence, uniqueness, regularity, asymptotic behavior and decay rate of solutions have been investigated in \cite{Evje2,Evje3,Evje4,Evje5,Evje6,Evje7,Evje8,Fan,Guo,Liu,Tan1,Yao4,Yao5}. For the results about 2D version of the model \eqref{1.2.3}, Yao, Zhang and Zhu \cite{Yao2,Yao3} studied the existence of the global weark solutions and blow-up criterion. The incompressible limit of the model \eqref{1.2.3} with periodic boundary conditions was investigated by Yao, Zhang and Zi \cite{Yao6}.  For a 3D version of the model \eqref{1.2.3}, the blow-up criterions have been studied by Hou $et\ al.$ \cite{Hou}, Wen $et\ al.$ \cite{Wen} and Yao $et\  al.$ \cite{Yao1}.
Under the framework of Besov spaces, Hao and Li \cite{Hao} proved the existence and uniqueness of the global solution to the Cauchy problem \eqref{1.2.3} in $L^2$ framework where the initial data are close to a constant equilibrium state. Then, Zhang and Zhu \cite{Zhang3} proved the global existence and optimal convergence rates of the strong solutions when $H^2$-norm of the initial perturbation around a constant state is sufficiently small and its $L^1$-norm is bounded. In the $L^p\ (1<p\leq d)$ framework, Xu and Yuan \cite{XuF} proved the local well-posedness for large data, provided that the initial liquid mass is bounded away from zero. Cui, Bie and Yao \cite{Cui} employed the Lagrangian approach and improved the local well-posedness result, where the range of $p$ is relaxed as $1<p<2d$. Xu and Zhu \cite{XuJ1} established the global well-posedness and optimal time-dacay rates of strong solutions to \eqref{1.2.3} in the $L^p$ critical regularity framework.

This paper is to analyze the friction effect of the damping on the qualitative behavior of the solution of
the Cauchy problem to the two-phase flow system (\ref{1.1}). There are rather few mathematical results on the inviscid liquid-gas two-phase flow model (\ref{1.1}) compared with the viscous one (\ref{1.2.3}). A first work of the model (\ref{1.1}) was represented by Zhang \cite{Zhang1}, where global existence and decay rates of solutions were established when the initial perturbation is small in $H^3$-norm and bounded in $L^p$-norm ($1\leq p<\frac{6}{5}$). Then, Wang \cite{wang1} reconsidered the global existence and large time behavior of the solutions to the system \eqref{1.1} under only the smallness assumption on $H^3$-norm of the initial perturbation. The free boundary problem for the inviscid two-phase flow model (\ref{1.1}) with moving physical vacuum boundary was investigated in  \cite{Zhang2}. 

So far, up to now, there are no result on the strong solution for the Cauchy problem of (\ref{1.1}), especially in $L^p$-framework of homogeneous (hybrid) Besov spaces. The motivation of this paper is to give a positive answer, and our result is a large extension of the previous works in \cite{wang1,Zhang2}. Due to the similarity of the model \eqref{1.1} to the compressible Euler equations with damping, we can apply some ideas developed in proving the existence, stability and convergence rates of solutions to the compressible Euler equations with damping in \cite{Crin-Barat1,Crin-Barat2,Crin-Barat3} to deal with our problem. More specifically,  Barat and Danchin \cite{Crin-Barat3} investigated the general partially dissipative hyperbolic systems in multi-dimensions in critical $L^2$ Besov space (homogeneous), which is different from the previous works by Xu and Kawashima \cite{XuJ3,XuJ4} in non-homogeneous Besov space. In particular, among them, the damped compressible Euler system is included. Later on, they in \cite{Crin-Barat1,Crin-Barat2} further studied these systems in critical $L^p$ Besov space (homogeneous) with $p\geq2$ for the low frequency of the solution, which is absolutely different from the compressible Navier-Stokes equations in $L^p$-framework by Charve and Danchin \cite{Charve}, Chen, Miao and Zhang \cite{Chen} and Haspot \cite{Haspot}. One can also refer to \cite{Danchin1,Danchin2,Danchin5,Danchin6,Xin,XuJ1} for the other well-known results in this direction for the compressible Navier-Stokes equations. This difference between these two typical models is arising from two facts: the wave operators exist in the high frequency of the spectrum in the damped compressible Euler system and the wave operators exist in the low frequency of the spectrum of the compressible Navier-Stokes equations, and the well-posedness of the hyperbolic system cannot generally be entirely justified in $L^p$-based
spaces for $p\neq2$ (see Brenner \cite{Brenner}). In the spirit of \cite{Crin-Barat1,Crin-Barat2}, some related results were also developed for Jin-Xin model in \cite{Crin-Barat5,Crin-Barat6} for $L^2$-framework and $L^p$-framework, respectively, and \cite{Crin-Barat4} for a hyperbolic-parabolic chemotaxis system for vasculogenesis.

Now, we shall present some new difficulties arising from our analysis on the system (\ref{1.1}). First, the system (\ref{1.1}) does not satisfy the (SK) condition developed by Kawashima and Shizuta, which is used in \cite{Crin-Barat1,Crin-Barat2} for the general partially dissipative hyperbolic systems (including the damped compressible Euler system). Second,  the pressure function $P(n,m)$ in (\ref{1.1}) is rather complex than that of the damped compressible Euler system with the $\gamma$-law in \cite{Crin-Barat1,Crin-Barat2},  and the general pressure function satisfying $P'(\bar\rho)>0$ for the constant reference density $\bar\rho>0$ in Xu \cite{XuJ2}. We partially used the key composition estimate constructed in \cite{XuJ2} to treat the pressure function (\ref{1.2.2}) together with the Taylor expansion of the pressure term, however, the non-dissipative gas phase $n$ (or $\tilde{c}$ given in the following perturbation) in the pressure function also makes the nonlinear analysis more complicated. For instance, see the details in (\ref{3.44}).

Before state our main result, we shall reformulate the original system (\ref{1.1}) as the same argument in \cite{Zhang1,wang1} due to the presence of non-dissipative gas phase $n$. Specifically speaking, we study an equivalent system about the variable $(P,\mathbf{u},\tilde c)$ to (\ref{1.1}). That is, by letting $\tilde{c}=a_0(\frac{n}{m}-\frac{n_\infty}{m_\infty})$, the system (\ref{1.1}) can be rewritten as
\begin{equation}\label{1.3}
\left\{\begin{array}{ll}
P_t+(\mathbf{u}\cdot\nabla) P+P_m m\nabla\cdot\mathbf{u}=0,\\[2mm]
\mathbf{u}_t+(\mathbf{u}\cdot\nabla)\mathbf{u}+\frac{1}{m}\nabla P=-\alpha\mathbf{u},\\[2mm]
\tilde{c}_t+(\mathbf{u}\cdot\nabla)\tilde{c}=0,
\end{array}\right.
\end{equation}
where
\begin{equation*}
\begin{array}{ll}
P(m,\tilde{c})=C_0\big\{(1+c_\infty+\tilde{c})m-k_0+\sqrt{[(1+c_\infty+\tilde{c})m-k_0]^2+4k_0(\tilde{c}+c_\infty)m}\big\},\\[3mm]
P_m(m,\tilde{c})=C_0\big(1+\tilde{c}+c_\infty
+\frac{-(k_0-m-m(\tilde{c}+c_\infty))(1+\tilde{c}+c_\infty)+2k_0(\tilde{c}+c_\infty)}{\sqrt{(k_0-m-m(\tilde{c}+c_\infty))^2+4k_0m(\tilde{c}+c_\infty)}}\big),\\[3mm]
m(P,\tilde{c})=\frac{P^2+2k_0C_0P}{2C_0(1+\tilde{c}+c_\infty)P+4k_0C_0^2(\tilde{c}+c_\infty)},
\end{array}
\end{equation*}
with $c_\infty=\frac{a_0n_\infty}{m_\infty}$. Then the initial data of (\ref{1.3}) correspondingly satisfy
\begin{equation}\label{1.4}
(P,\mathbf{u},\tilde{c})(x,t)|_{t=0}=(P_0(x),\mathbf{u}_0(x),\tilde{c}_0(x))\rightarrow(P_\infty,\mathbf{0},0),\ as\ |x|\rightarrow\infty,
\end{equation}
where $P_\infty>0$ is a given constant satisfying
$$P_\infty=P(m_\infty,0)=C_0\big\{-(k_0-m_\infty-m_\infty c_\infty)+\sqrt{(k_0-m_\infty-m_\infty c_\infty)^2+4k_0m_\infty c_\infty}\big\}.$$
Denote $\kappa_1=\frac{1}{\sqrt{P_m(m_\infty,0)}m_\infty}$ and $\kappa_2=\sqrt{P_m(m_\infty,0)}$. Then using the perturbation $(P,\mathbf{u},\tilde{c})\rightarrow(\tilde{P}+P_\infty,\kappa_1\tilde{\mathbf{u}},\tilde{c})$, we rewrite the initial value problem (\ref{1.3})-(\ref{1.4}) as the following system
\begin{equation}\label{1.5}
\left\{\begin{array}{ll}
\tilde{P}_t+\kappa_2\nabla\cdot\tilde{\mathbf{u}}=\tilde{G}_1,\\[2mm]
\tilde{\mathbf{u}}_t+\kappa_2\nabla\tilde{P}+\alpha\tilde{\mathbf{u}}=\tilde{G}_2,\\[2mm]
\tilde{c}_t+\kappa_1(\tilde{\mathbf{u}}\cdot\nabla)\tilde{c}=0,\\[2mm]
(\tilde{P},\tilde{\mathbf{u}},\tilde{c})(x,t)|_{t=0}=(\tilde{p}_0,\tilde{\mathbf{u}}_0,\tilde{c}_0)\rightarrow(0,\mathbf{0},0),\ as\ |x|\rightarrow\infty,
\end{array}\right.
\end{equation}
with the source terms
\begin{equation*}
\begin{array}{ll}
\tilde{G}_1=\tilde{G}_1(\tilde{P},\tilde{\mathbf{u}},\tilde{c})=-\kappa_1h(\tilde{P},\tilde{c})\nabla\cdot\tilde{\mathbf{u}}-\kappa_1\tilde{\mathbf{u}}\cdot\nabla\tilde{P},\\[2mm]
\tilde{G}_2=\tilde{G}_2(\tilde{P},\tilde{\mathbf{u}},\tilde{c})=-\kappa_1(\tilde{\mathbf{u}}\cdot\nabla)\tilde{\mathbf{u}}-\frac{1}{\kappa_1}(\frac{1}{m}-\frac{1}{m_\infty})\nabla\tilde{P},
\end{array}
\end{equation*}
and the nonlinear function $h(\tilde{P},\tilde{c})$ are defined by
$$h(\tilde{P},\tilde{c})=P_m(m,\tilde{c})m-P_m(m_\infty,0)m_\infty.$$

Now, we are in a position to state our results on the global existence and decay rate of the strong solution in $L^p$ Besov space for the Cauchy problem (\ref{1.5}).

\begin{theorem}\label{1 1.1}
Let $2\leq p\leq 4$ if $d=1$, and $2\leq p\leq {\rm min}\{4,\frac{2d}{d-2}\}$ if $d\geq 2$. There exists a small constant $\eta_1\geq0$ such that if
\begin{equation}\label{1.5(1)}
\mathcal{X}_{p,0}\triangleq \|\tilde P_0\|^l_{\dot B_{p,1}^{\frac{d}{p}-1}}+\|\tilde{\mathbf{u}}_0\|^l_{\dot B_{p,1}^{\frac{d}{p}}}+\|\tilde c_0\|^l_{\dot B_{2,1}^{\frac{d}{2}-1}}+\|(\tilde P_0,\tilde{\mathbf{u}}_0,\tilde c_0)\|^h_{\dot B_{2,1}^{\frac{d}{2}+1}}\leq\eta_1,
\end{equation}
then, the Cauchy problem (\ref{1.5}) admits a unique global strong solution $(\tilde P,\tilde{\mathbf{u}},\tilde c)$ satisfying $P^l\in\mathcal{C}(\mathbb{R}^+;\dot B_{p,1}^{\frac{d}{p}-1})$, $ \tilde{\mathbf{u}}^l\in\mathcal{C}(\mathbb{R}^+;\dot B_{p,1}^{\frac{d}{p}}),\ \tilde c^l\in\mathcal{C}(\mathbb{R}^+;\dot B_{2,1}^{\frac{d}{2}-1}),\ (\tilde P^h,\tilde{\mathbf{u}}^h,\tilde c^h)\in\mathcal{C}(\mathbb{R}^+;\dot B_{2,1}^{\frac{d}{2}+1})$ and
\begin{equation}\label{1.6}
\begin{array}{ll}
\|\tilde P\|^l_{\tilde L_t^\infty(\dot B_{p,1}^{\frac{d}{p}-1})}+\|\tilde P\|^l_{L_t^1(\dot B_{p,1}^{\frac{d}{p}+1})}+\|\tilde{\mathbf{u}}\|^l_{\tilde L_t^\infty(\dot B_{p,1}^{\frac{d}{p}})}+\|\tilde{\mathbf{u}}\|^l_{L_t^1(\dot B_{p,1}^{\frac{d}{p}})}\\[5mm]
+\|\tilde c\|^l_{\tilde L_t^\infty(\dot B_{2,1}^{\frac{d}{2}-1})}+\|(\tilde P,\tilde{\mathbf{u}},\tilde c)\|^h_{\tilde L_t^\infty(\dot B_{2,1}^{\frac{d}{2}+1})}+\|(\tilde P,\tilde{\mathbf{u}})\|^h_{L_t^1(\dot B_{2,1}^{\frac{d}{2}+1})}
\lesssim \mathcal{X}_{p,0}.
\end{array}
\end{equation}
\end{theorem}
\begin{remark}Theorem \ref{1 1.1} gives the first global-in-time
existence of the inviscid  liquid-gas two-phase
flow in the Besov space of  $L^p$-type, which improves the previous efforts in \cite{wang1,Zhang2}. One can see that the low frequency of dissipative variables $(\tilde{P},\tilde{u})$ can be bounded in the $L^p$-type Besov space with $p\geq2$ and the low frequency of non-dissipative variable $\tilde{c}$ has to be bounded in the $L^2$-type Besov space.
\end{remark}

\begin{remark}The process of deriving Theorem \ref{1 1.1} can be applied to the other compressible fluid model with damping, for instance, the system of compressible adiabatic flow through porous media in \cite{Wuzg}.

\end{remark}

\begin{theorem}\label{1 1.2}
For $2\leq p\leq 4$ if $d=1$, and $2\leq p\leq {\rm min}\{4,\frac{2d}{d-2}\}$ if $d\geq 2$, let $(\tilde{P},\tilde{\mathbf{u}},\tilde{c})$ be the global solution given in Theorem \ref{1 1.1}. In addition to (\ref{1.5(1)})
and assume that
\begin{equation}\label{1.10}
\|(\tilde{P}_0,\tilde{\mathbf{u}}_0,\tilde{c}_0)\|^l_{\dot{B}_{p,\infty}^{-\sigma_1}}\ is\ uniformly\ bounded\ with\ -\frac{d}{p}\leq-\sigma_1<\frac{d}{p}-1.
\end{equation}
Then, we have the following decay estimates:
\begin{align}
\|\Lambda^\sigma\tilde{P}\|_{L^p}\lesssim(1+t)^{-\frac{\sigma_1+\sigma}{2}},\ \ &if\ -\sigma_1<\sigma\leq\frac{d}{p}-1,\label{1.7}\\
\|\Lambda^\sigma\tilde{\mathbf{u}}\|_{L^p}\lesssim(1+t)^{-\frac{\sigma+\sigma_1+1}{2}},\ \ &if\ -\sigma_1<\sigma\leq\frac{d}{p}-2,\label{1.8}\\
\|\Lambda^\sigma\tilde{\mathbf{u}}\|_{L^p}\lesssim(1+t)^{-\frac{\frac{d}{p}-1+\sigma_1}{2}},\ \ &if\ \frac{d}{p}-2<\sigma\leq\frac{d}{p}.\label{1.9}
\end{align}
\end{theorem}

\begin{remark}Choosing $p=2,d=3$ and $\sigma=0,\sigma_1=\frac{d}{2}$, one has the optimal decay of $\|\tilde{P}\|_{L^2}\lesssim(1+t)^{-\frac{3}{4}}$, which were shown in \cite{wang1,Zhang2} in the
framework of Sobolev spaces with higher regularity.
\end{remark}

\section{Priori estimate}
\quad\quad In this section, we give the key a-priori estimates leading to the global existence of solutions for \eqref{1.5}.
\begin{proposition}\label{3 1.1}
Let $2\leq p\leq 4$ if $d=1$, and $2\leq p\leq {\rm min}\{4,\frac{2d}{d-2}\}$ if $d\geq 2$. For given time $T>0$, let $(\tilde P,\tilde{\mathbf{u}},\tilde c)$ be a solution to the Cauchy problem (\ref{1.5}) satisfying for $0\leq t\leq T$ that
$$\|\tilde{P},\tilde{c}\|_{L_t^\infty(L^\infty)}\leq 1.$$
Then, for all $0\leq t\leq T$, it holds that
\begin{align}
\mathcal{X}_p(t)\leq C(\mathcal{X}_{p,0}+\mathcal{X}_p^2(t)+\mathcal{X}_p^3(t)+\mathcal{X}_p^4(t)),
\end{align}
where $\mathcal{X}_p(t)$ is defined by
\begin{align}
\mathcal{X}_p(t)\triangleq\|\tilde P\|^l_{\tilde L_t^\infty(\dot B_{p,1}^{\frac{d}{p}-1})}+\|\tilde P\|^l_{L_t^1(\dot B_{p,1}^{\frac{d}{p}+1})}+\|\tilde{\mathbf{u}}\|^l_{\tilde L_t^\infty(\dot B_{p,1}^{\frac{d}{p}})}+\|\tilde{\mathbf{u}}\|^l_{L_t^1(\dot B_{p,1}^{\frac{d}{p}})}\notag\\[3mm]
+\|\tilde c\|^l_{\tilde L_t^\infty(\dot B_{2,1}^{\frac{d}{2}-1})}+\|(\tilde P,\tilde{\mathbf{u}},\tilde c)\|^h_{\tilde L_t^\infty(\dot B_{2,1}^{\frac{d}{2}+1})}+\|(\tilde P,\tilde{\mathbf{u}})\|^h_{L_t^1(\dot B_{2,1}^{\frac{d}{2}+1})}.
\end{align}
\end{proposition}

\subsection{The low-frequency estimate}
\quad\quad Let $\mathbf{Z}=\kappa_2\nabla\tilde{P}+\alpha\tilde{\mathbf{u}}$. Then, we obtain
\begin{equation}\label{3.1}
\left\{\begin{array}{ll}
\tilde{P}_t-\frac{\kappa_2^2}{\alpha}\Delta\tilde{P}=-\frac{\kappa_2}{\alpha}{\rm div}\mathbf{Z}+\tilde{G}_1:=F_1,\\[2mm]
\frac{1}{\alpha}\mathbf{Z}_t+\mathbf{Z}=\frac{\kappa_2^3}{\alpha^2}\nabla\Delta\tilde{P}-\frac{\kappa_2^2}{\alpha^2}\nabla{\rm div}\mathbf{Z}+\frac{\kappa_2}{\alpha}\nabla\tilde{G}_1+\tilde{G}_2:=F_2.
\end{array}\right.
\end{equation}
By virtue of Lemma \ref{2 2.1} for $(\ref{3.1})_1$, we have
\begin{equation}\label{3.2}
\begin{array}{ll}
\|\tilde{P}\|^l_{\tilde{L}_T^{\infty}(\dot{B}_{p,1}^{\frac{d}{p}-1})}
+\|\tilde{P}\|^l_{\tilde{L}_T^{1}(\dot{B}_{p,1}^{\frac{d}{p}+1})}\lesssim \|\tilde{P}_0\|^l_{\dot{B}_{p,1}^{\frac{d}{p}-1}}+\|F_1\|^l_{\tilde{L}_T^{1}(\dot{B}_{p,1}^{\frac{d}{p}-1})}\\[6mm]
\quad\quad\quad\quad\quad\quad\quad\quad\quad\lesssim \|\tilde{P}_0\|^l_{\dot{B}_{p,1}^{\frac{d}{p}-1}}+\|\mathbf{Z}\|^l_{\tilde{L}_T^{1}(\dot{B}_{p,1}^{\frac{d}{p}})}
+\|\tilde{G}_1\|^l_{\tilde{L}_T^{1}(\dot{B}_{p,1}^{\frac{d}{p}-1})}.
\end{array}
\end{equation}
As for $\mathbf{Z}$, according to Lemma \ref{2 2.2} for $(\ref{3.1})_2$, we obtain
\begin{equation}\label{3.3}
\begin{array}{ll}
\|\mathbf{Z}\|^l_{\tilde{L}_T^{\infty}(\dot{B}_{p,1}^{\frac{d}{p}})}
+\|\mathbf{Z}\|^l_{\tilde{L}_T^{1}(\dot{B}_{p,1}^{\frac{d}{p}})}\lesssim \|\mathbf{Z}_0\|^l_{\dot{B}_{p,1}^{\frac{d}{p}}}+\|F_2\|^l_{\tilde{L}_T^{1}(\dot{B}_{p,1}^{\frac{d}{p}})}.
\end{array}
\end{equation}
By using $\|u\|^l_{\dot{B}_{p,1}^{s_1}}\lesssim 2^{js}\|u\|^l_{\dot{B}_{p,1}^{s_1-s}}\ (j\leq j_0)$ and the definition of $F_2$, one has
\begin{equation}\label{3.4}
\begin{array}{ll}
\|\mathbf{Z}\|^l_{\tilde{L}_T^{\infty}(\dot{B}_{p,1}^{\frac{d}{p}})}
+\|\mathbf{Z}\|^l_{\tilde{L}_T^{1}(\dot{B}_{p,1}^{\frac{d}{p}})}
\lesssim \|\mathbf{Z}_0\|^l_{\dot{B}_{p,1}^{\frac{d}{p}}}
+2^{2j}\|\tilde{P}\|^l_{\tilde{L}_T^{1}(\dot{B}_{p,1}^{\frac{d}{p}+1})}\\[6mm]
\quad\quad\quad\quad\quad\quad\quad\quad\quad+2^{2j}\|\mathbf{Z}\|^l_{\tilde{L}_T^{1}(\dot{B}_{p,1}^{\frac{d}{p}})}
+2^{2j}\|\tilde{G}_1\|^l_{\tilde{L}_T^{1}(\dot{B}_{p,1}^{\frac{d}{p}-1})}
+\|\tilde{G}_2\|^l_{\tilde{L}_T^{1}(\dot{B}_{p,1}^{\frac{d}{p}})}.
\end{array}
\end{equation}
Then, inserting (\ref{3.3}) into (\ref{3.4}), we deduce
\begin{equation}\label{3.5}
\begin{array}{ll}
\|\mathbf{Z}\|^l_{\tilde{L}_T^{\infty}(\dot{B}_{p,1}^{\frac{d}{p}})}
+\|\mathbf{Z}\|^l_{\tilde{L}_T^{1}(\dot{B}_{p,1}^{\frac{d}{p}})}
\lesssim \|\mathbf{Z}_0\|^l_{\dot{B}_{p,1}^{\frac{d}{p}}}
+2^{2j}\|\tilde{P}_0\|^l_{\dot{B}_{p,1}^{\frac{d}{p}-1}}\\[6mm]
\quad\quad\quad\quad\quad\quad+2\cdot2^{2j}\|\mathbf{Z}\|^l_{\tilde{L}_T^{1}(\dot{B}_{p,1}^{\frac{d}{p}})}
+2\cdot2^{2j}\|\tilde{G}_1\|^l_{\tilde{L}_T^{1}(\dot{B}_{p,1}^{\frac{d}{p}-1})}
+\|\tilde{G}_2\|^l_{\tilde{L}_T^{1}(\dot{B}_{p,1}^{\frac{d}{p}})}.
\end{array}
\end{equation}
We are able to choose the integer $j_0$ such that $2\cdot2^{2j}\leq2^{4j_0}<\frac{1}{2}$. Therefore, combining (\ref{3.2}) and (\ref{3.5}), we obtain
\begin{equation}\label{3.6}
\begin{array}{ll}
\|\tilde{P}\|^l_{\tilde{L}_T^{\infty}(\dot{B}_{p,1}^{\frac{d}{p}-1})}
+\|\tilde{P}\|^l_{\tilde{L}_T^{1}(\dot{B}_{p,1}^{\frac{d}{p}+1})}
+\|\mathbf{Z}\|^l_{\tilde{L}_T^{\infty}(\dot{B}_{p,1}^{\frac{d}{p}})}
+\|\mathbf{Z}\|^l_{\tilde{L}_T^{1}(\dot{B}_{p,1}^{\frac{d}{p}})}\\[6mm]
\lesssim \|\mathbf{Z}_0\|^l_{\dot{B}_{p,1}^{\frac{d}{p}}}
+\|\tilde{P}_0\|^l_{\dot{B}_{p,1}^{\frac{d}{p}-1}}
+\|\tilde{G}_1\|^l_{\tilde{L}_T^{1}(\dot{B}_{p,1}^{\frac{d}{p}-1})}
+\|\tilde{G}_2\|^l_{\tilde{L}_T^{1}(\dot{B}_{p,1}^{\frac{d}{p}})}.
\end{array}
\end{equation}
Next, we bound the nonlinear term $\tilde{G}_1$ and $\tilde{G}_2$. Due to Lemma \ref{2 2.7} and $\dot{B}_{2,1}^{\frac{d}{2}}\hookrightarrow\dot{B}_{p,1}^{\frac{d}{p}}(p\geq 2)$, we get
\begin{equation}\label{3.7}
\begin{array}{ll}
\|\tilde{G}_1\|^l_{\tilde{L}_T^{1}(\dot{B}_{p,1}^{\frac{d}{p}-1})}
\!&\lesssim\|\kappa_1h(\tilde{P},\tilde{c})\nabla\cdot\tilde{\mathbf{u}}\|^l_{\tilde{L}_T^{1}(\dot{B}_{p,1}^{\frac{d}{p}-1})}
+\|\kappa_1\tilde{\mathbf{u}}\cdot\nabla\tilde{P}\|^l_{\tilde{L}_T^{1}(\dot{B}_{p,1}^{\frac{d}{p}-1})}\\[6mm]
&\lesssim\|h(\tilde{P},\tilde{c})\|_{\tilde{L}_T^{\infty}(\dot{B}_{p,1}^{\frac{d}{p}})}\|{\rm div}\tilde{\mathbf{u}}\|_{\tilde{L}_T^{1}(\dot{B}_{p,1}^{\frac{d}{p}-1})}
+\|\tilde{\mathbf{u}}\|_{\tilde{L}_T^{1}(\dot{B}_{p,1}^{\frac{d}{p}})}\|\nabla\tilde{P}\|_{\tilde{L}_T^{\infty}(\dot{B}_{p,1}^{\frac{d}{p}-1})}\\[6mm]
&\lesssim(\|(\tilde{P},\tilde{c})\|^l_{\tilde{L}_T^{\infty}(\dot{B}_{p,1}^{\frac{d}{p}-1})}+\|(\tilde{P},\tilde{c})\|^h_{\tilde{L}_T^{\infty}(\dot{B}_{2,1}^{\frac{d}{2}+1})})
(\|\tilde{\mathbf{u}}\|^l_{\tilde{L}_T^{1}(\dot{B}_{p,1}^{\frac{d}{p}})}+\|\tilde{\mathbf{u}}\|^h_{\tilde{L}_T^{1}(\dot{B}_{2,1}^{\frac{d}{2}+1})})\\[6mm]
&\quad+(\|\tilde{\mathbf{u}}\|^l_{\tilde{L}_T^{1}(\dot{B}_{p,1}^{\frac{d}{p}})}+\|\tilde{\mathbf{u}}\|^h_{\tilde{L}_T^{1}(\dot{B}_{2,1}^{\frac{d}{2}+1})})
(\|\tilde{P}\|^l_{\tilde{L}_T^{\infty}(\dot{B}_{p,1}^{\frac{d}{p}})}+\|\tilde{P}\|^h_{\tilde{L}_T^{\infty}(\dot{B}_{2,1}^{\frac{d}{2}+1})})
\lesssim \mathcal{X}_p^2(t),
\end{array}
\end{equation}
and
\begin{equation}\label{3.8}
\begin{array}{ll}
\|\tilde{G}_2\|^l_{\tilde{L}_T^{1}(\dot{B}_{p,1}^{\frac{d}{p}})}
\!&\lesssim\|\kappa_1\tilde{\mathbf{u}}\cdot\nabla\tilde{\mathbf{u}}\|^l_{\tilde{L}_T^{1}(\dot{B}_{p,1}^{\frac{d}{p}})}
+\|\frac{1}{\kappa_1}(\frac{1}{m}-\frac{1}{m_\infty})\nabla\tilde{P}\|^l_{\tilde{L}_T^{1}(\dot{B}_{p,1}^{\frac{d}{p}})}\\[6mm]
&\lesssim\|\tilde{\mathbf{u}}\|_{\tilde{L}_T^{\infty}(\dot{B}_{p,1}^{\frac{d}{p}})}\|\nabla\tilde{\mathbf{u}}\|_{\tilde{L}_T^{1}(\dot{B}_{p,1}^{\frac{d}{p}-1})}
+\|\frac{1}{m}-\frac{1}{m_\infty}\|_{\tilde{L}_T^{\infty}(\dot{B}_{p,1}^{\frac{d}{p}-1})}\|\nabla\tilde{P}\|_{\tilde{L}_T^{1}(\dot{B}_{p,1}^{\frac{d}{p}})}\\[6mm]
&\lesssim(\|\tilde{\mathbf{u}}\|^l_{\tilde{L}_T^{1}(\dot{B}_{p,1}^{\frac{d}{p}})}+\|\tilde{\mathbf{u}}\|^h_{\tilde{L}_T^{1}(\dot{B}_{2,1}^{\frac{d}{2}+1})})
(\|\tilde{\mathbf{u}}\|^l_{\tilde{L}_T^{1}(\dot{B}_{p,1}^{\frac{d}{p}})}+\|\tilde{\mathbf{u}}\|^h_{\tilde{L}_T^{1}(\dot{B}_{2,1}^{\frac{d}{2}+1})})\\[6mm]
&\quad+(\|(\tilde{P},\tilde{c})\|^l_{\tilde{L}_T^{\infty}(\dot{B}_{p,1}^{\frac{d}{p}-1})}+\|(\tilde{P},\tilde{c})\|^h_{\tilde{L}_T^{\infty}(\dot{B}_{2,1}^{\frac{d}{2}+1})})
(\|\tilde{P}\|^l_{\tilde{L}_T^{1}(\dot{B}_{p,1}^{\frac{d}{p}+1})}+\|\tilde{P}\|^h_{\tilde{L}_T^{1}(\dot{B}_{2,1}^{\frac{d}{2}+1})})\\[6mm]
&\lesssim \mathcal{X}_p^2(t).
\end{array}
\end{equation}
Combining (\ref{3.6}), (\ref{3.7}) and (\ref{3.8}), we have
\begin{equation}\label{3.9}
\begin{array}{ll}
\|\tilde{P}\|^l_{\tilde{L}_T^{\infty}(\dot{B}_{p,1}^{\frac{d}{p}-1})}
+\|\tilde{P}\|^l_{\tilde{L}_T^{1}(\dot{B}_{p,1}^{\frac{d}{p}+1})}
+\|\mathbf{Z}\|^l_{\tilde{L}_T^{\infty}(\dot{B}_{p,1}^{\frac{d}{p}})}
+\|\mathbf{Z}\|^l_{\tilde{L}_T^{1}(\dot{B}_{p,1}^{\frac{d}{p}})}\\[6mm]
\lesssim \|\mathbf{Z}_0\|^l_{\dot{B}_{p,1}^{\frac{d}{p}}}
+\|\tilde{P}_0\|^l_{\dot{B}_{p,1}^{\frac{d}{p}-1}}+\mathcal{X}_p(t)^2.
\end{array}
\end{equation}
From $\mathbf{Z}=\kappa_2\nabla\tilde{P}+\alpha\tilde{\mathbf{u}}$, we easily get
\begin{equation}\label{3.10}
\begin{array}{ll}
\|\tilde{\mathbf{u}}\|^l_{\tilde{L}_T^{\infty}(\dot{B}_{p,1}^{\frac{d}{p}})}
&\lesssim2^{2j_0}\|\tilde{P}\|^l_{\tilde{L}_T^{\infty}(\dot{B}_{p,1}^{\frac{d}{p}-1})}
+\|\mathbf{Z}\|^l_{\tilde{L}_T^{\infty}(\dot{B}_{p,1}^{\frac{d}{p}})}\\[6mm]
&\lesssim\|\tilde{P}\|^l_{\tilde{L}_T^{\infty}(\dot{B}_{p,1}^{\frac{d}{p}-1})}
+\|\mathbf{Z}\|^l_{\tilde{L}_T^{\infty}(\dot{B}_{p,1}^{\frac{d}{p}})},\\[6mm]
\|\tilde{\mathbf{u}}\|^l_{\tilde{L}_T^{1}(\dot{B}_{p,1}^{\frac{d}{p}})}
&\lesssim\|\tilde{P}\|^l_{\tilde{L}_T^{1}(\dot{B}_{p,1}^{\frac{d}{p}+1})}
+\|\mathbf{Z}\|^l_{\tilde{L}_T^{1}(\dot{B}_{p,1}^{\frac{d}{p}})}.
\end{array}
\end{equation}
One deduces from (\ref{3.9}) and (\ref{3.10}) that
\begin{equation}\label{3.11}
\begin{array}{ll}
\|\tilde{P}\|^l_{\tilde{L}_T^{\infty}(\dot{B}_{p,1}^{\frac{d}{p}-1})}
+\|\tilde{P}\|^l_{\tilde{L}_T^{1}(\dot{B}_{p,1}^{\frac{d}{p}+1})}
+\|\tilde{\mathbf{u}}\|^l_{\tilde{L}_T^{\infty}(\dot{B}_{p,1}^{\frac{d}{p}})}
+\|\tilde{\mathbf{u}}\|^l_{\tilde{L}_T^{1}(\dot{B}_{p,1}^{\frac{d}{p}})}\\[6mm]
+\|\mathbf{Z}\|^l_{\tilde{L}_T^{\infty}(\dot{B}_{p,1}^{\frac{d}{p}})}
+\|\mathbf{Z}\|^l_{\tilde{L}_T^{1}(\dot{B}_{p,1}^{\frac{d}{p}})}
\lesssim \mathcal{X}_{p,0}+\mathcal{X}_p^2(t).
\end{array}
\end{equation}

Then, we give the low-frequency estimate of $\tilde{c}$.
Applying $\dot{\Delta}_j(\ref{1.5})_3$ to $\dot{\Delta}_j\tilde{c}$ and integrate over $\mathbb{R}^d$. It holds that
\begin{equation}\label{3.12}
\begin{array}{ll}
\frac{1}{2}\frac{d}{dt}\|\dot{\Delta}_j\tilde{c}\|_{L^2}^2=-\kappa_1\langle \dot{\Delta}_j\tilde{c}, [\dot{\Delta}_j,\tilde{\mathbf{u}}]\nabla\tilde{c} \rangle
-\kappa_1\langle\dot{\Delta}_j\tilde{c},\tilde{\mathbf{u}}\nabla\dot{\Delta}_j\tilde{c} \rangle.
\end{array}
\end{equation}
Using function of integration by parts and H\"{o}lder inequality, we have
\begin{equation}\label{3.13}
\begin{array}{ll}
\frac{1}{2}\frac{d}{dt}\|\dot{\Delta}_j\tilde{c}\|_{L^2}^2
\!\!\!\!&=-\kappa_1\langle \dot{\Delta}_j\tilde{c}, [\dot{\Delta}_j,\tilde{\mathbf{u}}]\nabla\tilde{c} \rangle
+\frac{\kappa_1}{2}\langle\dot{\Delta}_j\tilde{c},{\rm div}\tilde{\mathbf{u}}\dot{\Delta}_j\tilde{c} \rangle\\[3mm]
&\lesssim\|\dot{\Delta}_j\tilde{c}\|_{L^2} \|[\dot{\Delta}_j,\tilde{\mathbf{u}}]\nabla\tilde{c} \|_{L^2}+
\|\dot{\Delta}_j\tilde{c}\|_{L^2}^2\ \|{\rm div}\tilde{\mathbf{u}}\|_{L^\infty}.
\end{array}
\end{equation}
Multiplying (\ref{3.13}) by $\frac{1}{\|\dot{\Delta}_j\tilde{c}\|_{L^2}}$ and integrate about $t$, one has
\begin{equation}\label{3.14}
\begin{array}{ll}
\displaystyle\|\dot{\Delta}_j\tilde{c}\|_{\tilde{L}_T^\infty({L^2})}
\lesssim\|\dot{\Delta}_j\tilde{c}_0\|_{L^2}+\int_0^T\|[\dot{\Delta}_j,\tilde{\mathbf{u}}]\nabla\tilde{c} \|_{L^2}dt+\int_0^T\|\dot{\Delta}_j\tilde{c}\|_{L^2}\ \|{\rm div}\tilde{\mathbf{u}}\|_{L^\infty}dt.
\end{array}
\end{equation}
By virtue of Lemma \ref{2 2.6} and $\dot{B}_{2,1}^{\frac{d}{2}}\hookrightarrow\dot{B}_{p,1}^{\frac{d}{p}}\ (p\geq 2)$, we easily get
\begin{equation}\label{3.15}
\begin{array}{ll}
\|\tilde{c}\|^l_{\tilde{L}_T^\infty(\dot{B}_{2,1}^{\frac{d}{2}-1})}
&\!\!\displaystyle\lesssim\|\tilde{c}_0\|^l_{\dot{B}_{2,1}^{\frac{d}{2}-1}}+\int_0^T\|\nabla\tilde{\mathbf{u}}\|_{\dot{B}_{p,1}^{\frac{d}{p}}}\|\tilde{c} \|_{\dot{B}_{2,1}^{\frac{d}{2}-1}}dt
+\int_0^T\|\tilde{c}\|_{\dot{B}_{2,1}^{\frac{d}{2}-1}}\ \|{\rm div}\tilde{\mathbf{u}}\|_{L^\infty}dt\\[5mm]
&\!\!\lesssim\|\tilde{c}_0\|^l_{\dot{B}_{2,1}^{\frac{d}{2}-1}}
+(\|\tilde{c}\|^l_{\tilde{L}_T^\infty(\dot{B}_{2,1}^{\frac{d}{2}-1})}+\|\tilde{c}\|^h_{\tilde{L}_T^\infty(\dot{B}_{2,1}^{\frac{d}{2}+1})})
(\|\tilde{\mathbf{u}}\|^l_{\tilde{L}_T^{1}\dot{B}_{p,1}^{\frac{d}{p}}}+\|\tilde{\mathbf{u}}\|^h_{\tilde{L}_T^{1}\dot{B}_{2,1}^{\frac{d}{2}+1}}).
\end{array}
\end{equation}
Combining (\ref{3.11}) and (\ref{3.15}), we have
\begin{equation}\label{3.16}
\begin{array}{ll}
\|\tilde{P}\|^l_{\tilde{L}_T^{\infty}(\dot{B}_{p,1}^{\frac{d}{p}-1})}
+\|\tilde{P}\|^l_{\tilde{L}_T^{1}(\dot{B}_{p,1}^{\frac{d}{p}+1})}
+\|\tilde{\mathbf{u}}\|^l_{\tilde{L}_T^{\infty}(\dot{B}_{p,1}^{\frac{d}{p}})}
+\|\tilde{\mathbf{u}}\|^l_{\tilde{L}_T^{1}(\dot{B}_{p,1}^{\frac{d}{p}})}\\[6mm]
+\|\tilde{c}\|^l_{\tilde{L}_T^\infty(\dot{B}_{2,1}^{\frac{d}{2}-1})}
\lesssim \mathcal{X}_{p,0}+\mathcal{X}_p^2(t).
\end{array}
\end{equation}

\subsection{The high-frequency estimate}
\quad\quad Applying $\dot\Delta_j$ to $(\ref{1.5})_1-(\ref{1.5})_3$, we have the following equations
\begin{equation}\label{3.17}
\left\{\begin{array}{ll}
(\dot\Delta_j\tilde{P})_t+\kappa_2{\rm div}\dot\Delta_j\tilde{\mathbf{u}}+\kappa_1\dot{S}_{j-1}h(\tilde{P},\tilde{c}){\rm div}\dot\Delta_j\tilde{\mathbf{u}}+\kappa_1\dot{S}_{j-1}\tilde{\mathbf{u}}\nabla\dot\Delta_j\tilde{P}\\[2mm]
\quad\quad\quad\quad=\kappa_1\dot{S}_{j-1}h(\tilde{P},\tilde{c}){\rm div}\dot\Delta_j\tilde{\mathbf{u}}-\kappa_1\dot\Delta_j(h(\tilde{P},\tilde{c}){\rm div}\tilde{\mathbf{u}})
+\kappa_1\dot{S}_{j-1}\tilde{\mathbf{u}}\nabla\dot\Delta_j\tilde{P}\\[2mm]
\quad\quad\quad\quad\quad-\kappa_1\dot\Delta_j(\tilde{\mathbf{u}}\cdot\nabla\tilde{P}):=R_j^1+R_j^2,\\[2mm]
(\dot\Delta_j\tilde{\mathbf{u}})_t+\kappa_2\nabla\dot\Delta_j\tilde{P}+\alpha\dot\Delta_j\tilde{\mathbf{u}}
+\kappa_1\dot{S}_{j-1}\tilde{\mathbf{u}}\cdot\nabla\dot\Delta_j\tilde{\mathbf{u}}+\frac{1}{\kappa_1}\dot{S}_{j-1}(\frac{1}{m}-\frac{1}{m_\infty})\nabla\dot\Delta_j\tilde{P}\\[2mm]
\quad\quad\quad\quad=\kappa_1\dot{S}_{j-1}\tilde{\mathbf{u}}\cdot\nabla\dot\Delta_j\tilde{\mathbf{u}}-\kappa_1\dot\Delta_j(\tilde{\mathbf{u}}\cdot\nabla\tilde{\mathbf{u}})
+\frac{1}{\kappa_1}\dot{S}_{j-1}(\frac{1}{m}-\frac{1}{m_\infty})\nabla\dot\Delta_j\tilde{P}\\[2mm]
\quad\quad\quad\quad\quad-\frac{1}{\kappa_1}\dot\Delta_j((\frac{1}{m}-\frac{1}{m_\infty})\nabla\tilde{P}):=R_j^3+R_j^4,\\[2mm]
(\dot\Delta_j\tilde{c})_t+\kappa_1\dot{S}_{j-1}\tilde{\mathbf{u}}\nabla\dot\Delta_j\tilde{c}=\kappa_1\dot{S}_{j-1}\tilde{\mathbf{u}}\nabla\dot\Delta_j\tilde{c}-\kappa_1\dot\Delta_j(\tilde{\mathbf{u}}\cdot\nabla\tilde{c}).
\end{array}\right.
\end{equation}

Multiplying $(\ref{3.17})_1-(\ref{3.17})_3$ by $(\kappa_2+\frac{1}{\kappa_1}\dot{S}_{j-1}(\frac{1}{m}-\frac{1}{m_\infty}))\dot\Delta_j\tilde{P},\ (\kappa_2+\kappa_1\dot{S}_{j-1}h(\tilde{P},\tilde{c}))\dot\Delta_j\tilde{\mathbf{u}},\ \dot\Delta_j\tilde{c}$, respectively, and integrating over $\mathbb{R}^d$, we obtain
\begin{align}
&\frac{1}{2}\frac{d}{dt}\int_{\mathbb{R}^d}(\kappa_2\!+\!\frac{1}{\kappa_1}\dot{S}_{j-1}(\frac{1}{m}\!-\!\frac{1}{m_\infty}))(\dot\Delta_j\tilde{P})^2dx
-\frac{1}{2}\int_{\mathbb{R}^d}\frac{d}{dt}(\kappa_2\!+\!\frac{1}{\kappa_1}\dot{S}_{j-1}(\frac{1}{m}\!-\!\frac{1}{m_\infty}))(\dot\Delta_j\tilde{P})^2dx\notag\\[3mm]
&+\langle(\kappa_2+\frac{1}{\kappa_1}\dot{S}_{j-1}(\frac{1}{m}-\frac{1}{m_\infty}))\dot\Delta_j\tilde{P},(\kappa_2+\kappa_1\dot{S}_{j-1}h(\tilde{P},\tilde{c})){\rm div}\dot\Delta_j\tilde{\mathbf{u}}\rangle\notag\\[3mm]
&+\langle(\kappa_2+\frac{1}{\kappa_1}\dot{S}_{j-1}(\frac{1}{m}-\frac{1}{m_\infty}))\dot\Delta_j\tilde{P},\kappa_1\dot{S}_{j-1}\tilde{\mathbf{u}}\nabla\dot\Delta_j\tilde{P}\rangle\notag\\[3mm]
=&\langle(\kappa_2+\frac{1}{\kappa_1}\dot{S}_{j-1}(\frac{1}{m}-\frac{1}{m_\infty}))\dot\Delta_j\tilde{P},R_j^1+R_j^2\rangle,\label{3.18}
\end{align}
\begin{align}
&\displaystyle\frac{1}{2}\frac{d}{dt}\int_{\mathbb{R}^d}(\kappa_2+\kappa_1\dot{S}_{j-1}h(\tilde{P},\tilde{c}))(\dot\Delta_j\tilde{\mathbf{u}})^2dx
-\frac{1}{2}\int_{\mathbb{R}^d}\frac{d}{dt}(\kappa_2+\kappa_1\dot{S}_{j-1}h(\tilde{P},\tilde{c}))(\dot\Delta_j\tilde{\mathbf{u}})^2dx\notag\\[3mm]
&\displaystyle+\langle(\kappa_2+\kappa_1\dot{S}_{j-1}h(\tilde{P},\tilde{c}))\dot\Delta_j\tilde{\mathbf{u}},
(\kappa_2+\frac{1}{\kappa_1}\dot{S}_{j-1}(\frac{1}{m}-\frac{1}{m_\infty}))\nabla\dot\Delta_j\tilde{P}\rangle\notag\\[3mm]
&\displaystyle+\langle(\kappa_2+\kappa_1\dot{S}_{j-1}h(\tilde{P},\tilde{c}))\dot\Delta_j\tilde{\mathbf{u}},\alpha\dot\Delta_j\tilde{\mathbf{u}}\rangle
+\langle(\kappa_2+\kappa_1\dot{S}_{j-1}h(\tilde{P},\tilde{c}))\dot\Delta_j\tilde{\mathbf{u}},\kappa_1\dot{S}_{j-1}\tilde{\mathbf{u}}\cdot\nabla\dot\Delta_j\tilde{\mathbf{u}}\rangle\notag\\[3mm]
\displaystyle=&\langle (\kappa_2+\kappa_1\dot{S}_{j-1}h(\tilde{P},\tilde{c}))\dot\Delta_j\tilde{\mathbf{u}},R_j^3+R_j^4\rangle,\label{3.19}
\end{align}
and
\begin{equation}\label{3.20}
\begin{array}{ll}
\displaystyle\frac{1}{2}\frac{d}{dt}\int_{\mathbb{R}^d}(\dot\Delta_j\tilde{c})^2dx
+\langle\kappa_1\dot{S}_{j-1}\tilde{\mathbf{u}}\dot\nabla\Delta_j\tilde{c},\dot\Delta_j\tilde{c}\rangle
=\langle \dot\Delta_j\tilde{c},R_j^5\rangle.
\end{array}
\end{equation}
Summing up (\ref{3.18}), (\ref{3.19}) and(\ref{3.20}) and using integration by parts, one has
\begin{align}
&\frac{1}{2}\frac{d}{dt}\int_{\mathbb{R}^d}(\kappa_2\!+\!\frac{1}{\kappa_1}\dot{S}_{j-1}(\frac{1}{m}\!-\!\frac{1}{m_\infty}))(\dot\Delta_j\tilde{P})^2
                                            +(\kappa_2+\kappa_1\dot{S}_{j-1}h(\tilde{P},\tilde{c}))(\dot\Delta_j\tilde{\mathbf{u}})^2+(\dot\Delta_j\tilde{c})^2dx\notag\\[3mm]
&+\alpha\langle(\kappa_2+\kappa_1\dot{S}_{j-1}h(\tilde{P},\tilde{c}))\dot\Delta_j\tilde{\mathbf{u}},\dot\Delta_j\tilde{\mathbf{u}}\rangle\notag\\[3mm]
=&\frac{1}{2}\int_{\mathbb{R}^d}\frac{d}{dt}(\kappa_2\!+\!\frac{1}{\kappa_1}\dot{S}_{j-1}(\frac{1}{m}\!-\!\frac{1}{m_\infty}))(\dot\Delta_j\tilde{P})^2dx
+\frac{1}{2}\int_{\mathbb{R}^d}\frac{d}{dt}(\kappa_2+\kappa_1\dot{S}_{j-1}h(\tilde{P},\tilde{c}))(\dot\Delta_j\tilde{\mathbf{u}})^2dx\notag\\[3mm]
&+\langle\nabla\big((\kappa_2+\kappa_1\dot{S}_{j-1}h(\tilde{P},\tilde{c}))(\kappa_2+\frac{1}{\kappa_1}\dot{S}_{j-1}(\frac{1}{m}-\frac{1}{m_\infty}))\big),\dot\Delta_j\tilde{\mathbf{u}}\dot\Delta_j\tilde{P}\rangle\notag\\[3mm]
&+\frac{1}{2}\langle{\rm div}\big(\kappa_1(\kappa_2+\frac{1}{\kappa_1}\dot{S}_{j-1}(\frac{1}{m}-\frac{1}{m_\infty}))\dot{S}_{j-1}\tilde{\mathbf{u}}\big),(\Delta_j\tilde{P})^2\rangle\notag\\[3mm]
&+\frac{1}{2}\langle{\rm div}\big(\kappa_1(\kappa_2+\kappa_1\dot{S}_{j-1}h(\tilde{P},\tilde{c}))\dot{S}_{j-1}\tilde{\mathbf{u}}\big),(\dot\Delta_j\tilde{\mathbf{u}})^2\rangle+\frac{\kappa_1}{2}\langle{\rm div}\dot{S}_{j-1}\tilde{\mathbf{u}},(\dot\Delta_j\tilde{c})^2\rangle+\langle \dot\Delta_j\tilde{c},R_j^5\rangle
\notag\\[3mm]
&+\langle(\kappa_2+\frac{1}{\kappa_1}\dot{S}_{j-1}(\frac{1}{m}-\frac{1}{m_\infty}))\dot\Delta_j\tilde{P},R_j^1+R_j^2\rangle+\langle (\kappa_2+\kappa_1\dot{S}_{j-1}h(\tilde{P},\tilde{c}))\dot\Delta_j\tilde{\mathbf{u}},R_j^3+R_j^4\rangle.\label{3.21}
\end{align}
By using H\"{o}lder inequality, we get
\begin{align}
&\frac{1}{2}\frac{d}{dt}\int_{\mathbb{R}^d}(\kappa_2\!+\!\frac{1}{\kappa_1}\dot{S}_{j-1}(\frac{1}{m}\!-\!\frac{1}{m_\infty}))(\dot\Delta_j\tilde{P})^2
                                            +(\kappa_2+\kappa_1\dot{S}_{j-1}h(\tilde{P},\tilde{c}))(\dot\Delta_j\tilde{\mathbf{u}})^2+(\dot\Delta_j\tilde{c})^2dx\notag\\[3mm]
&+\alpha\langle(\kappa_2+\kappa_1\dot{S}_{j-1}h(\tilde{P},\tilde{c}))\dot\Delta_j\tilde{\mathbf{u}},\dot\Delta_j\tilde{\mathbf{u}}\rangle\notag\\[3mm]
\lesssim&(\big\|\partial_t(\kappa_2\!+\!\frac{1}{\kappa_1}\dot{S}_{j-1}(\frac{1}{m}\!-\!\frac{1}{m_\infty}))\|_{L^\infty}
+\|\partial_t(\kappa_2+\kappa_1\dot{S}_{j-1}h(\tilde{P},\tilde{c}))\|_{L^\infty}\notag\\[3mm]
&+\|{\rm div}\big(\kappa_1(\kappa_2\!+\!\frac{1}{\kappa_1}\dot{S}_{j-1}(\frac{1}{m}\!-\!\frac{1}{m_\infty}))\dot{S}_{j-1}\tilde{\mathbf{u}}\big)\|_{L^\infty}
+\|{\rm div}\big(\kappa_1(\kappa_2+\kappa_1\dot{S}_{j-1}h(\tilde{P},\tilde{c}))\dot{S}_{j-1}\tilde{\mathbf{u}}\big)\|_{L^\infty}\notag\\[3mm]
&+\|\nabla\big((\kappa_2+\kappa_1\dot{S}_{j-1}h(\tilde{P},\tilde{c}))(\kappa_2+\frac{1}{\kappa_1}\dot{S}_{j-1}(\frac{1}{m}-\frac{1}{m_\infty}))\big)\|_{L^\infty}\big)\|(\dot\Delta_j\tilde{\mathbf{u}},\dot\Delta_j\tilde{P})\|_{L^2}^2\notag\\[3mm]
&+\|{\rm div}\dot{S}_{j-1}\tilde{\mathbf{u}}\|_{L^\infty}\|\dot\Delta_j\tilde{c}\|_{L^2}^2+\| \dot\Delta_j\tilde{c}\|_{L^2}\|R_j^5\|_{L^2}\notag\\[3mm]
&+\|(\kappa_2+\frac{1}{\kappa_1}\dot{S}_{j-1}(\frac{1}{m}-\frac{1}{m_\infty}))\|_{L^\infty}\|\dot\Delta_j\tilde{P}\|_{L^2}\|R_j^1+R_j^2\|_{L^2}\notag\\[3mm]
&+\| (\kappa_2+\kappa_1\dot{S}_{j-1}h(\tilde{P},\tilde{c}))\|_{L^\infty}\|\dot\Delta_j\tilde{\mathbf{u}}\|_{L^2}\|R_j^3+R_j^4\|_{L^2}.\label{3.22}
\end{align}
From $(\ref{3.17})_1$ and $(\ref{3.17})_2$, we obtain
\begin{align}
\displaystyle\frac{d}{dt}\int_{\mathbb{R}^d}\nabla\dot\Delta_j\tilde{P}&\dot\Delta_j\tilde{\mathbf{u}}dx=-\kappa_2\langle\nabla{\rm div}\dot\Delta_j\tilde{\mathbf{u}},\Delta_j\tilde{\mathbf{u}}\rangle-\kappa_1\langle\nabla(\dot{S}_{j-1}h(\tilde{P},\tilde{c}){\rm div}\dot\Delta_j\tilde{\mathbf{u}}),\dot\Delta_j\tilde{\mathbf{u}}\rangle\notag\\[3mm]
&\displaystyle-\kappa_1\langle\nabla(\dot{S}_{j-1}\tilde{\mathbf{u}}\cdot\nabla\dot\Delta_j\tilde{P}),\dot\Delta_j\tilde{\mathbf{u}}\rangle
+\langle\nabla R_j^1,\dot\Delta_j\tilde{\mathbf{u}}\rangle+\langle\nabla R_j^2,\dot\Delta_j\tilde{\mathbf{u}}\rangle\notag\\[3mm]
&-\kappa_2\langle\nabla\dot\Delta_j\tilde{P},\nabla\dot\Delta_j\tilde{P}\rangle-\alpha\langle\nabla\dot\Delta_j\tilde{P},\dot\Delta_j\tilde{\mathbf{u}}\rangle
-\kappa_1\langle\nabla\dot\Delta_j\tilde{P},\dot{S}_{j-1}\tilde{\mathbf{u}}\cdot\nabla\dot\Delta_j\tilde{\mathbf{u}}\rangle\notag\\[3mm]
&-\frac{1}{\kappa_1}\langle\nabla\dot\Delta_j\tilde{P},\dot{S}_{j-1}(\frac{1}{m}-\frac{1}{m_\infty})\nabla\dot\Delta_j\tilde{P}\rangle
+\langle\nabla\dot\Delta_j\tilde{P},R_j^3\rangle+\langle\nabla\dot\Delta_j\tilde{P},R_j^4\rangle,\label{3.23}
\end{align}
thus,
\begin{align}
&\displaystyle\frac{d}{dt}\int_{\mathbb{R}^d}\nabla\dot\Delta_j\tilde{P}\dot\Delta_j\tilde{\mathbf{u}}dx
+\langle(\kappa_2+\frac{1}{\kappa_1}\dot{S}_{j-1}(\frac{1}{m}-\frac{1}{m_\infty})),(\nabla\dot\Delta_j\tilde{P})^2\rangle\notag\\[3mm]
&-\langle(\kappa_2+\kappa_1\dot{S}_{j-1}h(\tilde{P},\tilde{c})){\rm div}\dot\Delta_j\tilde{\mathbf{u}},{\rm div}\dot\Delta_j\tilde{\mathbf{u}}\rangle
+\alpha\langle\nabla\dot\Delta_j\tilde{P},\dot\Delta_j\tilde{\mathbf{u}}\rangle\notag\\[3mm]
&-\kappa_1\langle\dot{S}_{j-1}\tilde{\mathbf{u}}\cdot\nabla\dot\Delta_j\tilde{P},{\rm div}\dot\Delta_j\tilde{\mathbf{u}}\rangle
+\kappa_1\langle\nabla\dot\Delta_j\tilde{P},\dot{S}_{j-1}\tilde{\mathbf{u}}\cdot\nabla\dot\Delta_j\tilde{\mathbf{u}}\rangle\notag\\[3mm]
=&
\langle\nabla R_j^1,\dot\Delta_j\tilde{\mathbf{u}}\rangle+\langle\nabla R_j^2,\dot\Delta_j\tilde{\mathbf{u}}\rangle
+\langle\nabla\dot\Delta_j\tilde{P},R_j^3\rangle+\langle\nabla\dot\Delta_j\tilde{P},R_j^4\rangle\notag\\[3mm]
\lesssim&\|R_j^1\|_{L^2}\|{\rm div}\dot\Delta_j\tilde{\mathbf{u}}\|_{L^2}
+\|R_j^2\|_{L^2}\|{\rm div}\dot\Delta_j\tilde{\mathbf{u}}\|_{L^2}
+\|\nabla\dot\Delta_j\tilde{P}\|_{L^2}\|R_j^3\|_{L^2}+\|\nabla\dot\Delta_j\tilde{P}\|_{L^2}\|R_j^4\|_{L^2}.\label{3.24}
\end{align}
Adding $(\ref{3.21})$ and $\beta_12^{-2j}\cdot(\ref{3.24})$, we have
\begin{align}
&\displaystyle\frac{1}{2}\frac{d}{dt}\int_{\mathbb{R}^d}\big((\kappa_2\!+\!\frac{1}{\kappa_1}\dot{S}_{j-1}(\frac{1}{m}\!-\!\frac{1}{m_\infty}))(\dot\Delta_j\tilde{P})^2
                                            +(\kappa_2+\kappa_1\dot{S}_{j-1}h(\tilde{P},\tilde{c}))(\dot\Delta_j\tilde{\mathbf{u}})^2+(\dot\Delta_j\tilde{c})^2\notag\\[3mm]
                                            &\displaystyle\quad+2\cdot\beta_12^{-2j}\nabla\dot\Delta_j\tilde{P}\dot\Delta_j\tilde{\mathbf{u}}\big)dx
+\alpha\langle(\kappa_2+\kappa_1\dot{S}_{j-1}h(\tilde{P},\tilde{c}))\dot\Delta_j\tilde{\mathbf{u}},\dot\Delta_j\tilde{\mathbf{u}}\rangle\notag\\[3mm]
&\quad\displaystyle+\beta_12^{-2j}\langle(\kappa_2+\frac{1}{\kappa_1}\dot{S}_{j-1}(\frac{1}{m}-\frac{1}{m_\infty})),(\nabla\dot\Delta_j\tilde{P})^2\rangle\notag\\[3mm]
&\quad-\beta_12^{-2j}\langle(\kappa_2+\kappa_1\dot{S}_{j-1}h(\tilde{P},\tilde{c})){\rm div}\dot\Delta_j\tilde{\mathbf{u}},{\rm div}\dot\Delta_j\tilde{\mathbf{u}}\rangle
+\beta_12^{-2j}\alpha\langle\nabla\dot\Delta_j\tilde{P},\dot\Delta_j\tilde{\mathbf{u}}\rangle\notag\\[3mm]
&\quad-\beta_12^{-2j}\kappa_1\langle\dot{S}_{j-1}\tilde{\mathbf{u}}\cdot\nabla\dot\Delta_j\tilde{P},{\rm div}\dot\Delta_j\tilde{\mathbf{u}}\rangle
+\beta_12^{-2j}\kappa_1\langle\nabla\dot\Delta_j\tilde{P},\dot{S}_{j-1}\tilde{\mathbf{u}}\cdot\nabla\dot\Delta_j\tilde{\mathbf{u}}\rangle\notag\\[3mm]
\displaystyle\lesssim&(\big\|\partial_t(\kappa_2\!+\!\frac{1}{\kappa_1}\dot{S}_{j-1}(\frac{1}{m}\!-\!\frac{1}{m_\infty}))\|_{L^\infty}
+\|\partial_t(\kappa_2+\kappa_1\dot{S}_{j-1}h(\tilde{P},\tilde{c}))\|_{L^\infty}\notag\\[3mm]
&\displaystyle+\|{\rm div}\big(\kappa_1(\kappa_2\!+\!\frac{1}{\kappa_1}\dot{S}_{j-1}(\frac{1}{m}\!-\!\frac{1}{m_\infty}))\dot{S}_{j-1}\tilde{\mathbf{u}}\big)\|_{L^\infty}
+\|{\rm div}\big(\kappa_1(\kappa_2\!+\!\kappa_1\dot{S}_{j-1}h(\tilde{P},\tilde{c}))\dot{S}_{j-1}\tilde{\mathbf{u}}\big)\|_{L^\infty}\notag\\[3mm]
&\displaystyle+\|\nabla\big((\kappa_2+\kappa_1\dot{S}_{j-1}h(\tilde{P},\tilde{c}))(\kappa_2+\frac{1}{\kappa_1}\dot{S}_{j-1}(\frac{1}{m}-\frac{1}{m_\infty}))\big)\|_{L^\infty}\big)\|(\dot\Delta_j\tilde{\mathbf{u}},\dot\Delta_j\tilde{P})\|_{L^2}^2\notag\\[3mm]
&\displaystyle+\|{\rm div}\dot{S}_{j-1}\tilde{\mathbf{u}}\|_{L^\infty}\|\dot\Delta_j\tilde{c}\|_{L^2}^2+\| \dot\Delta_j\tilde{c}\|_{L^2}\|R_j^5\|_{L^2}\notag\\[3mm]
&\displaystyle+\|(\kappa_2+\frac{1}{\kappa_1}\dot{S}_{j-1}(\frac{1}{m}-\frac{1}{m_\infty}))\|_{L^\infty}\|\dot\Delta_j\tilde{P}\|_{L^2}\|R_j^1+R_j^2\|_{L^2}\notag\\[3mm]
&+\| (\kappa_2+\kappa_1\dot{S}_{j-1}h(\tilde{P},\tilde{c}))\|_{L^\infty}\|\dot\Delta_j\tilde{\mathbf{u}}\|_{L^2}\|R_j^3+R_j^4\|_{L^2}
+\beta_12^{-2j}\|R_j^1\|_{L^2}\|{\rm div}\dot\Delta_j\tilde{\mathbf{u}}\|_{L^2}\notag\\[3mm]
&+\beta_12^{-2j}\|R_j^2\|_{L^2}\|{\rm div}\dot\Delta_j\tilde{\mathbf{u}}\|_{L^2}
+\beta_12^{-2j}\|\nabla\dot\Delta_j\tilde{P}\|_{L^2}\|R_j^3\|_{L^2}+\beta_12^{-2j}\|\nabla\dot\Delta_j\tilde{P}\|_{L^2}\|R_j^4\|_{L^2}.\label{3.25}
\end{align}
By virtue of Bernstein inequality, Young inequality and taking $\beta_1\ll1$, we deduce
\begin{align}
\displaystyle\quad\int_{\mathbb{R}^d}\Big((\kappa_2\!&+\!\frac{1}{\kappa_1}\dot{S}_{j-1}(\frac{1}{m}\!-\!\frac{1}{m_\infty}))(\dot\Delta_j\tilde{P})^2
                                            +(\kappa_2+\kappa_1\dot{S}_{j-1}h(\tilde{P},\tilde{c}))(\dot\Delta_j\tilde{\mathbf{u}})^2+(\dot\Delta_j\tilde{c})^2\notag\\[3mm]
                                            \displaystyle&+2\cdot\beta_12^{-2j}\nabla\dot\Delta_j\tilde{P}\dot\Delta_j\tilde{\mathbf{u}}\Big)\sim\|\dot\Delta_j(\tilde{P},\tilde{\mathbf{u}},\tilde{c})\|^2_{L^2},\label{3.26}
\end{align}
and
\begin{align}
&\displaystyle\alpha\langle(\kappa_2+\kappa_1\dot{S}_{j-1}h(\tilde{P},\tilde{c}))\dot\Delta_j\tilde{\mathbf{u}},\dot\Delta_j\tilde{\mathbf{u}}\rangle
+\beta_12^{-2j}\langle(\kappa_2+\frac{1}{\kappa_1}\dot{S}_{j-1}(\frac{1}{m}-\frac{1}{m_\infty})),(\nabla\dot\Delta_j\tilde{P})^2\rangle\notag\\[3mm]
&\displaystyle-\beta_12^{-2j}\langle(\kappa_2+\kappa_1\dot{S}_{j-1}h(\tilde{P},\tilde{c})){\rm div}\dot\Delta_j\tilde{\mathbf{u}},{\rm div}\dot\Delta_j\tilde{\mathbf{u}}\rangle
+\beta_12^{-2j}\alpha\langle\nabla\dot\Delta_j\tilde{P},\dot\Delta_j\tilde{\mathbf{u}}\rangle\notag\\[3mm]
&\displaystyle-\beta_12^{-2j}\kappa_1\langle\dot{S}_{j-1}\tilde{\mathbf{u}}\cdot\nabla\dot\Delta_j\tilde{P},{\rm div}\dot\Delta_j\tilde{\mathbf{u}}\rangle
+\beta_12^{-2j}\kappa_1\langle\nabla\dot\Delta_j\tilde{P},\dot{S}_{j-1}\tilde{\mathbf{u}}\cdot\nabla\dot\Delta_j\tilde{\mathbf{u}}\rangle\notag\\[3mm]
\sim&\displaystyle\|\dot\Delta_j(\tilde{P},\tilde{\mathbf{u}})\|^2_{L^2}.\label{3.27}
\end{align}
Then, by using $\|\dot{S}_{j-1}(\frac{1}{m}-\frac{1}{m_\infty}))\|_{L^\infty}\lesssim\|(\tilde{P},\tilde{c})\|_{L^\infty}\ll1$ and $\|\dot{S}_{j-1}h(\tilde{P},\tilde{c})\|_{L^\infty}\lesssim\|(\tilde{P},\tilde{c})\|_{L^\infty}\ll1$, (\ref{3.25}) can be rewritten as
\begin{align}
&\displaystyle\quad\frac{d}{dt}\|\dot\Delta_j(\tilde{P},\tilde{\mathbf{u}},\tilde{c})\|_{L^2}^2+\|\dot\Delta_j(\tilde{P},\tilde{\mathbf{u}})\|_{L^2}^2\notag\\[3mm]
\displaystyle\lesssim&\Big(\|\partial_t(\kappa_2\!+\!\frac{1}{\kappa_1}\dot{S}_{j-1}(\frac{1}{m}\!-\!\frac{1}{m_\infty}))\|_{L^\infty}
+\|\partial_t(\kappa_2+\kappa_1\dot{S}_{j-1}h(\tilde{P},\tilde{c}))\|_{L^\infty}\notag\\[3mm]
&\displaystyle+\|{\rm div}\big(\kappa_1(\kappa_2\!+\!\frac{1}{\kappa_1}\dot{S}_{j-1}(\frac{1}{m}\!-\!\frac{1}{m_\infty}))\dot{S}_{j-1}\tilde{\mathbf{u}}\big)\|_{L^\infty}
+\|{\rm div}\big(\kappa_1(\kappa_2\!+\!\kappa_1\dot{S}_{j-1}h(\tilde{P},\tilde{c}))\dot{S}_{j-1}\tilde{\mathbf{u}}\big)\|_{L^\infty}\notag\\[3mm]
&\displaystyle+\|\nabla\big((\kappa_2+\kappa_1\dot{S}_{j-1}h(\tilde{P},\tilde{c}))(\kappa_2+\frac{1}{\kappa_1}\dot{S}_{j-1}(\frac{1}{m}-\frac{1}{m_\infty}))\big)\|_{L^\infty}\Big)\|(\dot\Delta_j\tilde{\mathbf{u}},\dot\Delta_j\tilde{P})\|_{L^2}^2\notag\\[3mm]
&\displaystyle+\|{\rm div}\dot{S}_{j-1}\tilde{\mathbf{u}}\|_{L^\infty}\|\dot\Delta_j\tilde{c}\|_{L^2}^2+\| \dot\Delta_j\tilde{c}\|_{L^2}\|R_j^5\|_{L^2}\notag\\[3mm]
&\displaystyle+\|\dot\Delta_j\tilde{P}\|_{L^2}\|R_j^1+R_j^2\|_{L^2}
+\|\dot\Delta_j\tilde{\mathbf{u}}\|_{L^2}\|R_j^3+R_j^4\|_{L^2}
+\beta_12^{-j}\|R_j^1\|_{L^2}\|\dot\Delta_j\tilde{\mathbf{u}}\|_{L^2}\notag\\[3mm]
&+\beta_12^{-j}\|R_j^2\|_{L^2}\|\dot\Delta_j\tilde{\mathbf{u}}\|_{L^2}
+\beta_12^{-j}\|\dot\Delta_j\tilde{P}\|_{L^2}\|R_j^3\|_{L^2}+\beta_12^{-j}\|\dot\Delta_j\tilde{P}\|_{L^2}\|R_j^4\|_{L^2}.\label{3.28}
\end{align}
Multiplying (\ref{3.28}) by $\frac{1}{\|\dot\Delta_j(\tilde{P},\tilde{\mathbf{u}},\tilde{c})\|_{L^2}}$ and integrating about $t$, we get
\begin{align}
&\displaystyle\quad\|\dot\Delta_j(\tilde{P},\tilde{\mathbf{u}},\tilde{c})\|_{L^2}+\|\dot\Delta_j(\tilde{P},\tilde{\mathbf{u}})\|_{L_t^1(L^2)}\notag\\[3mm]
\displaystyle\lesssim&\|\dot\Delta_j(\tilde{P}_0,\tilde{\mathbf{u}}_0,\tilde{c}_0)\|_{L^2}+\int_0^T\Big(\|\partial_t(\kappa_2\!+\!\frac{1}{\kappa_1}\dot{S}_{j-1}(\frac{1}{m}\!-\!\frac{1}{m_\infty}))\|_{L^\infty}
+\|\partial_t(\kappa_2+\kappa_1\dot{S}_{j-1}h(\tilde{P},\tilde{c}))\|_{L^\infty}\notag\\[3mm]
&\displaystyle+\|{\rm div}\big(\kappa_1(\kappa_2\!+\!\frac{1}{\kappa_1}\dot{S}_{j-1}(\frac{1}{m}\!-\!\frac{1}{m_\infty}))\dot{S}_{j-1}\tilde{\mathbf{u}}\big)\|_{L^\infty}
+\|{\rm div}\big(\kappa_1(\kappa_2\!+\!\kappa_1\dot{S}_{j-1}h(\tilde{P},\tilde{c}))\dot{S}_{j-1}\tilde{\mathbf{u}}\big)\|_{L^\infty}\notag\\[3mm]
&\displaystyle+\|\nabla\big((\kappa_2+\kappa_1\dot{S}_{j-1}h(\tilde{P},\tilde{c}))(\kappa_2+\frac{1}{\kappa_1}\dot{S}_{j-1}(\frac{1}{m}-\frac{1}{m_\infty}))\big)\|_{L^\infty}\Big)\|(\dot\Delta_j\tilde{\mathbf{u}},\dot\Delta_j\tilde{P})\|_{L^2}\ dt\notag\\[3mm]
&\displaystyle+\int_0^T\|{\rm div}\dot{S}_{j-1}\tilde{\mathbf{u}}\|_{L^\infty}\|\dot\Delta_j\tilde{c}\|_{L^2}dt+
\int_0^T\big(\|R_j^1\|_{L^2}+\|R_j^2\|_{L^2}+\|R_j^3\|_{L^2}+\|R_j^4\|_{L^2}+\|R_j^5\|_{L^2}\big)dt.\label{3.29}
\end{align}
Applying $2^{j(\frac{d}{2}+1)}$ to the both sides of (\ref{3.29}) and summing up with respect to $j\geq j_0$, we obtain
\begin{align}
&\displaystyle\|(\tilde{P},\tilde{\mathbf{u}},\tilde{c})\|^h_{\tilde{L}_t^\infty(\dot{B}_{2,1}^{\frac{d}{2}+1})}+\|(\tilde{P},\tilde{\mathbf{u}})\|^h_{\tilde{L}_t^1(\dot{B}_{2,1}^{\frac{d}{2}+1})}\notag\\[3mm]
\displaystyle\lesssim&\|(\tilde{P}_0,\tilde{\mathbf{u}}_0,\tilde{c}_0)\|^h_{\dot{B}_{2,1}^{\frac{d}{2}+1}}+\Big(\|\partial_t(\kappa_2\!+\!\frac{1}{\kappa_1}\dot{S}_{j-1}(\frac{1}{m}\!-\!\frac{1}{m_\infty}))\|_{\tilde{L}_t^\infty(L^\infty)}
+\|\partial_t(\kappa_2+\kappa_1\dot{S}_{j-1}h(\tilde{P},\tilde{c}))\|_{\tilde{L}_t^\infty(L^\infty)}\notag\\[3mm]
&\displaystyle+\|{\rm div}\big(\kappa_1(\kappa_2\!+\!\frac{1}{\kappa_1}\dot{S}_{j-1}(\frac{1}{m}\!-\!\frac{1}{m_\infty}))\dot{S}_{j-1}\tilde{\mathbf{u}}\big)\|_{\tilde{L}_t^\infty(L^\infty)}\notag\\[3mm]
&+\|{\rm div}\big(\kappa_1(\kappa_2\!+\!\kappa_1\dot{S}_{j-1}h(\tilde{P},\tilde{c}))\dot{S}_{j-1}\tilde{\mathbf{u}}\big)\|_{\tilde{L}_t^\infty(L^\infty)}\notag\\[3mm]
&\displaystyle+\|\nabla\big((\kappa_2+\kappa_1\dot{S}_{j-1}h(\tilde{P},\tilde{c}))(\kappa_2+\frac{1}{\kappa_1}\dot{S}_{j-1}(\frac{1}{m}-\frac{1}{m_\infty}))\big)\|_{\tilde{L}_t^\infty(L^\infty)}\Big)\|
(\tilde{\mathbf{u}},\tilde{P})\|^h_{\tilde{L}_t^1(\dot{B}_{2,1}^{\frac{d}{2}+1})}\notag\\[3mm]
&\displaystyle+\|{\rm div}\dot{S}_{j-1}\tilde{\mathbf{u}}\|_{\tilde{L}_t^1(L^\infty)}\|\tilde{c}\|^h_{\tilde{L}_t^\infty(\dot{B}_{2,1}^{\frac{d}{2}+1})}\notag\\[3mm]
&\displaystyle+\int_0^T\sum_{j\geq j_0}2^{j(\frac{d}{2}+1)}\big(\|R_j^1\|_{L^2}+\|R_j^2\|_{L^2}+\|R_j^3\|_{L^2}+\|R_j^4\|_{L^2}+\|R_j^5\|_{L^2}\big)dt.\label{3.30}
\end{align}

Then, we bound the right hand side of (\ref{3.30}) term-by-term. By taking advantage of  $\eqref{1.5}_1$, $\eqref{1.5}_2$ and Lemma \ref{2 2.7}, we obtain
\begin{align}
&\|\partial_t(\kappa_2\!+\!\frac{1}{\kappa_1}\dot{S}_{j-1}(\frac{1}{m}\!-\!\frac{1}{m_\infty}))\|_{\tilde{L}_t^\infty(L^\infty)}\notag\\
\lesssim&\|\partial_t(\frac{1}{m}\!-\!\frac{1}{m_\infty})\|_{\tilde{L}_t^\infty(L^\infty)}\lesssim\|\partial_t(\tilde{P},\tilde{c})\|_{\tilde{L}_t^\infty(L^\infty)}\notag\\[3mm]
\lesssim&\|-\kappa_2{\rm div}\tilde{\mathbf{u}}-\kappa_1h(\tilde{P},\tilde{c}){\rm div}\tilde{\mathbf{u}}-\kappa_1\tilde{\mathbf{u}}\cdot\nabla\tilde{P}-\kappa_1\tilde{\mathbf{u}}\cdot\nabla\tilde{c}\|_{\tilde{L}_t^\infty(\dot{B}_{p,1}^{\frac{d}{p}})}\notag\\[3mm]
\lesssim&\|\tilde{\mathbf{u}}\|_{\tilde{L}_t^\infty(\dot{B}_{p,1}^{\frac{d}{p}+1})}+\|h(\tilde{P},\tilde{c}){\rm div}\tilde{\mathbf{u}}\|_{\tilde{L}_t^\infty(\dot{B}_{p,1}^{\frac{d}{p}})}+\|\tilde{\mathbf{u}}\cdot\nabla\tilde{P}\|_{\tilde{L}_t^\infty(\dot{B}_{p,1}^{\frac{d}{p}})}
+\|\tilde{\mathbf{u}}\cdot\nabla\tilde{c}\|_{\tilde{L}_t^\infty(\dot{B}_{p,1}^{\frac{d}{p}})}\notag\\[3mm]
\lesssim&\|\tilde{\mathbf{u}}\|_{\tilde{L}_t^\infty(\dot{B}_{p,1}^{\frac{d}{p}+1})}+\|(\tilde{P},\tilde{c})\|_{\tilde{L}_t^\infty(\dot{B}_{p,1}^{\frac{d}{p}})}\|\tilde{\mathbf{u}}\|_{\tilde{L}_t^\infty(\dot{B}_{p,1}^{\frac{d}{p}+1})}
+\|\tilde{\mathbf{u}}\|_{\tilde{L}_t^\infty(\dot{B}_{p,1}^{\frac{d}{p}})}\|\tilde{P}\|_{\tilde{L}_t^\infty(\dot{B}_{p,1}^{\frac{d}{p}+1})}\notag\\
&+\|\tilde{\mathbf{u}}\|_{\tilde{L}_t^\infty(\dot{B}_{p,1}^{\frac{d}{p}})}\|\tilde{c}\|_{\tilde{L}_t^\infty(\dot{B}_{p,1}^{\frac{d}{p}+1})}\notag\\[3mm]
\lesssim&(\|\tilde{\mathbf{u}}\|^l_{\tilde{L}_t^\infty(\dot{B}_{p,1}^{\frac{d}{p}})}+\|\tilde{\mathbf{u}}\|^h_{\tilde{L}_t^\infty(\dot{B}_{2,1}^{\frac{d}{2}+1})})
+(\|(\tilde{P},\tilde{c})\|^l_{\tilde{L}_t^\infty(\dot{B}_{p,1}^{\frac{d}{p}})}+\|(\tilde{P},\tilde{c})\|^h_{\tilde{L}_t^\infty(\dot{B}_{2,1}^{\frac{d}{2}+1})})\notag\\
&(\|\tilde{\mathbf{u}}\|^l_{\tilde{L}_t^\infty(\dot{B}_{p,1}^{\frac{d}{p}})}+\|\tilde{\mathbf{u}}\|^h_{\tilde{L}_t^\infty(\dot{B}_{2,1}^{\frac{d}{2}+1})})\lesssim \mathcal{X}_p^2(t),\label{3.31}
\end{align}
similarly, we have
\begin{align}
&\|\partial_t(\kappa_2+\kappa_1\dot{S}_{j-1}h(\tilde{P},\tilde{c}))\|_{\tilde{L}_t^\infty(L^\infty)}\notag\\[3mm]
\lesssim&(\|\tilde{\mathbf{u}}\|^l_{\tilde{L}_t^\infty(\dot{B}_{p,1}^{\frac{d}{p}})}+\|\tilde{\mathbf{u}}\|^h_{\tilde{L}_t^\infty(\dot{B}_{2,1}^{\frac{d}{2}+1})})
+(\|(\tilde{P},\tilde{c})\|^l_{\tilde{L}_t^\infty(\dot{B}_{p,1}^{\frac{d}{p}})}+\|(\tilde{P},\tilde{c})\|^h_{\tilde{L}_t^\infty(\dot{B}_{2,1}^{\frac{d}{2}+1})})\notag\\
&(\|\tilde{\mathbf{u}}\|^l_{\tilde{L}_t^\infty(\dot{B}_{p,1}^{\frac{d}{p}})}+\|\tilde{\mathbf{u}}\|^h_{\tilde{L}_t^\infty(\dot{B}_{2,1}^{\frac{d}{2}+1})})\lesssim \mathcal{X}_p^2(t).\label{3.32}
\end{align}
Then, we have from Lemma \ref{2 2.7} that
\begin{align}
&\|{\rm div}\big(\kappa_1(\kappa_2\!+\!\frac{1}{\kappa_1}\dot{S}_{j-1}(\frac{1}{m}\!-\!\frac{1}{m_\infty}))\dot{S}_{j-1}\tilde{\mathbf{u}}\big)\|_{\tilde{L}_t^\infty(L^\infty)}\notag\\[3mm]
\lesssim&\|\nabla\big(\kappa_1(\kappa_2\!+\!\frac{1}{\kappa_1}\dot{S}_{j-1}(\frac{1}{m}\!-\!\frac{1}{m_\infty}))\big)\dot{S}_{j-1}\tilde{\mathbf{u}}\|_{\tilde{L}_t^\infty(L^\infty)}
+\|\kappa_1(\kappa_2\!+\!\frac{1}{\kappa_1}\dot{S}_{j-1}(\frac{1}{m}\!-\!\frac{1}{m_\infty})){\rm div}\dot{S}_{j-1}\tilde{\mathbf{u}}\|_{\tilde{L}_t^\infty(L^\infty)}\notag\\[3mm]
\lesssim&\|\nabla(\frac{1}{m}\!-\!\frac{1}{m_\infty})\|_{\tilde{L}_t^\infty(\dot{B}_{p,1}^{\frac{d}{p}})}\|\tilde{\mathbf{u}}\|_{\tilde{L}_t^\infty(\dot{B}_{p,1}^{\frac{d}{p}})}
+\|{\rm div}\tilde{\mathbf{u}}\|_{\tilde{L}_t^\infty(\dot{B}_{p,1}^{\frac{d}{p}})}+\|(\frac{1}{m}\!-\!\frac{1}{m_\infty})\|_{\tilde{L}_t^\infty(\dot{B}_{p,1}^{\frac{d}{p}})}\|{\rm div}\tilde{\mathbf{u}}\|_{\tilde{L}_t^\infty(\dot{B}_{p,1}^{\frac{d}{p}})}\notag\\[3mm]
\lesssim&(\|\tilde{\mathbf{u}}\|^l_{\tilde{L}_t^\infty(\dot{B}_{p,1}^{\frac{d}{p}})}+\|\tilde{\mathbf{u}}\|^h_{\tilde{L}_t^\infty(\dot{B}_{2,1}^{\frac{d}{2}+1})})
+(\|(\tilde{P},\tilde{c})\|^l_{\tilde{L}_t^\infty(\dot{B}_{p,1}^{\frac{d}{p}})}+\|(\tilde{P},\tilde{c})\|^h_{\tilde{L}_t^\infty(\dot{B}_{2,1}^{\frac{d}{2}+1})})\notag\\
&(\|\tilde{\mathbf{u}}\|^l_{\tilde{L}_t^\infty(\dot{B}_{p,1}^{\frac{d}{p}})}+\|\tilde{\mathbf{u}}\|^h_{\tilde{L}_t^\infty(\dot{B}_{2,1}^{\frac{d}{2}+1})})\lesssim \mathcal{X}_p^2(t).\label{3.33}
\end{align}
Similar to (\ref{3.33}), we acquire
\begin{align}
&\|{\rm div}\big(\kappa_1(\kappa_2\!+\!\kappa_1\dot{S}_{j-1}h(\tilde{P},\tilde{c}))\dot{S}_{j-1}\tilde{\mathbf{u}}\big)\|_{\tilde{L}_t^\infty(L^\infty)}\notag\\[2mm]
\lesssim&(\|\tilde{\mathbf{u}}\|^l_{\tilde{L}_t^\infty(\dot{B}_{p,1}^{\frac{d}{p}})}\!+\!\|\tilde{\mathbf{u}}\|^h_{\tilde{L}_t^\infty(\dot{B}_{2,1}^{\frac{d}{2}+1})})
\!+\!(\|(\tilde{P},\tilde{c})\|^l_{\tilde{L}_t^\infty(\dot{B}_{p,1}^{\frac{d}{p}})}\!+\!\|(\tilde{P},\tilde{c})\|^h_{\tilde{L}_t^\infty(\dot{B}_{2,1}^{\frac{d}{2}+1})})(\|\tilde{\mathbf{u}}\|^l_{\tilde{L}_t^\infty(\dot{B}_{p,1}^{\frac{d}{p}})}+\|\tilde{\mathbf{u}}\|^h_{\tilde{L}_t^\infty(\dot{B}_{2,1}^{\frac{d}{2}+1})})\notag\\[2mm] \lesssim&\mathcal{X}_p ^2(t),\notag\\[2mm]
&\|\nabla\big((\kappa_2+\kappa_1\dot{S}_{j-1}h(\tilde{P},\tilde{c}))(\kappa_2+\frac{1}{\kappa_1}\dot{S}_{j-1}(\frac{1}{m}-\frac{1}{m_\infty}))\big)\|_{\tilde{L}_t^\infty(L^\infty)}\notag\\[2mm]
\lesssim&(\|(\tilde{P},\tilde{c})\|^l_{\tilde{L}_t^\infty(\dot{B}_{p,1}^{\frac{d}{p}})}+\|(\tilde{P},\tilde{c})\|^h_{\tilde{L}_t^\infty(\dot{B}_{2,1}^{\frac{d}{2}+1})})
+(\|(\tilde{P},\tilde{c})\|^l_{\tilde{L}_t^\infty(\dot{B}_{p,1}^{\frac{d}{p}})}+\|(\tilde{P},\tilde{c})\|^h_{\tilde{L}_t^\infty(\dot{B}_{2,1}^{\frac{d}{2}+1})})^2\lesssim \mathcal{X}_p^2(t).\label{3.35}
\end{align}
By letting $k=0,\ \sigma_1=\frac{d}{p}+2,\ \sigma_2=\frac{d}{p}$ in Lemma \ref{2 2.3}, it holds that
\begin{align}
&\displaystyle\int_0^T\!\!\sum_{j\geq j_0}2^{j(\frac{d}{2}+1)}\|R_j^1\|_{L^2}
\lesssim \int_0^T\!\!(\|\nabla h(\tilde{P},\tilde{c})\|_{\dot{B}_{p,1}^{\frac{d}{p}}}\|{\rm div}\tilde{\mathbf{u}}\|^h_{\dot{B}_{2,1}^{\frac{d}{2}}}
+\|{\rm div}\tilde{\mathbf{u}}\|^l_{\dot{B}_{p,1}^{\frac{d}{p}-\frac{d}{p^*}}}\|h(\tilde{P},\tilde{c})\|^l_{\dot{B}_{p,1}^{\frac{d}{p}+2}}\notag\\[2mm]
&\displaystyle+\|{\rm div}\tilde{\mathbf{u}}\|_{\dot{B}_{p,1}^{\frac{d}{p}}}\|h(\tilde{P},\tilde{c})\|^h_{\dot{B}_{2,1}^{\frac{d}{2}+1}}
+\|{\rm div}\tilde{\mathbf{u}}\|^l_{\dot{B}_{p,1}^{\frac{d}{p}}}\|\nabla h(\tilde{P},\tilde{c})\|^l_{\dot{B}_{p,1}^{\frac{d}{p}-\frac{d}{p^*}}})\notag\\[2mm]
\displaystyle\lesssim& \int_0^T(\|h(\tilde{P},\tilde{c})\|_{\dot{B}_{p,1}^{\frac{d}{p}+1}}\|{\rm div}\tilde{\mathbf{u}}\|^h_{\dot{B}_{2,1}^{\frac{d}{2}}}
+2^{j_0(1-\frac{d}{p^*})}\|{\rm div}\tilde{\mathbf{u}}\|^l_{\dot{B}_{p,1}^{\frac{d}{p}-1}}\|h(\tilde{P},\tilde{c})\|^l_{\dot{B}_{p,1}^{\frac{d}{p}+2}}\notag\\[2mm]
&\displaystyle+\|{\rm div}\tilde{\mathbf{u}}\|_{\dot{B}_{p,1}^{\frac{d}{p}}}\|h(\tilde{P},\tilde{c})\|^h_{\dot{B}_{2,1}^{\frac{d}{2}+1}}
+2^{j_0(1-\frac{d}{p^*})}\|{\rm div}\tilde{\mathbf{u}}\|^l_{\dot{B}_{p,1}^{\frac{d}{p}}}\|\nabla h(\tilde{P},\tilde{c})\|^l_{\dot{B}_{p,1}^{\frac{d}{p}-1}}).\label{3.36}
\end{align}
Then $\|u\|^l_{\dot{B}_{p,1}^{s_1}}\lesssim 2^{j_0s}\|u\|^l_{\dot{B}_{p,1}^{s_1-s}},\ \|u\|^h_{\dot{B}_{p,1}^{s_1}}\lesssim 2^{-j_0s}\|u\|^h_{\dot{B}_{p,1}^{s_1+s}}\ (s>0,\ j\geq j_0)$  and Lemma \ref{2 2.4} give
\begin{equation}\label{3.37}
\begin{array}{ll}
\|h(\tilde{P},\tilde{c})\|^l_{\dot{B}_{p,1}^{\frac{d}{p}+2}}
\lesssim\|h(\tilde{P},\tilde{c})\|^l_{\dot{B}_{p,1}^{\frac{d}{p}+1}}
\lesssim\|h(\tilde{P},\tilde{c})\|_{\dot{B}_{p,1}^{\frac{d}{p}+1}}
\lesssim\|(\tilde{P},\tilde{c})\|^l_{\dot{B}_{p,1}^{\frac{d}{p}}}+\|(\tilde{P},\tilde{c})\|^h_{\dot{B}_{2,1}^{\frac{d}{2}+1}},
\end{array}
\end{equation}
\begin{equation}\label{3.38}
\begin{array}{ll}
\|\nabla h(\tilde{P},\tilde{c})\|^l_{\dot{B}_{p,1}^{\frac{d}{p}-1}}
\lesssim\|h(\tilde{P},\tilde{c})\|^l_{\dot{B}_{p,1}^{\frac{d}{p}}}
\lesssim\|h(\tilde{P},\tilde{c})\|_{\dot{B}_{p,1}^{\frac{d}{p}}}
\lesssim\|(\tilde{P},\tilde{c})\|^l_{\dot{B}_{p,1}^{\frac{d}{p}}}+\|(\tilde{P},\tilde{c})\|^h_{\dot{B}_{2,1}^{\frac{d}{2}+1}},
\end{array}
\end{equation}
\begin{equation}\label{3.39}
\begin{array}{ll}
\|h(\tilde{P},\tilde{c})\|^h_{\dot{B}_{2,1}^{\frac{d}{2}+1}}\leq C(1+\|(\tilde{P},\tilde{c})\|^l_{\dot{B}_{p,1}^{\frac{d}{p}}}+2^{-j_0}\|(\tilde{P},\tilde{c})\|^h_{\dot{B}_{2,1}^{\frac{d}{2}+1}})(\|(\tilde{P},\tilde{c})\|^l_{\dot{B}_{p,1}^{\frac{d}{p}+1}}
+\|(\tilde{P},\tilde{c})\|^h_{\dot{B}_{2,1}^{\frac{d}{2}+1}}).
\end{array}
\end{equation}
Inserting (\ref{3.37})-(\ref{3.39}) into (\ref{3.36}), we get
\begin{align}
\displaystyle\int_0^T\!\!\sum_{j\geq j_0}&2^{j(\frac{d}{2}+1)}\|R_j^1\|_{L^2}
\displaystyle\lesssim \int_0^T\!\!(\|\tilde{\mathbf{u}}\|^h_{\dot{B}_{2,1}^{\frac{d}{2}+1}}
+\|\tilde{\mathbf{u}}\|^l_{\dot{B}_{p,1}^{\frac{d}{p}}})(\|(\tilde{P},\tilde{c})\|^l_{\dot{B}_{p,1}^{\frac{d}{p}}}+\|(\tilde{P},\tilde{c})\|^h_{\dot{B}_{2,1}^{\frac{d}{2}+1}})
\notag\\[2mm]
&\displaystyle+(\|\tilde{\mathbf{u}}\|^l_{\dot{B}_{p,1}^{\frac{d}{p}}}
\!\!+\|\tilde{\mathbf{u}}\|^h_{\dot{B}_{2,1}^{\frac{d}{2}+1}})(1\!+\!\|(\tilde{P},\tilde{c})\|^l_{\dot{B}_{p,1}^{\frac{d}{p}}}
\!+2^{-j_0}\|(\tilde{P},\tilde{c})\|^h_{\dot{B}_{2,1}^{\frac{d}{2}+1}})
(\|(\tilde{P},\tilde{c})\|^l_{\dot{B}_{p,1}^{\frac{d}{p}+1}}
\!\!+\|(\tilde{P},\tilde{c})\|^h_{\dot{B}_{2,1}^{\frac{d}{2}+1}})\notag\\[2mm]
\lesssim&(\|\tilde{\mathbf{u}}\|^h_{\tilde{L}_t^1(\dot{B}_{2,1}^{\frac{d}{2}+1})}+\|\tilde{\mathbf{u}}\|^l_{\tilde{L}_t^1(\dot{B}_{p,1}^{\frac{d}{p}})})
(\|(\tilde{P},\tilde{c})\|^l_{\tilde{L}_t^\infty(\dot{B}_{p,1}^{\frac{d}{p}})}+\|(\tilde{P},\tilde{c})\|^h_{\tilde{L}_t^\infty(\dot{B}_{2,1}^{\frac{d}{2}+1})})
\notag\\[2mm]
&+(\|\tilde{\mathbf{u}}\|^l_{\tilde{L}_t^1(\dot{B}_{p,1}^{\frac{d}{p}})}\!\!+\|\tilde{\mathbf{u}}\|^h_{\tilde{L}_t^1(\dot{B}_{2,1}^{\frac{d}{2}+1})})
(\|(\tilde{P},\tilde{c})\|^l_{\tilde{L}_t^\infty(\dot{B}_{p,1}^{\frac{d}{p}})}+\|(\tilde{P},\tilde{c})\|^h_{\tilde{L}_t^\infty(\dot{B}_{2,1}^{\frac{d}{2}+1})})^2\notag\\
\lesssim& \mathcal{X}_p^2(t)+\mathcal{X}_p^3(t).\label{3.40}
\end{align}
By taking $k=0,\ \sigma_1=\frac{d}{p}+2,\ \sigma_2=\frac{d}{p}$ in Lemma \ref{2 2.3} again, we have
\begin{align}
\displaystyle\int_0^T\sum_{j\geq j_0}2^{j(\frac{d}{2}+1)}\|R_j^2\|_{L^2}
\displaystyle\lesssim&\int_0^T(\|\nabla \tilde{\mathbf{u}}\|_{\dot{B}_{p,1}^{\frac{d}{p}}}\|\nabla\tilde{P}\|^h_{\dot{B}_{2,1}^{\frac{d}{2}}}
+\|\nabla\tilde{P}\|^l_{\dot{B}_{p,1}^{\frac{d}{p}-\frac{d}{p^*}}}\|\tilde{\mathbf{u}}\|^l_{\dot{B}_{p,1}^{\frac{d}{p}+2}}\notag\\[2mm]
&\displaystyle+\|\nabla\tilde{P}\|_{\dot{B}_{p,1}^{\frac{d}{p}}}\|\tilde{\mathbf{u}}\|^h_{\dot{B}_{2,1}^{\frac{d}{2}+1}}
+\|\nabla\tilde{P}\|^l_{\dot{B}_{p,1}^{\frac{d}{p}}}\|\nabla \tilde{\mathbf{u}}\|^l_{\dot{B}_{p,1}^{\frac{d}{p}-\frac{d}{p^*}}})\notag\\[2mm]
\displaystyle\lesssim&\int_0^T(\|\nabla \tilde{\mathbf{u}}\|_{\dot{B}_{p,1}^{\frac{d}{p}}}\|\nabla\tilde{P}\|^h_{\dot{B}_{2,1}^{\frac{d}{2}}}
+2^{j_0(1-\frac{d}{p^*})}\|\nabla\tilde{P}\|^l_{\dot{B}_{p,1}^{\frac{d}{p}-1}}\|\tilde{\mathbf{u}}\|^l_{\dot{B}_{p,1}^{\frac{d}{p}+2}}\notag\\[2mm]
&\displaystyle+\|\nabla\tilde{P}\|_{\dot{B}_{p,1}^{\frac{d}{p}}}\|\tilde{\mathbf{u}}\|^h_{\dot{B}_{2,1}^{\frac{d}{2}+1}}
+2^{j_0(1-\frac{d}{p^*})}\|\nabla\tilde{P}\|^l_{\dot{B}_{p,1}^{\frac{d}{p}}}\|\nabla \tilde{\mathbf{u}}\|^l_{\dot{B}_{p,1}^{\frac{d}{p}-1}})\notag\\[2mm]
\displaystyle\lesssim&\int_0^T(\|\tilde{\mathbf{u}}\|^l_{\dot{B}_{p,1}^{\frac{d}{p}+1}}+\|\tilde{\mathbf{u}}\|^h_{\dot{B}_{2,1}^{\frac{d}{2}+1}})
(\|\tilde{P}\|^l_{\dot{B}_{p,1}^{\frac{d}{p}+1}}+\|\nabla\tilde{P}\|^h_{\dot{B}_{2,1}^{\frac{d}{2}}})
+\|\tilde{P}\|^l_{\dot{B}_{p,1}^{\frac{d}{p}}}\|\tilde{\mathbf{u}}\|^l_{\dot{B}_{p,1}^{\frac{d}{p}}}\notag\\[2mm]
\displaystyle\lesssim&(\|\tilde{\mathbf{u}}\|^l_{\tilde{L}_t^1(\dot{B}_{p,1}^{\frac{d}{p}+1})}+\|\tilde{\mathbf{u}}\|^h_{\tilde{L}_t^1(\dot{B}_{2,1}^{\frac{d}{2}+1})})
(\|\tilde{P}\|^l_{\tilde{L}_t^\infty(\dot{B}_{p,1}^{\frac{d}{p}+1})}+\|\nabla\tilde{P}\|^h_{\tilde{L}_t^\infty(\dot{B}_{2,1}^{\frac{d}{2}})})\notag\\
&\!+\!\|\tilde{P}\|^l_{\tilde{L}_t^\infty(\dot{B}_{p,1}^{\frac{d}{p}})}\|\tilde{\mathbf{u}}\|^l_{\tilde{L}_t^1(\dot{B}_{p,1}^{\frac{d}{p}})}\notag\\[2mm]
\lesssim& \mathcal{X}_p^2(t),\label{3.41}
\end{align}
and
\begin{align}
\displaystyle\int_0^T\sum_{j\geq j_0}2^{j(\frac{d}{2}+1)}\|R_j^3\|_{L^2}
\lesssim& \int_0^T(\|\nabla \tilde{\mathbf{u}}\|_{\dot{B}_{p,1}^{\frac{d}{p}}}\|\nabla\tilde{\mathbf{u}}\|^h_{\dot{B}_{2,1}^{\frac{d}{2}}}
+\|\nabla\tilde{\mathbf{u}}\|^l_{\dot{B}_{p,1}^{\frac{d}{p}-\frac{d}{p^*}}}\|\tilde{\mathbf{u}}\|^l_{\dot{B}_{p,1}^{\frac{d}{p}+2}}\notag\\[2mm]
&\ \ \ \ \ +\|\nabla\tilde{\mathbf{u}}\|_{\dot{B}_{p,1}^{\frac{d}{p}}}\|\tilde{\mathbf{u}}\|^h_{\dot{B}_{2,1}^{\frac{d}{2}+1}}
+\|\nabla\tilde{\mathbf{u}}\|^l_{\dot{B}_{p,1}^{\frac{d}{p}}}\|\nabla \tilde{\mathbf{u}}\|^l_{\dot{B}_{p,1}^{\frac{d}{p}-\frac{d}{p^*}}})\notag\\[2mm]
\displaystyle\lesssim&\int_0^T\|\nabla\tilde{\mathbf{u}}\|_{\dot{B}_{p,1}^{\frac{d}{p}}}\|\tilde{\mathbf{u}}\|^h_{\dot{B}_{2,1}^{\frac{d}{2}+1}}
+2^{j_0(1-\frac{d}{p^*})}\|\nabla\tilde{\mathbf{u}}\|^l_{\dot{B}_{p,1}^{\frac{d}{p}-1}}(\|\tilde{\mathbf{u}}\|^l_{\dot{B}_{p,1}^{\frac{d}{p}+2}}
+\|\nabla\tilde{\mathbf{u}}\|^l_{\dot{B}_{p,1}^{\frac{d}{p}}})\notag\\[2mm]
\displaystyle\lesssim&(\|\tilde{\mathbf{u}}\|^l_{\tilde{L}_t^1(\dot{B}_{p,1}^{\frac{d}{p}+1})}+\|\tilde{\mathbf{u}}\|^h_{\tilde{L}_t^1(\dot{B}_{2,1}^{\frac{d}{2}+1})})
\|\tilde{\mathbf{u}}\|^h_{\tilde{L}_t^\infty(\dot{B}_{2,1}^{\frac{d}{2}+1})}
+\|\tilde{\mathbf{u}}\|^l_{\tilde{L}_t^\infty(\dot{B}_{p,1}^{\frac{d}{p}})}\|\tilde{\mathbf{u}}\|^l_{\tilde{L}_t^1(\dot{B}_{p,1}^{\frac{d}{p}+1})}\notag\\
\lesssim &\mathcal{X}_p^2(t).\label{3.42}
\end{align}
Similar to $R_j^1$, we have
\begin{align}
&\displaystyle\int_0^T\sum_{j\geq j_0}2^{j(\frac{d}{2}+1)}\|R_j^4\|_{L^2}\notag\\[3mm]
\lesssim& \int_0^T(\|\nabla (\frac{1}{m}-\frac{1}{m_\infty})\|_{\dot{B}_{p,1}^{\frac{d}{p}}}\|\nabla\tilde{P}\|^h_{\dot{B}_{2,1}^{\frac{d}{2}}}
+\|\nabla\tilde{P}\|^l_{\dot{B}_{p,1}^{\frac{d}{p}-\frac{d}{p^*}}}\|(\frac{1}{m}-\frac{1}{m_\infty})\|^l_{\dot{B}_{p,1}^{\sigma_1}}\notag\\[3mm]
&\ \ \ \ +\|\nabla\tilde{P}\|_{\dot{B}_{p,1}^{\frac{d}{p}}}\|(\frac{1}{m}-\frac{1}{m_\infty})\|^h_{\dot{B}_{2,1}^{\frac{d}{2}+1}}
+\|\nabla\tilde{P}\|^l_{\dot{B}_{p,1}^{\sigma_2}}\|\nabla (\frac{1}{m}-\frac{1}{m_\infty})\|^l_{\dot{B}_{p,1}^{\frac{d}{p}-\frac{d}{p^*}}})\notag\\[3mm]
\displaystyle\lesssim& \mathcal{X}_p^2(t)+\mathcal{X}_p^3(t)+\int_0^T\|\nabla\tilde{P}\|^l_{\dot{B}_{p,1}^{\frac{d}{p}-\frac{d}{p^*}}}\|(\frac{1}{m}-\frac{1}{m_\infty})\|^l_{\dot{B}_{p,1}^{\sigma_1}}.\label{3.43}
\end{align}
Letting $M(\tilde{P},\tilde{c}):=\frac{1}{m}-\frac{1}{m_\infty}$ and noticing that $M(0,0)=0$,  we may write that $M(\tilde{P},\tilde{c})=M'_{\tilde{P}}(0)\tilde{P}+M'_{\tilde{c}}(0)\tilde{c}+\tilde{M}(\tilde{P},\tilde{c})\tilde{P}\tilde{c}$ and get
\begin{align}
&\int_0^T\|\nabla\tilde{P}\|^l_{\dot{B}_{p,1}^{\frac{d}{p}-\frac{d}{p^*}}}\|(\frac{1}{m}-\frac{1}{m_\infty})\|^l_{\dot{B}_{p,1}^{\sigma_1}}\notag\\[3mm]
\lesssim&\int_0^T\|\nabla\tilde{P}\|^l_{\dot{B}_{p,1}^{\frac{d}{p}-\frac{d}{p^*}}}
\|M'_{\tilde{P}}(0)\tilde{P}+M'_{\tilde{c}}(0)\tilde{c}+\tilde{M}(\tilde{P},\tilde{c})\tilde{P}\tilde{c}\|^l_{\dot{B}_{p,1}^{\sigma_1}}\notag\\[3mm]
\lesssim&\int_0^T\|\nabla\tilde{P}\|^l_{\dot{B}_{p,1}^{\frac{d}{p}-\frac{d}{p^*}}}
(\|M'_{\tilde{P}}(0)\tilde{P}\|^l_{\dot{B}_{p,1}^{\sigma_1}}+\|M'_{\tilde{c}}(0)\tilde{c}\|^l_{\dot{B}_{p,1}^{\sigma_1}}
+\|\tilde{M}(\tilde{P},\tilde{c})\tilde{P}\tilde{c}\|^l_{\dot{B}_{p,1}^{\sigma_1}})\notag\\[3mm]
\lesssim&\int_0^T\|\nabla\tilde{P}\|^l_{\dot{B}_{p,1}^{\frac{d}{p}-\frac{d}{p^*}}}
\|\tilde{P}\|^l_{\dot{B}_{p,1}^{\sigma_1}}+\|\nabla\tilde{P}\|^l_{\dot{B}_{p,1}^{\frac{d}{p}-\frac{d}{p^*}}}\|\tilde{c}\|^l_{\dot{B}_{p,1}^{\sigma_1}}
+\|\nabla\tilde{P}\|^l_{\dot{B}_{p,1}^{\frac{d}{p}-\frac{d}{p^*}}}\|\tilde{M}(\tilde{P},\tilde{c})\tilde{P}\tilde{c}\|^l_{\dot{B}_{p,1}^{\sigma_1}}.\label{3.44}
\end{align}
Then, we bounded the three terms on the right hand side of \eqref{3.44}, respectively. From $\|u\|^l_{\dot{B}_{p,1}^{s_1}}\lesssim 2^{j_0s}\|u\|^l_{\dot{B}_{p,1}^{s_1-s}}\ (1-\frac{d}{p^\ast}>0)$, the first term can be estimated as following:
\begin{align}
\int_0^T\|\nabla\tilde{P}\|^l_{\dot{B}_{p,1}^{\frac{d}{p}-\frac{d}{p^*}}}\|\tilde{P}\|^l_{\dot{B}_{p,1}^{\sigma_1}}
&\lesssim\int_0^T 2^{j_0(1-\frac{d}{p^\ast})}\|\nabla\tilde{P}\|^l_{\dot{B}_{p,1}^{\frac{d}{p}-1}}\|\tilde{P}\|^l_{\dot{B}_{p,1}^{\frac{d}{p}+1}}\notag\\
&\lesssim\|\nabla\tilde{P}\|^l_{\tilde{L}_t^\infty(\dot{B}_{p,1}^{\frac{d}{p}-1})}\|\tilde{P}\|^l_{\tilde{L}_t^1(\dot{B}_{p,1}^{\frac{d}{p}+1})}
\lesssim\mathcal{X}_p^2(t).\label{3.441}
\end{align}
Let $p^*=\infty\ (p=2)$ and $\sigma_1=\frac{d}{2}+1$, we acquire
\begin{align}
\int_0^T\|\nabla\tilde{P}\|^l_{\dot{B}_{p,1}^{\frac{d}{p}-\frac{d}{p^*}}}\|\tilde{c}\|^l_{\dot{B}_{p,1}^{\sigma_1}}
&\lesssim\int_0^T\|\nabla\tilde{P}\|^l_{\dot{B}_{p,1}^{\frac{d}{p}}}\|\tilde{c}\|^l_{\dot{B}_{2,1}^{\frac{d}{2}+1}}
\lesssim\|\nabla\tilde{P}\|^l_{\tilde{L}_t^1(\dot{B}_{p,1}^{\frac{d}{p}})}\|\tilde{c}\|^l_{\tilde{L}_t^\infty(\dot{B}_{2,1}^{\frac{d}{2}+1})}
\lesssim\mathcal{X}_p^2(t).\label{3.442}
\end{align}
From Lemma \ref{2 2.7}, we have
\begin{align}
&\int_0^T\|\nabla\tilde{P}\|^l_{\dot{B}_{p,1}^{\frac{d}{p}-\frac{d}{p^*}}}\|\tilde{M}(\tilde{P},\tilde{c})\tilde{P}\tilde{c}\|^l_{\dot{B}_{p,1}^{\frac{d}{p}+1}}\notag\\
\lesssim&\int_0^T\|\nabla\tilde{P}\|^l_{\dot{B}_{p,1}^{\frac{d}{p}-\frac{d}{p^*}}}
(\|\tilde{M}(\tilde{P},\tilde{c})\tilde{P}\|_{\dot{B}_{p,1}^{\frac{d}{p}+1}}\|\tilde{c}\|_{L^\infty}
+\|\tilde{c}\|_{\dot{B}_{p,1}^{\frac{d}{p}+1}}\|\tilde{M}(\tilde{P},\tilde{c})\tilde{P}\|_{L^\infty})\notag\\[3mm]
\lesssim&\int_0^T\|\nabla\tilde{P}\|^l_{\dot{B}_{p,1}^{\frac{d}{p}-\frac{d}{p^*}}}[
(\|\tilde{M}(\tilde{P},\tilde{c})\|_{L^\infty}\|\tilde{P}\|_{\dot{B}_{p,1}^{\frac{d}{p}+1}}+\|\tilde{M}(\tilde{P},\tilde{c})\|_{\dot{B}_{p,1}^{\frac{d}{p}+1}}\|\tilde{P}\|_{L^\infty})\|\tilde{c}\|_{L^\infty}
\notag\\[3mm]
&\quad\ \ \ \ +\|\tilde{c}\|_{\dot{B}_{p,1}^{\frac{d}{p}+1}}\|\tilde{M}(\tilde{P},\tilde{c})\|_{L^\infty}\|\tilde{P}\|_{L^\infty}]\notag\\[3mm]
\lesssim&\|\nabla\tilde{P}\|^l_{\tilde{L}_t^\infty(\dot{B}_{p,1}^{\frac{d}{p}-\frac{d}{p^*}})}\|\tilde{M}(\tilde{P},\tilde{c})\|_{\tilde{L}_t^\infty(L^\infty)}\|\tilde{P}\|_{\tilde{L}_t^1(\dot{B}_{p,1}^{\frac{d}{p}+1})}\|\tilde{c}\|_{\tilde{L}_t^\infty(L^\infty)}\notag\\[3mm]
&+\|\nabla\tilde{P}\|^l_{\tilde{L}_t^2(\dot{B}_{p,1}^{\frac{d}{p}-\frac{d}{p^*}})}\|\tilde{M}(\tilde{P},\tilde{c})\|_{\tilde{L}_t^\infty(\dot{B}_{p,1}^{\frac{d}{p}+1})}\|\tilde{P}\|_{\tilde{L}_t^2(L^\infty)}\|\tilde{c}\|_{\tilde{L}_t^\infty(L^\infty)}\notag\\[3mm]
&+\|\nabla\tilde{P}\|^l_{\tilde{L}_t^2(\dot{B}_{p,1}^{\frac{d}{p}-\frac{d}{p^*}})}\|\tilde{c}\|_{\tilde{L}_t^\infty(\dot{B}_{p,1}^{\frac{d}{p}+1})}\|\tilde{M}(\tilde{P},\tilde{c})\|_{\tilde{L}_t^\infty(L^\infty)}\|\tilde{P}\|_{\tilde{L}_t^2(L^\infty)}.\label{3.45}
\end{align}
By using interpolation inequalities, we have the following inequalities:
\begin{align}
\|\tilde{P}\|^l_{\tilde{L}_t^2(\dot{B}_{p,1}^{\frac{d}{p}})}\leq\big(\|\tilde{P}\|^l_{\tilde{L}_t^\infty(\dot{B}_{p,1}^{\frac{d}{p}-1})}\big)^{\frac{1}{2}}
\big(\|\tilde{P}\|^l_{\tilde{L}_t^1(\dot{B}_{p,1}^{\frac{d}{p}+1})}\big)^{\frac{1}{2}},\notag\\[3mm]
\|\tilde{P}\|^h_{\tilde{L}_t^2(\dot{B}_{2,1}^{\frac{d}{2}+1})}\leq\big(\|\tilde{P}\|^h_{\tilde{L}_t^\infty(\dot{B}_{2,1}^{\frac{d}{2}-1})}\big)^{\frac{1}{2}}
\big(\|\tilde{P}\|^h_{\tilde{L}_t^1(\dot{B}_{2,1}^{\frac{d}{2}+1})}\big)^{\frac{1}{2}}.\label{3.46}
\end{align}
Inserting (\ref{3.46}) into (\ref{3.45}) and using Lemma \ref{A.13}, we obtain
\begin{align}
&\int_0^T\|\nabla\tilde{P}\|^l_{\dot{B}_{p,1}^{\frac{d}{p}-\frac{d}{p^*}}}\|\tilde{M}(\tilde{P},\tilde{c})\tilde{P}\tilde{c}\|^l_{\dot{B}_{p,1}^{\frac{d}{p}+1}}\notag\\
\lesssim&\|\nabla\tilde{P}\|^l_{\tilde{L}_t^\infty(\dot{B}_{p,1}^{\frac{d}{p}-1})}\|(\tilde{P},\tilde{c})\|_{\tilde{L}_t^\infty(\dot{B}_{p,1}^{\frac{d}{p}})}\|\tilde{P}\|_{\tilde{L}_t^1(\dot{B}_{p,1}^{\frac{d}{p}+1})}\|\tilde{c}\|_{\tilde{L}_t^1(\dot{B}_{p,1}^{\frac{d}{p}})}\notag\\[3mm]
&+\|\tilde{P}\|^l_{\tilde{L}_t^2(\dot{B}_{p,1}^{\frac{d}{p}})}(\|\tilde{P}\|^l_{\tilde{L}_t^2(\dot{B}_{p,1}^{\frac{d}{p}})}+\|\tilde{P}\|^h_{\tilde{L}_t^2(\dot{B}_{2,1}^{\frac{d}{2}+1})})
\|(\tilde{P},\tilde{c})\|_{\tilde{L}_t^\infty(\dot{B}_{p,1}^{\frac{d}{p}+1})}\|\tilde{c}\|_{\tilde{L}_t^\infty(\dot{B}_{p,1}^{\frac{d}{p}})}\notag\\[3mm]
&+\|\tilde{P}\|^l_{\tilde{L}_t^2(\dot{B}_{p,1}^{\frac{d}{p}})}(\|\tilde{P}\|^l_{\tilde{L}_t^2(\dot{B}_{p,1}^{\frac{d}{p}})}+\|\tilde{P}\|^h_{\tilde{L}_t^2(\dot{B}_{2,1}^{\frac{d}{2}+1})})
\|\tilde{c}\|_{\tilde{L}_t^\infty(\dot{B}_{p,1}^{\frac{d}{p}+1})}\|(\tilde{P},\tilde{c})\|_{\tilde{L}_t^\infty(\dot{B}_{p,1}^{\frac{d}{p}})}\notag\\[3mm]
\lesssim&\|\nabla\tilde{P}\|^l_{\tilde{L}_t^\infty(\dot{B}_{p,1}^{\frac{d}{p}-1})}\|(\tilde{P},\tilde{c})\|_{\tilde{L}_t^\infty(\dot{B}_{p,1}^{\frac{d}{p}})}\|\tilde{P}\|_{\tilde{L}_t^1(\dot{B}_{p,1}^{\frac{d}{p}+1})}\|\tilde{c}\|_{\tilde{L}_t^1(\dot{B}_{p,1}^{\frac{d}{p}})}\notag\\[3mm]
&+(\|\tilde{P}\|^l_{\tilde{L}_t^\infty(\dot{B}_{p,1}^{\frac{d}{p}-1})}+\|\tilde{P}\|^h_{\tilde{L}_t^\infty(\dot{B}_{2,1}^{\frac{d}{2}-1})})
 (\|\tilde{P}\|^l_{\tilde{L}_t^1(\dot{B}_{p,1}^{\frac{d}{p}+1})}+\|\tilde{P}\|^h_{\tilde{L}_t^1(\dot{B}_{2,1}^{\frac{d}{2}+1})})
 \|(\tilde{P},\tilde{c})\|_{\tilde{L}_t^\infty(\dot{B}_{p,1}^{\frac{d}{p}+1})}\|\tilde{c}\|_{\tilde{L}_t^\infty(\dot{B}_{p,1}^{\frac{d}{p}})}\notag\\[3mm]
&+(\|\tilde{P}\|^l_{\tilde{L}_t^\infty(\dot{B}_{p,1}^{\frac{d}{p}-1})}+\|\tilde{P}\|^h_{\tilde{L}_t^\infty(\dot{B}_{2,1}^{\frac{d}{2}-1})})
 (\|\tilde{P}\|^l_{\tilde{L}_t^1(\dot{B}_{p,1}^{\frac{d}{p}+1})}+\|\tilde{P}\|^h_{\tilde{L}_t^1(\dot{B}_{2,1}^{\frac{d}{2}+1})})
 \|\tilde{c}\|_{\tilde{L}_t^\infty(\dot{B}_{p,1}^{\frac{d}{p}+1})}\|(\tilde{P},\tilde{c})\|_{\tilde{L}_t^\infty(\dot{B}_{p,1}^{\frac{d}{p}})}\notag\\[5mm]
\lesssim &\mathcal{X}_p^4(t).\label{3.47}
\end{align}
Thus, combining (\ref{3.44}), (\ref{3.441}), (\ref{3.442}) and (\ref{3.47}), one has
\begin{align}
\int_0^T\|\nabla\tilde{P}\|^l_{\dot{B}_{p,1}^{\frac{d}{p}-\frac{d}{p^*}}}\|(\frac{1}{m}-\frac{1}{m_\infty})\|^l_{\dot{B}_{p,1}^{\sigma_1}}
\lesssim\mathcal{X}_p^2(t).\label{3.471}
\end{align}
Hence, from \eqref{3.43}, we easily deduce that
\begin{align}
\displaystyle\int_0^T\sum_{j\geq j_0}2^{j(\frac{d}{2}+1)}\|R_j^4\|_{L^2}
\displaystyle\lesssim \mathcal{X}_p^2(t)+\mathcal{X}_p^3(t)+\mathcal{X}_p^4(t).\label{3.472}
\end{align}
Then, we bound the last term of (\ref{3.30}) as
\begin{align}
&\displaystyle\int_0^T\sum_{j\geq j_0}2^{j(\frac{d}{2}+1)}\|R_j^5\|_{L^2}\notag\\[3mm]
\displaystyle\lesssim& \int_0^T(\|\nabla \tilde{\mathbf{u}}\|_{\dot{B}_{p,1}^{\frac{d}{p}}}\|\nabla\tilde{c}\|^h_{\dot{B}_{2,1}^{\frac{d}{2}}}
+\|\nabla\tilde{c}\|^l_{\dot{B}_{p,1}^{\frac{d}{p}-\frac{d}{p^*}}}\|\tilde{\mathbf{u}}\|^l_{\dot{B}_{p,1}^{\frac{d}{p}+2}}
+\|\nabla\tilde{c}\|_{\dot{B}_{p,1}^{\frac{d}{p}}}\|\tilde{\mathbf{u}}\|^h_{\dot{B}_{2,1}^{\frac{d}{2}+1}}
+\|\nabla\tilde{c}\|^l_{\dot{B}_{p,1}^{\frac{d}{p}}}\|\nabla \tilde{\mathbf{u}}\|^l_{\dot{B}_{p,1}^{\frac{d}{p}-\frac{d}{p^*}}})\notag\\[3mm]
\displaystyle\lesssim&(\|\nabla \tilde{\mathbf{u}}\|_{\tilde{L}_t^1(\dot{B}_{p,1}^{\frac{d}{p}})}\|\nabla\tilde{c}\|^h_{\tilde{L}_t^\infty(\dot{B}_{2,1}^{\frac{d}{2}})}
+\|\nabla\tilde{c}\|^l_{\tilde{L}_t^\infty(\dot{B}_{p,1}^{\frac{d}{p}-1})}\|\tilde{\mathbf{u}}\|^l_{\tilde{L}_t^1(\dot{B}_{p,1}^{\frac{d}{p}+2})}\notag\\[3mm]
&+\|\nabla\tilde{c}\|_{\tilde{L}_t^\infty(\dot{B}_{p,1}^{\frac{d}{p}})}\|\tilde{\mathbf{u}}\|^h_{\tilde{L}_t^1(\dot{B}_{2,1}^{\frac{d}{2}+1})}
+\|\nabla\tilde{c}\|^l_{\tilde{L}_t^\infty(\dot{B}_{p,1}^{\frac{d}{p}})}\|\nabla \tilde{\mathbf{u}}\|^l_{\tilde{L}_t^1(\dot{B}_{p,1}^{\frac{d}{p}-1})})\notag\\[3mm]
\displaystyle\lesssim&\|\nabla \tilde{\mathbf{u}}\|_{\tilde{L}_t^1(\dot{B}_{p,1}^{\frac{d}{p}})}\|\nabla\tilde{c}\|^h_{\tilde{L}_t^\infty(\dot{B}_{2,1}^{\frac{d}{2}})}
+\|\nabla\tilde{c}\|_{\tilde{L}_t^\infty(\dot{B}_{p,1}^{\frac{d}{p}})}\|\tilde{\mathbf{u}}\|^h_{\tilde{L}_t^1(\dot{B}_{2,1}^{\frac{d}{2}+1})}
+\|\tilde{c}\|^l_{\tilde{L}_t^\infty(\dot{B}_{p,1}^{\frac{d}{p}})}\|\tilde{\mathbf{u}}\|^l_{\tilde{L}_t^1(\dot{B}_{p,1}^{\frac{d}{p}})}
\lesssim \mathcal{X}_p^2(t).\label{3.48}
\end{align}
Therefore, we get
\begin{align}
&\displaystyle\|(\tilde{P},\tilde{\mathbf{u}},\tilde{c})\|^h_{\tilde{L}_t^\infty(\dot{B}_{2,1}^{\frac{d}{2}+1})}+\|(\tilde{P},\tilde{\mathbf{u}})\|^h_{\tilde{L}_t^1(\dot{B}_{2,1}^{\frac{d}{2}+1})}
\displaystyle\lesssim\mathcal{X}_{p,0}+\mathcal{X}_p^2(t)+\mathcal{X}_p^3(t)+\mathcal{X}_p^4(t).\label{3.49}
\end{align}
Summing up (\ref{3.16}) and (\ref{3.49}), we have
\begin{align}
&\|\tilde{P}\|^l_{\tilde{L}_T^{\infty}(\dot{B}_{p,1}^{\frac{d}{p}-1})}
+\|\tilde{P}\|^l_{\tilde{L}_T^{1}(\dot{B}_{p,1}^{\frac{d}{p}+1})}
+\|\tilde{\mathbf{u}}\|^l_{\tilde{L}_T^{\infty}(\dot{B}_{p,1}^{\frac{d}{p}})}
+\|\tilde{\mathbf{u}}\|^l_{\tilde{L}_T^{1}(\dot{B}_{p,1}^{\frac{d}{p}})}
+\|\tilde{c}\|^l_{\tilde{L}_T^\infty(\dot{B}_{2,1}^{\frac{d}{2}-1})}\notag\\[6mm]
&+\|(\tilde{P},\tilde{\mathbf{u}},\tilde{c})\|^h_{\tilde{L}_t^\infty(\dot{B}_{2,1}^{\frac{d}{2}+1})}+\|(\tilde{P},\tilde{\mathbf{u}})\|^h_{\tilde{L}_t^1(\dot{B}_{2,1}^{\frac{d}{2}+1})}
\lesssim \mathcal{X}_{p,0}+\mathcal{X}^2_p(t)+\mathcal{X}^3_p(t)++\mathcal{X}^4_p(t).
\end{align}
Thus, we have completed the proof of Proposition (\ref{3 1.1}).

\section{Proof of global existence and uniqueness}
\quad\quad Now, we want to focus on the existence and uniqueness of local-in-time solutions for system \ref{1.1}. It is essential to prove the local existence before proving the global one. First, define the space
\begin{align*}
	E(T)\triangleq&\big\{(\tilde{P},\tilde{\mathbf{u}},\tilde{c}) : \tilde{P}^l\in \mathcal{C}([0,T];\dot{B}^{\frac{d}{p}-1}_{p,1}), \tilde{\mathbf{u}}^l\in \mathcal{C}([0,T];\dot{B}^{\frac{d}{p}}_{p,1}), \tilde{c}^l\in \mathcal{C}([0,T];\dot{B}^{\frac{d}{2}-1}_{2,1}),\\
	&(\tilde{P}^h, \tilde{\mathbf{u}}^h, \tilde{c}^h)\in \mathcal{C}([0,T];\dot{B}^{\frac{d}{2}+1}_{2,1})\big\}
\end{align*}
and its norm
\begin{align*}
	\|(\tilde{P}, \tilde{\mathbf{u}}, \tilde{c})\|_{E(T)}\triangleq\|\tilde{P}\|^l_{L^\infty_t(\dot{B}^{\frac{d}{p}-1}_{p,1})}+\|\tilde{\mathbf{u}}\|^l_{L^\infty_t(\dot{B}^{\frac{d}{p}}_{p,1})}+\|\tilde{c}\|^l_{L^\infty_t(\dot{B}^{\frac{d}{2}-1}_{2,1})}
	+\|(\tilde{P},\tilde{\mathbf{u}},\tilde{c})\|^h_{\tilde{L}_t^\infty(\dot{B}_{2,1}^{\frac{d}{2}+1})}.
\end{align*}
The space of initial data is
\begin{align*}
	E_0\triangleq\big\{(\tilde{P}_0, \tilde{\mathbf{u}}_0, \tilde{c}_0) : \tilde{P}_0^l\in \dot{B}^{\frac{d}{p}-1}_{p,1}, \tilde{\mathbf{u}}_0^l\in \dot{B}^{\frac{d}{p}}_{p,1}, \tilde{c}_0^l\in \dot{B}^{\frac{d}{2}-1}_{2,1}, (\tilde{P}_0^h,\tilde{\mathbf{u}}_0^h, \tilde{c}_0^h)\in \dot{B}^{\frac{d}{2}+1}_{2,1}\big\}
\end{align*}
and associated norm is
$\|(\tilde{P}_0, \tilde{\mathbf{u}}_0, \tilde{c}_0)\|_{E_0}\triangleq\mathcal{X}_{p,0}$, which has been defined in Theorem \ref{1 1.1}.

\begin{theorem}(Local well-posedness)
	Let $p$ be in Theorem \ref{1 1.1} and assume $(\tilde{P}_0,\tilde{\mathbf{u}}_0,\tilde{c}_0)\in E_0$. Then, there exists a time $T_*>0$ such that the System \eqref{1.5} connected to the initial data $(\tilde{P}_0,\tilde{\mathbf{u}}_0,\tilde{c}_0)$ has a unique strong solution $(\tilde{P},\tilde{\mathbf{u}},\tilde{c})\in E(T_*)$.
\end{theorem}

\begin{proof}
	We divide the process of proof into four steps.
	\begin{enumerate}[Step 1:]
		\item Construction of the approximate sequence.\\
		We define the regularized initial data as
		\begin{equation}
			(\tilde{P}_0^k, \tilde{\mathbf{u}}_0^k, \tilde{c}_0^k)\triangleq\chi(L_kx)\sum_{|j|\le k}\dot{\Delta}_j(\tilde{P}_0, \tilde{\mathbf{u}}_0, \tilde{c}_0),
		\end{equation}
		where $\chi(x)\in\mathcal{S}(\mathbb{R}^d)$ satisfies $\chi(0)=1$ and Supp $\mathcal{F}(\chi)(\xi)\subset B(0,1)$, $L_k>0$ is a constant chosen later and $\lim_{k\to \infty}L_k=0$.\\
		Now we claim that $(\tilde{P}_0^k, \tilde{\mathbf{u}}_0^k, \tilde{c}_0^k)$ belongs to $H^{s_0}(\mathbb{R}^d)$ for fixed $k\ge0$ and $s_0>\frac{d}{2}+1$. According to Bernstein inequality and the embedding $\dot{B}^{\frac{d}{p}}_{p,1}\hookrightarrow L^\infty$ that change the space
		$$\|(\tilde{P}_0^k, \tilde{\mathbf{u}}_0^k, \tilde{c}_0^k)\|_{H^{s_0}(\mathbb{R}^d)}\lesssim2^{ks_0}\|(\tilde{P}_0^k, \tilde{\mathbf{u}}_0^k, \tilde{c}_0^k)\|_{L^2}\lesssim2^{ks_0}\|\chi(L_k\cdot)\|_{L^2}\|(\tilde{P}_0, \tilde{\mathbf{u}}_0, \tilde{c}_0)\|_{L^\infty}<\infty.$$
		Besides, we want to prove that there exists a suitably large integer $k_0$ such that for all $k\ge k_0$, $(\tilde{P}_0^k, \tilde{\mathbf{u}}_0^k, \tilde{c}_0^k)$ has the uniform bound as
		\begin{align}
			\|(\tilde{P}_0^k, \tilde{\mathbf{u}}_0^k, \tilde{c}_0^k)\|_{E_0}\lesssim\|(\tilde{P}_0, \tilde{\mathbf{u}}_0, \tilde{c}_0)\|_{E_0}.\label{4.2}
		\end{align}
		At the beginning, recalling [Proposition 2.18] in \cite{Bahouri}, there is a law that
$$
\|\chi(\frac{\cdot}{k})\|_{\dot{B}^{\frac{d}{p}}_{p,1}}\sim\|\chi\|_{\dot{B}^{\frac{d}{p}}_{p,1}}\lesssim1.
$$
Thus, Lemma \ref{2 2.7} and Bernstein inequality give the low frequency estimate as
		\begin{align*}
			\|\tilde{P}_0^k\|^l_{\dot{B}^{\frac{d}{p}-1}_{p,1}}\lesssim&\|\chi(L_kx)\|_{\dot{B}^{\frac{d}{p}}_{p,1}}\|\sum_{|j|\le k}\dot{\Delta}_j\tilde{P}_0\|_{\dot{B}^{\frac{d}{p}-1}_{p,1}}\lesssim\|\chi(\cdot)\|_{{\dot{B}^{\frac{d}{p}}_{p,1}}}\|\tilde{P}_0\|_{{\dot{B}^{\frac{d}{p}-1}_{p,1}}}\\
			&\lesssim\|\tilde{P}_0\|^l_{\dot{B}^{\frac{d}{p}-1}_{p,1}}+\|\tilde{P}_0\|^h_{\dot{B}^{\frac{d}{p}+1}_{p,1}}\lesssim\|\tilde{P}_0\|^l_{\dot{B}^{\frac{d}{p}-1}_{p,1}}+\|\tilde{P}_0\|^h_{\dot{B}^{\frac{d}{2}+1}_{2,1}},
		\end{align*}
		and the high frequency estimate as
		\begin{align*}
			\|\tilde{P}_0^k\|^h_{\dot{B}^{\frac{d}{2}+1}_{2,1}}\lesssim&\|\chi(L_kx)\sum_{|j|\le k}\dot{\Delta}_j\tilde{P}_0\|^h_{\dot{B}^{\frac{d}{2}+1}_{2,1}}\\
			\lesssim&\|\chi\|_{\dot{B}^{\frac{d}{p}}_{p,1}}\|\tilde{P}_0\|^h_{\dot{B}^{\frac{d}{2}+1}_{2,1}}+\|\tilde{P}_0\|_{\dot{B}^{\frac{d}{p}}_{p,1}}\|\chi\|^h_{\dot{B}^{\frac{d}{2}+1}_{2,1}}\\
			&+\|\tilde{P}_0\|^l_{\dot{B}^{\frac{d}{p}}_{p,1}}\|\chi\|^l_{\dot{B}^{\frac{d}{p}}_{p,1}}+\|\tilde{P}_0\|_{\dot{B}^{\frac{d}{p}}_{p,1}}\|\chi\|^l_{\dot{B}^{\frac{d}{p}+1}_{p,1}}\\
			\lesssim&(\|\chi\|_{\dot{B}^{\frac{d}{p}}_{p,1}}+\|\chi\|_{\dot{B}^{\frac{d}{2}+1}_{2,1}})(\|\tilde{P}_0\|^l_{\dot{B}^{\frac{d}{p}}_{p,1}}+\|\tilde{P}_0\|^h_{\dot{B}^{\frac{d}{2}+1}_{2,1}})\\
			\lesssim&\|\tilde{P}_0\|^l_{\dot{B}^{\frac{d}{p}}_{p,1}}+\|\tilde{P}_0\|^h_{\dot{B}^{\frac{d}{2}+1}_{2,1}}.
		\end{align*}
		Then, using the same method to deal with the $( \tilde{\mathbf{u}}_0^k, \tilde{c}_0^k)$, we can get the similar results,  and then (\ref{4.2}) is true.
        And now, we are going to explain that $(\tilde{P}_0^k, \tilde{\mathbf{u}}_0^k, \tilde{c}_0^k)$ has a limit
		\begin{align}
			\lim_{k\to \infty}\|\tilde{P}_0^k-\tilde{P}_0\|^l_{\dot{B}^{\frac{d}{p}-1}_{p,1}}=0,
		\end{align}
and $(\tilde{\mathbf{u}}_0^k, \tilde{c}_0^k)$ have their corresponding limits. In fact, according to [Lemma 4.2] in \cite{Abiidi}, we decompose
		\begin{align}
			\tilde{P}_0^k-\tilde{P}_0=(\chi(L_kx)-1)\sum_{|j|\le k}\dot{\Delta}_j\tilde{P}_0-\sum_{|j|\ge k+1}\dot{\Delta}_j\tilde{P}_0.\label{4.4}
		\end{align}
		
		For any $\eta_1>0$, we may find $k_0^*=k_0^*(\eta_1)$ such that for all $k>k_0^*$,
		$$
\|\sum_{|j|\ge k+1}\dot{\Delta}_j\tilde{P}_0\|^l_{\dot{B}^{\frac{d}{p}-1}_{p,1}}\lesssim\sum_{j\le -k, j\le J_\epsilon}2^{(\frac{d}{p}-1)j}\|\dot{\Delta}_j\tilde{P}_0\|_{L^p}<\eta_1.
$$
		On the other side, since Supp $\mathcal{F}(\chi(L_k\cdot))\subset B(0,L^k)$ and
		$$
{\rm Supp}\ \mathcal{F}(\sum_{|j|\le k}\dot{\Delta}_j\tilde{P}_0)\subset\big\{\xi\in \mathbb{R^d}:\frac{3}{4}2^{-k}\le |\xi|\le \frac{8}{3}2^k\big\},
$$
		we get $\mathcal{F}(\tilde{P}_0^k)\subset\{\xi \in \mathbb{R}^d:\frac{3}{8}2^{-k}\le |\xi|\le \frac{11}{3}2^k\}$ for $L_k\le \frac{3}{8}2^{-k}$, so $\dot{\Delta}_{j'}\tilde{P}_0^k=0$ if $|j'|\ge k+3$. Hence, the first term in the right-hand side of (\ref{4.4}) can be estimated as
		\begin{align*}
			\|(\chi(L_kx)-1)\sum_{|j|\le k}\dot{\Delta}_j\tilde{P}_0\|^l_{\dot{B}^{\frac{d}{p}-1}_{p,1}}\lesssim&2^k\|(\chi(L_k\cdot)-1)\sum_{|j|\le k}\dot{\Delta}_j\tilde{P}_0\|^l_{\dot{B}^{\frac{d}{p}}_{p,1}}\\
			\lesssim&2^k\|\chi(L_k\cdot)-1\|_{\dot{B}^{\frac{d}{p}}_{p,1}}\|\tilde{P}_0\|_{\dot{B}^{\frac{d}{p}}_{p,1}}\to 0\quad{\rm as}\quad k\to\infty,
		\end{align*}
		where we choose a suitable constant $L_k$.\\
		Thanks to the classical local well-posedness theorem for hyperbolic systems (\cite{Dafermos}), for fixed $k\geq k_0,\ s_0>\frac{d}{2}+1$, there exists a time $T_k>0$ such that the Cauchy problem of the system (\ref{1.5}) associated with the initial datum $(\tilde{P}^k_0,\tilde{\mathbf{u}}^k_0,\tilde{c}^k_0)$ has a unique solution $(\tilde{P}^k,\tilde{\mathbf{u}}^k,\tilde{c}^k)\in \mathcal{C}([0,T_k];H^{s_0}(\mathbb{R}^d)).$
		
		\item Uniform estimates.\\
		According to the similar calculation as in Section 2, for all $k\ge k_0$ and $0<t<T_k$, we can get that
		\begin{align*}
			&\|\tilde{P}^k\|^l_{\tilde{L}_t^{\infty}(\dot{B}_{p,1}^{\frac{d}{p}-1})}
			+\|\tilde{\mathbf{u}}^k\|^l_{\tilde{L}_t^{\infty}(\dot{B}_{p,1}^{\frac{d}{p}})}
			+\|\tilde{c}^k\|^l_{\tilde{L}_t^\infty(\dot{B}_{2,1}^{\frac{d}{2}-1})}\notag
			+\|(\tilde{P}^k,\tilde{\mathbf{u}}^k,\tilde{c}^k)\|^h_{\tilde{L}_t^\infty(\dot{B}_{2,1}^{\frac{d}{2}+1})}\\
			\lesssim& \mathcal{X}_{p,0}+\mathcal{X}^2_p(t)+\mathcal{X}^3_p(t)+\mathcal{X}^4_p(t).
		\end{align*}
		Additionally, some terms in $\mathcal{X}_p(t)$ can be treated as
		\begin{align}
			&\|\tilde{P}^k\|^l_{\tilde{L}_t^{1}(\dot{B}_{p,1}^{\frac{d}{p}+1})}
			+\|\tilde{\mathbf{u}}^k\|^l_{\tilde{L}_t^{1}(\dot{B}_{p,1}^{\frac{d}{p}})}
			+\|(\tilde{P}^k,\tilde{\mathbf{u}}^k)\|^h_{\tilde{L}_t^1(\dot{B}_{2,1}^{\frac{d}{2}+1})}\notag\\
			\lesssim&T_k((\|\tilde{P}^k\|^l_{\tilde{L}_t^{\infty}(\dot{B}_{p,1}^{\frac{d}{p}-1})}
			+\|\tilde{\mathbf{u}}^k\|^l_{\tilde{L}_t^{\infty}(\dot{B}_{p,1}^{\frac{d}{p}})}
			+\|(\tilde{P}^k,\tilde{\mathbf{u}}^k)\|^h_{\tilde{L}_t^\infty(\dot{B}_{2,1}^{\frac{d}{2}+1})}),
		\end{align}
		so we can acquire
		\begin{align}
			\|(\tilde{P}^k,\tilde{\mathbf{u}}^k,\tilde{c}^k)\|_{E(T_k)}\lesssim&\|(\tilde{P}_0, \tilde{\mathbf{u}}_0, \tilde{c}_0)\|_{E_0}+T_k\|(\tilde{P}^k,\tilde{\mathbf{u}}^k,\tilde{c}^k)\|_{E(T_k)}^2\notag\\
			&+T_k\|(\tilde{P}^k,\tilde{\mathbf{u}}^k,\tilde{c}^k)\|_{E(T_k)}^3+T_k\|(\tilde{P}^k,\tilde{\mathbf{u}}^k,\tilde{c}^k)\|_{E(T_k)}^4.
		\end{align}
		Hence, there exists a constant $C_*>0$ such that
		\begin{align}
			\|(\tilde{P}^k,\tilde{\mathbf{u}}^k,\tilde{c}^k)\|_{E(T_k)}\le&C_*\|(\tilde{P}_0, \tilde{\mathbf{u}}_0, \tilde{c}_0)\|_{E_0}+C_*T_k\|(\tilde{P}^k,\tilde{\mathbf{u}}^k,\tilde{c}^k)\|_{E(T_k)}^2\notag\\
			&+C_*T_k\|(\tilde{P}^k,\tilde{\mathbf{u}}^k,\tilde{c}^k)\|_{E(T^k)}^3+C_*T_k\|(\tilde{P}^k,\tilde{\mathbf{u}}^k,\tilde{c}^k)\|_{E(T_k)}^4.
		\end{align}
		Now, we define the time
		\begin{align}
			T_*\triangleq\frac{1}{4C_*^2\|(\tilde{P}_0, \tilde{\mathbf{u}}_0, \tilde{c}_0)\|_{E_0}(1+2C_*\|(\tilde{P}_0, \tilde{\mathbf{u}}_0, \tilde{c}_0)\|_{E_0}^2+4C_*^2\|(\tilde{P}_0, \tilde{\mathbf{u}}_0, \tilde{c}_0)\|_{E_0}^2)}
		\end{align}
		and the set
		\begin{align}
			I^k\triangleq\big\{t\in [0,T_*]:\|(\tilde{P}^k,\tilde{\mathbf{u}}^k,\tilde{c}^k)\|_{E(t)}\le 2C_*\|(\tilde{P}_0, \tilde{\mathbf{u}}_0, \tilde{c}_0)\|_{E_0}\big\}.\label{4.9}
		\end{align}
		It is easy to see that $T_*<T_k$. On the one hand, owing to the time continuity of $(\tilde{P}^k,\tilde{\mathbf{u}}^k,\tilde{c}^k)$ on $[0,T_k)$, $I^k$ is a nonempty closed subset of $[0,T_k]$ for every $k\ge k_0$. On the other hand, recalling the time continuity of $(\tilde{P}^k,\tilde{\mathbf{u}}^k,\tilde{c}^k)$ on $[0,T_k)$, one can present that there exists a ball $B(t,\eta_*)$ for a enough small constant $\eta_*>0$ such that $[0,T_k]\cap B(t,\eta_*)\subset I^k$, which implies that $I^k$ is an open subset of $[0,T_k]$. So, we can acquire $I^k=[0,T_k]$ and combine with (\ref{4.9}), we know that $\|(\tilde{P}^k,\tilde{\mathbf{u}}^k,\tilde{c}^k)\|_{E(T_*)}\le 2C_*\|(\tilde{P}_0, \tilde{\mathbf{u}}_0, \tilde{c}_0)\|_{E_0}$, which is uniform with respect to $k\ge k_0$.
		\item Convergence of the approximate sequence.\\
		Since the uniform estimate has been calculated in Step 2, we can claim that there exists $(\tilde{P},\tilde{\mathbf{u}},\tilde{c})$ such that as $k\to \infty$, it holds up to a subsequence that
		\begin{align*}
			(\tilde{P}^k,\tilde{\mathbf{u}}^k,\tilde{c}^k)\stackrel{*}\to(\tilde{P},\tilde{\mathbf{u}},\tilde{c})\ \ \ \ {\rm in}\ \ \ \ L^\infty([0,T_*];L^\infty(\mathbb{R}^d)).
		\end{align*}
		In order to deal with the nonlinear terms in $\tilde{G}_1$ and $\tilde{G}_2$, we first need to prove the strong compactness of $\{\tilde{P}^k,\tilde{\mathbf{u}}^k,\tilde{c}^k\}_{k\ge k_0}$. Based on $(\ref{1.5})_1$ we get
		\begin{align}
			\|\partial_t\tilde{P}^k\|_{L^\infty(0,T;{\dot{B}^{\frac{d}{p}-1}_{p,1}})}\lesssim \|\tilde{\mathbf{u}}^k\|_{L^\infty(0,T;{\dot{B}^{\frac{d}{p}}_{p,1}})}+\|\tilde{G_1}^k\|_{L^\infty(0,T;{\dot{B}^{\frac{d}{p}-1}_{p,1}})}.
		\end{align}
		For $\tilde{G_1}^k$,  using Lemma \ref{2 2.7} to give that
		\begin{align}
			\|\tilde{G_1}^k&\|_{L^\infty(0,T;{\dot{B}^{\frac{d}{p}-1}_{p,1}})}\lesssim\|h(\tilde{P}^k,\tilde{c}^k)\nabla\cdot\tilde{\mathbf{u}}^k\|_{L^\infty(0,T;{\dot{B}^{\frac{d}{p}-1}_{p,1}})}+\|
			\tilde{\mathbf{u}}^k\cdot\nabla\tilde{P}^k\|_{L^\infty(0,T;{\dot{B}^{\frac{d}{p}-1}_{p,1}})}\notag\\
			\lesssim&\|(\tilde{P}^k,\tilde{c}^k)\|_{L^\infty(0,T;{\dot{B}^{\frac{d}{p}}_{p,1}})}\|\tilde{\mathbf{u}}^k\|_{L^\infty(0,T;{\dot{B}^{\frac{d}{p}}_{p,1}})}+\|\tilde{\mathbf{u}}^k\|_{L^\infty(0,T;{\dot{B}^{\frac{d}{p}}_{p,1}})}\|\tilde{P}^k\|_{L^\infty(0,T;{\dot{B}^{\frac{d}{p}}_{p,1}})}\notag\\
			\lesssim&\|(\tilde{P}^k,\tilde{\mathbf{u}}^k,\tilde{c}^k)\|_{E(T)}^2,\label{4.11}
		\end{align}
		 and then we eventually have
		\begin{align}
			&\|\partial_t\tilde{P}^k\|_{L^\infty(0,T;{\dot{B}^{\frac{d}{p}-1}_{p,1}})}\lesssim\|(\tilde{P}^k,\tilde{\mathbf{u}}^k,\tilde{c}^k)\|_{E(T)}+\|(\tilde{P}^k,\tilde{\mathbf{u}}^k,\tilde{c}^k)\|_{E(T)}^2\notag\\
			&\lesssim\|(\tilde{P}_0, \tilde{\mathbf{u}}_0, \tilde{c}_0)\|_{E_0}+\|(\tilde{P}_0, \tilde{\mathbf{u}}_0, \tilde{c}_0)\|_{E_0}^2.
		\end{align}
		From compact embedding ${\dot{B}^{\frac{d}{p}}_{p,1}}\hookrightarrow L^1_{loc}(\mathbb{R}^d)$, we deduce from the Aubin-Lions lemma and the Cantor diagonal argument that, as $k\to \infty$, for any bounded set $K\subset\mathbb{R}^d$, there holds
		\begin{align}
			\tilde{P}^k\to \tilde{P}\ \ \ \ {\rm in}\ \ \ \ L^1(0,T_*;L^1(K)).
		\end{align}
	The compactness of $\tilde{\mathbf{u}}^k$ and $\tilde{c}^k$ can be verified similarly, and we denote their limits by $\tilde{\mathbf{u}}$ and $\tilde{c}$ respectively.
		Thus, $(\tilde{P},\tilde{\mathbf{u}},\tilde{c})$ is indeed the  solution of the system (\ref{1.5}) in the sense of distributions.

At last, we verify the time continuity of the solution $(\tilde{P},\tilde{\mathbf{u}},\tilde{c})$. Taking advantage of $(\ref{1.5})_1$, for any $0\le t_1,t_2\le T_*$, we get
		\begin{align}\label{4.14}
			\|\tilde{P}^l(t_1)-\tilde{P}^l(t_2)\|_{\dot{B}^{\frac{d}{p}-1}_{p,1}}\lesssim&\sup_{\tau\in[0,1]}\|\partial_t\tilde{P}(t_2+\tau(t_1-t_2))|t_1-t_2|\|^l_{\dot{B}^{\frac{d}{p}-1}_{p,1}}\notag\\
			\lesssim&\|\nabla\cdot\tilde{\mathbf{u}}+\tilde{G}_1\|^l_{L^\infty(0,T_*;\dot{B}^{\frac{d}{p}-1}_{p,1})}|t_1-t_2|.
		\end{align}
		The term $\|\nabla\cdot\tilde{\mathbf{u}}+\tilde{G}_1\|^l_{L^\infty(0,T_*;\dot{B}^{\frac{d}{p}-1}_{p,1})}<\infty$ has been justified in (\ref{4.11}), so (\ref{4.14}) implies that $\tilde{P}^l\in \mathcal{C}([0,T_*];\dot{B}^{\frac{d}{p}-1}_{p,1})$. Using the similar method, we get $\tilde{\mathbf{u}}^l\in \mathcal{C}([0,T_*];\dot{B}^{\frac{d}{p}}_{p,1})$ and $\tilde{c}^l\in \mathcal{C}([0,T_*];\dot{B}^{\frac{d}{2}-1}_{2,1})$. In terms of the high frequency part, it only requires to prove $(\tilde{P},\tilde{\mathbf{u}},\tilde{c})^h\in\mathcal{C}([0,T_*];\dot{B}^{\frac{d}{2}}_{2,1})$ by using the above method.  To this end,  we introduced the decomposition $(\tilde{P},\tilde{\mathbf{u}},\tilde{c})^h=S_{N_0}(\tilde{P},\tilde{\mathbf{u}},\tilde{c})^h+(Id-S_{N_0})(\tilde{P},\tilde{\mathbf{u}},\tilde{c})^h$. On the one hand, we can acquire the low frequency $S_{N_0}(\tilde{P},\tilde{\mathbf{u}},\tilde{c})^h\in\mathcal{C}([0,T_*];\dot{B}^{\frac{d}{2}+1}_{2,1})$. On the other hand, the high frequency part can be bounded as
		\begin{align*}
			\|(Id-S_{N_0})(\tilde{P},\tilde{\mathbf{u}},\tilde{c})^h\|_{L^\infty(0,T_*;\dot{B}^{\frac{d}{2}+1}_{2,1})}\lesssim \sum_{j\ge \max\{J_\epsilon,N_0\}-1}2^{(\frac{d}{2}+1)j}\sup_{t\in[0,T_*]}\|\dot{\Delta}_j(\tilde{P},\tilde{\mathbf{u}},\tilde{c})\|_{L^2},
		\end{align*}
where the right hand side may be arbitrarily small as long as $N_0$ is chosen to be large enough. Therefore, we have $(\tilde{P},\tilde{\mathbf{u}},\tilde{c})^h\in\mathcal{C}([0,T_*];\dot{B}^{\frac{d}{2}+1}_{2,1})$.
		\item Proof of the uniqueness.\\
		For a given time $T>0$, let $(\tilde{P}_1,\tilde{\mathbf{u}}_1,\tilde{c}_1)$ and $(\tilde{P}_2,\tilde{\mathbf{u}}_2,\tilde{c}_2)$ be two solutions with the initial data $(\tilde{P}_0,\tilde{\mathbf{u}}_0,\tilde{c}_0)$. Thus define $(\delta\tilde{P},\delta\tilde{\mathbf{u}},\delta\tilde{c})=(\tilde{P}_1-\tilde{P}_2,\tilde{\mathbf{u}}_1-\tilde{\mathbf{u}}_2,\tilde{c}_1-\tilde{c}_1)$ and the system (\ref{1.5}) can be transformed as
		\begin{equation}
			\left\{\begin{array}{ll}
				(\delta\tilde{P})_t+\kappa_2\nabla\cdot\delta\tilde{\mathbf{u}}=H_1,\\[2mm]
				(\delta\tilde{\mathbf{u}})_t+\kappa_2\nabla\delta\tilde{P}+\alpha\delta\tilde{\mathbf{u}}=H_2,\\[2mm]
				(\delta\tilde{c})_t=H_3,
			\end{array}\right.
		\end{equation}
		where
		\begin{equation}
			\begin{split}
				H_1:=&\tilde{G}_1(\tilde{P}_1,\tilde{\mathbf{u}}_1,\tilde{c}_1)-{G}_1(\tilde{P}_2,\tilde{\mathbf{u}}_2,\tilde{c}_2)=-\kappa_1h(\tilde{P}_1,\tilde{c}_1)\nabla\cdot\delta \tilde{\mathbf{u}}\\
				&-\kappa_1(h(\tilde{P}_1,\tilde{c}_1)-h(\tilde{P}_1,\tilde{c}_2))\nabla\cdot \tilde{\mathbf{u}}_2-\kappa_1(h(\tilde{P}_1,\tilde{c}_2)-h(\tilde{P}_2,\tilde{c}_2))\nabla\cdot \tilde{\mathbf{u}}_2\\
				&-\kappa_1\tilde{\mathbf{u}}_1\cdot\nabla\delta \tilde{P}-\kappa_1\delta\tilde{\mathbf{u}}\cdot\nabla\tilde{P}_2,\\
				H_2:=&\tilde{G}_2(\tilde{P}_1,\tilde{\mathbf{u}}_1,\tilde{c}_1)-\tilde{G}_2(\tilde{P}_2,\tilde{\mathbf{u}}_2,\tilde{c}_2)=-\kappa_1\tilde{\mathbf{u}}_1\cdot\nabla\delta\tilde{\mathbf{u}}-\kappa_1\delta\tilde{\mathbf{u}}\cdot\nabla\tilde{\mathbf{u}}_2\\
				&-\frac{1}{\kappa_1}M(\tilde{P}_1,\tilde{c}_1)\cdot\nabla\delta\tilde{P}-\kappa_1(M(\tilde{P}_1,\tilde{c}_1)-M(\tilde{P}_1,\tilde{c}_2))\cdot\nabla \tilde{P}_2\\
				&-\kappa_1(M(\tilde{P}_1,\tilde{c}_2)-M(\tilde{P}_2,\tilde{c}_2))\cdot\nabla \tilde{P}_2,\\
				H_3:=&-\kappa_1\tilde{\mathbf{u}}_1\cdot\nabla\tilde{c}_1+\kappa_1\tilde{\mathbf{u}}_2\cdot\nabla\tilde{c}_2=-\kappa_1\tilde{\mathbf{u}}_1\cdot\nabla\delta\tilde{c}-\kappa_1\delta\tilde{\mathbf{u}}\cdot\nabla\tilde{c}_2.
			\end{split}
		\end{equation}
		At first, we shall estimate $(\delta\tilde{P},\delta\tilde{\mathbf{u}},\delta\tilde{c})$ in the space
		\begin{align}
			\Big\{(&\delta\tilde{P},\delta\tilde{\mathbf{u}},\delta\tilde{c}) : \delta\tilde{P}^l\in \mathcal{C}([0,T_*];\dot{B}^{\frac{d}{p}-1}_{p,1}), \delta\tilde{\mathbf{u}}^l\in \mathcal{C}([0,T_*];\dot{B}^{\frac{d}{p}}_{p,1}), \delta\tilde{c}^l\in \mathcal{C}([0,T_*];\dot{B}^{\frac{d}{2}-1}_{2,1}),\notag\\
			&(\delta\tilde{P}^h, \delta\tilde{\mathbf{u}}^h, \delta\tilde{c}^h)\in \mathcal{C}([0,T_*];\dot{B}^{\frac{d}{2}}_{2,1})\Big\}.\label{4.18}
		\end{align}
		Arguing similarly as in Section 2, the low frequency estimate can be given easily as
		\begin{align}
			&\|\delta\tilde{P}\|^l_{\tilde{L}^\infty_{t}(\dot{B}^{\frac{d}{p}-1}_{p,1})}+\|\delta\tilde{P}\|^l_{L^1_{t}(\dot{B}^{\frac{d}{p}+1}_{p,1})}+\|\delta\tilde{\mathbf{u}}\|^l_{\tilde{L}^\infty_{t}(\dot{B}^{\frac{d}{p}}_{p,1})}+\|\delta\tilde{\mathbf{u}}\|^l_{L^1_{t}(\dot{B}^{\frac{d}{p}}_{p,1})}+\|\delta\tilde{c}\|^l_{\tilde{L}^\infty_{t}(\dot{B}^{\frac{d}{2}-1}_{2,1})}\notag\\
			\lesssim&\|H_1\|^l_{L^1_{t}(\dot{B}^{\frac{d}{p}-1}_{p,1})}+\|H_2\|^l_{L^1_{t}(\dot{B}^{\frac{d}{p}}_{p,1})}+\|H_3\|^l_{L^1_{t}(\dot{B}^{\frac{d}{2}-1}_{2,1})}.
		\end{align}
		Most of the above terms on $H_1$, $H_2$ and $H_3$ have been estimated in (\ref{3.7}), (\ref{3.8}) and (\ref{3.15}), so we just deal with the remain terms.
		By using Lemma \ref{A.F(m)}, we get
		\begin{align}
			&\|(h(\tilde{P}_1,\tilde{c}_1)-h(\tilde{P}_1,\tilde{c}_2))\nabla\cdot \tilde{\mathbf{u}}_2\|^l_{L^1_{t}(\dot{B}^{\frac{d}{p}-1}_{p,1})}\notag\\
			\lesssim& \int_{0}^{T}\|h(\tilde{P}_1,\tilde{c}_1)-h(\tilde{P}_1,\tilde{c}_2)\|_{\dot{B}^{\frac{d}{p}}_{p,1}}\|\nabla\cdot \tilde{\mathbf{u}}_2\|_{\dot{B}^{\frac{d}{p}-1}_{p,1}}\notag\\
			\lesssim&\int_{0}^{T}\|\tilde{\mathbf{u}}_2\|_{\dot{B}^{\frac{d}{p}}_{p,1}}(1+\|(\tilde{P}_1,\tilde{c}_1,\tilde{c}_2)\|_{\dot{B}^{\frac{d}{p}}_{p,1}})
\|\delta\tilde{c}\|_{\dot{B}^{\frac{d}{p}}_{p,1}}\notag\\
			\lesssim&\int_{0}^{T}\|\tilde{\mathbf{u}}_2\|_{\dot{B}^{\frac{d}{p}}_{p,1}}(1+\|(\tilde{P}_1,\tilde{c}_1,\tilde{c}_2)\|_{\dot{B}^{\frac{d}{p}}_{p,1}})
(\|\delta\tilde{c}\|^l_{\dot{B}^{\frac{d}{2}-1}_{2,1}}+\|\delta\tilde{c}\|^h_{\dot{B}^{\frac{d}{2}}_{2,1}}).
		\end{align}
		Similarly, the terms  $(h(\tilde{P}_1,\tilde{c}_2)-h(\tilde{P}_2,\tilde{c}_2))\nabla\cdot \tilde{\mathbf{u}}_2$,  $(M(\tilde{P}_1,\tilde{c}_1)-M(\tilde{P}_1,\tilde{c}_2))\cdot\nabla \tilde{P}_2$ and  $(M(\tilde{P}_1,\tilde{c}_2)-M(\tilde{P}_2,\tilde{c}_2))\cdot\nabla \tilde{P}_2$ can be bounded in the same way, we omit here.
		Thus, we can draw a conclusion that
		\begin{align}
			&\|\delta\tilde{P}\|^l_{\tilde{L}^\infty_{t}(\dot{B}^{\frac{d}{p}-1}_{p,1})}+\|\delta\tilde{P}\|^l_{L^1_{t}(\dot{B}^{\frac{d}{p}+1}_{p,1})}+\|\delta\tilde{\mathbf{u}}\|^l_{\tilde{L}^\infty_{t}(\dot{B}^{\frac{d}{p}}_{p,1})}+\|\delta\tilde{\mathbf{u}}\|^l_{L^1_{t}(\dot{B}^{\frac{d}{p}}_{p,1})}+\|\delta\tilde{c}\|^l_{\tilde{L}^\infty_{t}(\dot{B}^{\frac{d}{2}-1}_{2,1})}\notag\\
			\lesssim&\int_{0}^{T}(\|(\tilde{P}_1,\tilde{P}_2)\|_{\dot{B}^{\frac{d}{p}}_{p,1}}+\|(\tilde{\mathbf{u}}_1,\tilde{\mathbf{u}}_2)\|_{\dot{B}^{\frac{d}{p}}_{p,1}}+\|(\tilde{c}_1,\tilde{c}_2)\|_{\dot{B}^{\frac{d}{2}}_{2,1}})(1+\|\tilde{P}_2\|_{\dot{B}^{\frac{d}{p}}_{p,1}}+\|\tilde{\mathbf{u}}_2\|_{\dot{B}^{\frac{d}{p}}_{p,1}})\notag\\
			&\ \ \ \ \cdot(\|\delta\tilde{P}\|^l_{\dot{B}^{\frac{d}{p}-1}_{p,1}}+\|\delta\tilde{\mathbf{u}}\|^l_{\dot{B}^{\frac{d}{p}}_{p,1}}+\|\delta\tilde{c}\|^l_{\dot{B}^{\frac{d}{2}-1}_{2,1}}+\|(\delta\tilde{P},\delta\tilde{\mathbf{u}},\delta\tilde{c})\|^h_{\dot{B}^{\frac{d}{2}}_{2,1}}).\label{4.21}
		\end{align}
		Next, we turn to deal with the high frequency part. It is similar to Subsection 2.2, employing the same method and transfer the system (\ref{1.5}) as
		\begin{equation}
			\left\{\begin{array}{ll}
				(\dot{\Delta}_j\delta\tilde{P})_t+\kappa_2{\rm div}\dot{\Delta}_j\delta\tilde{\mathbf{u}}+\kappa_1\dot{S}_{j-1}h(\tilde{P}_1,\tilde{c}_1){\rm div}\dot{\Delta}_j\delta\tilde{\mathbf{u}}+\kappa_1\dot{S}_{j-1}\tilde{\mathbf{u}}_1\nabla\dot{\Delta}_j\delta\tilde{P}\\[2mm]
				\quad\quad\quad\quad=\kappa_1\dot{S}_{j-1}h(\tilde{P}_1,\tilde{c}_1){\rm div}\dot{\Delta}_j\delta\tilde{\mathbf{u}}-\kappa_1\dot{\Delta}_j(h(\tilde{P}_1,\tilde{c}_1){\rm div}\delta \tilde{\mathbf{u}})\\[2mm]
				\quad\quad\quad\quad\quad+\kappa_1\dot{S}_{j-1}\tilde{\mathbf{u}}_1\cdot\nabla\dot{\Delta}_j\delta\tilde{P}-\kappa_1\dot{\Delta}_j(\tilde{\mathbf{u}}_1\cdot\nabla\delta \tilde{P})\\[2mm]
				\quad\quad\quad\quad\quad-\kappa_1\dot{\Delta}_j((h(\tilde{P}_1,\tilde{c}_1)-h(\tilde{P}_1,\tilde{c}_2))\nabla\cdot \tilde{\mathbf{u}}_2)\\[2mm]
				\quad\quad\quad\quad\quad-\kappa_1\dot{\Delta}_j((h(\tilde{P}_1,\tilde{c}_2)-h(\tilde{P}_2,\tilde{c}_2))\nabla\cdot \tilde{\mathbf{u}}_2)\\[2mm]
				\quad\quad\quad\quad\quad-\kappa_1\dot{\Delta}_j(\delta\tilde{\mathbf{u}}\cdot\nabla\tilde{P}_2):=I_1+I_2+R_1,\\[2mm]
				(\dot{\Delta}_j\delta\tilde{\mathbf{u}})_t+\kappa_2\nabla\dot{\Delta}_j\delta\tilde{P}+\alpha\dot{\Delta}_j\delta\tilde{\mathbf{u}}+k_1\dot{S}_{j-1}\tilde{\mathbf{u}}_1\cdot\nabla\dot{\Delta}_j\delta\tilde{\mathbf{u}}+\frac{1}{k_1}\dot{S}_{j-1}M(\tilde{P}_1,\tilde{c}_1)\cdot\dot{\Delta}_j\delta\tilde{P}\\[2mm]
				\quad\quad\quad\quad=k_1\dot{S}_{j-1}\tilde{\mathbf{u}}_1\cdot\nabla\dot{\Delta}_j\delta\tilde{\mathbf{u}}-\dot{\Delta}_j(k_1\tilde{\mathbf{u}}_1\cdot\nabla\delta\tilde{\mathbf{u}})\\[2mm]
				\quad\quad\quad\quad\quad+\frac{1}{k_1}\dot{S}_{j-1}M(\tilde{P}_1,\tilde{c}_1)\cdot\nabla\dot{\Delta}_j\delta\tilde{P}-\dot{\Delta}_j(\frac{1}{k_1}M(\tilde{P}_1,\tilde{c}_1)\cdot\nabla\delta\tilde{P})\\[2mm]
				\quad\quad\quad\quad\quad-\dot{\Delta}_j(k_1\delta\tilde{\mathbf{u}}\cdot\nabla\tilde{\mathbf{u}}_2)\\[2mm]
				\quad\quad\quad\quad\quad-\dot{\Delta}_j(k_1(M(\tilde{P}_1,\tilde{c}_1)-M(\tilde{P}_1,\tilde{c}_2))\cdot\nabla\tilde{P}_2)\\[2mm]
				\quad\quad\quad\quad\quad-\dot{\Delta}_j(k_1(M(\tilde{P}_1,\tilde{c}_2)-M(\tilde{P}_2,\tilde{c}_2))\cdot\nabla\tilde{P}_2):=I_3+I_4+R_2,\\[2mm]
				(\dot{\Delta}_j\delta\tilde{c})_t+k_1\dot{S}_{j-1}\tilde{\mathbf{u}}_1\cdot\nabla\dot{\Delta}_j\delta\tilde{c}=k_1\dot{S}_{j-1}\tilde{\mathbf{u}}_1\cdot\nabla\dot{\Delta}_j\delta\tilde{c}-\dot{\Delta}_j(k_1\tilde{\mathbf{u}}_1\cdot\nabla\delta\tilde{c})\\[2mm]
				\quad\quad\quad\quad\quad-\dot{\Delta}_j(k_1\delta\tilde{\mathbf{u}}\cdot\nabla\tilde{c}_2):=I_5+R_3,
			\end{array}\right.
		\end{equation}
		we can get a result like (\ref{3.30}) that
		\begin{align}
			&\displaystyle\|(\delta\tilde{P},\delta\tilde{\mathbf{u}},\delta\tilde{c})\|^h_{\tilde{L}_t^\infty(\dot{B}_{2,1}^{\frac{d}{2}})}+\|(\delta\tilde{P},\delta\tilde{\mathbf{u}})\|^h_{\tilde{L}_t^1(\dot{B}_{2,1}^{\frac{d}{2}})}\notag\\[2mm]
			\displaystyle\lesssim&\Big(\|\partial_t(\kappa_2\!+\!\frac{1}{\kappa_1}\dot{S}_{j-1}M(\tilde{P}_1,\tilde{c}_1)\|_{\tilde{L}_t^\infty(L^\infty)}
			+\|\partial_t(\kappa_2+\kappa_1\dot{S}_{j-1}h(\tilde{P}_1,\tilde{c}_1))\|_{\tilde{L}_t^\infty(L^\infty)}\notag\\[2mm]
			&\displaystyle+\|{\rm div}\big(\kappa_1(\kappa_2\!+\!\frac{1}{\kappa_1}\dot{S}_{j-1}(M(\tilde{P}_1,\tilde{c}_1)))\dot{S}_{j-1}\tilde{\mathbf{u}}_1\big)\|_{\tilde{L}_t^\infty(L^\infty)}\notag\notag\\[2mm]
			&+\|{\rm div}\big(\kappa_1(\kappa_2\!+\!\kappa_1\dot{S}_{j-1}h(\tilde{P}_1,\tilde{c}_1))\dot{S}_{j-1}\tilde{\mathbf{u}}_1\big)\|_{\tilde{L}_t^\infty(L^\infty)}\notag\\[2mm]
			&\displaystyle+\|\nabla\big((\kappa_2+\kappa_1\dot{S}_{j-1}h(\tilde{P}_1,\tilde{c}_1))(\kappa_2+\frac{1}{\kappa_1}\dot{S}_{j-1}M(\tilde{P}_1,\tilde{c}_1)\big)\|_{\tilde{L}_t^\infty(L^\infty)}\Big)\|
			(\delta\tilde{\mathbf{u}},\delta\tilde{P})\|^h_{\tilde{L}_t^1(\dot{B}_{2,1}^{\frac{d}{2}})}\notag\\[2mm]
			&\displaystyle+\|{\rm div}\dot{S}_{j-1}\tilde{\mathbf{u}}_1\|_{\tilde{L}_t^1(L^\infty)}\|\delta\tilde{c}\|^h_{\tilde{L}_t^\infty(\dot{B}_{2,1}^{\frac{d}{2}})}\notag\\[2mm]
			&\displaystyle+\int_0^T\sum_{j\geq j_0}2^{j\frac{d}{2}}\big(\|I_1\|_{L^2}+\|I_2\|_{L^2}+\|I_3\|_{L^2}+\|I_4\|_{L^2}+\|I_5\|_{L^2}\notag\\[2mm]
			&+\|R_1\|_{L^2}+\|R_2\|_{L^2}+\|R_3\|_{L^2}\big).\label{4.23}
		\end{align}
		In Subsection 2.2, we have already gotten the results as
		\begin{align}
			&\|\partial_t(\kappa_2\!+\!\frac{1}{\kappa_1}\dot{S}_{j-1}M(\tilde{P}_1,\tilde{c}_1)\|_{\tilde{L}_t^\infty(L^\infty)}
			+\|\partial_t(\kappa_2+\kappa_1\dot{S}_{j-1}h(\tilde{P}_1,\tilde{c}_1))\|_{\tilde{L}_t^\infty(L^\infty)}\notag\\[3mm]
			&\displaystyle+\|{\rm div}\big(\kappa_1(\kappa_2\!+\!\frac{1}{\kappa_1}\dot{S}_{j-1}(M(\tilde{P}_1,\tilde{c}_1)))\dot{S}_{j-1}\tilde{\mathbf{u}}_1\big)\|_{\tilde{L}_t^\infty(L^\infty)}\notag\\[3mm]
			&+\|{\rm div}\big(\kappa_1(\kappa_2\!+\!\kappa_1\dot{S}_{j-1}h(\tilde{P}_1,\tilde{c}_1))\dot{S}_{j-1}\tilde{\mathbf{u}}_1\big)\|_{\tilde{L}_t^\infty(L^\infty)}\notag\\[3mm]
			&\displaystyle+\|\nabla\big((\kappa_2+\kappa_1\dot{S}_{j-1}h(\tilde{P}_1,\tilde{c}_1))(\kappa_2+\frac{1}{\kappa_1}\dot{S}_{j-1}M(\tilde{P}_1,\tilde{c}_1)\big)\|_{\tilde{L}_t^\infty(L^\infty)}\notag\\[3mm]
			&+\|{\rm div}\dot{S}_{j-1}\tilde{\mathbf{u}}_1\|_{\tilde{L}_t^1(L^\infty)}<\infty.
		\end{align}
		It just remains the terms $I_i$ $(i=1,\cdots,5)$ and $R_j$ ($j=1,2,3$) to be estimated. Here we only calculate three terms in detail, since the other terms can be estimated similarly. By picking $\sigma_1=1$, $\sigma_2=\frac{d}{p}+1$ and $\sigma_3=\frac{d}{p}-1$ in Lemma \ref{2 2.3},  we get
		\begin{align}
			&\int_{0}^{T}\sum_{j\ge j_0}2^{j\frac{d}{2}}\|I_1\|_{L^2}\notag\\
			\lesssim&\int_{0}^{T}\|{\rm div}\delta\tilde{\mathbf{u}}\|_{\dot{B}^{\frac{d}{p}}_{p,1}}\|h(\tilde{P}_1,\tilde{c}_1)\|^h_{\dot{B}^{\frac{d}{2}+1}_{2,1}}+\|{\rm div}\delta\tilde{\mathbf{u}}\|^l_{\dot{B}^{\frac{d}{p}}_{p,1}}\|\nabla h(\tilde{P}_1,\tilde{c}_1)\|^l_{\dot{B}^{\frac{d}{p}-\frac{d}{p^*}}_{p,1}}\notag\\
			&\ \ \ \ +\|{\rm div}\delta\tilde{\mathbf{u}}\|^l_{\dot{B}^{\frac{d}{p}-\frac{d}{p^*}}_{p,1}}\|h(\tilde{P}_1,\tilde{c}_1)\|^l_{\dot{B}^{\frac{d}{p}+1}_{p,1}}+\|\nabla h(\tilde{P}_1,\tilde{c}_1)\|_{\dot{B}^{\frac{d}{p}}_{p,1}}\|{\rm div}\delta\tilde{\mathbf{u}}\|^h_{\dot{B}^{\frac{d}{2}-1}_{2,1}}\notag\\
			\lesssim&\int_{0}^{T}(\|\tilde{P}_1\|^l_{\dot{B}^{\frac{d}{p}-1}_{p,1}}+\|\tilde{c}_1\|^l_{\dot{B}^{\frac{d}{2}-1}_{2,1}}
+\|(\tilde{P}_1,\tilde{c}_1)\|^h_{\dot{B}^{\frac{d}{2}+1}_{2,1}})(\|\delta\tilde{\mathbf{u}}\|^l_{\dot{B}^{\frac{d}{p}}_{p,1}}
+\|\delta\tilde{\mathbf{u}}\|^h_{\dot{B}^{\frac{d}{2}}_{2,1}}),
		\end{align}
		and $I_i$ ($i=2,3,4,5)$ can be estimated similarly. Then we arrive at the result that
		\begin{align}
			&\displaystyle\int_0^T\sum_{j\geq j_0}2^{j\frac{d}{2}}\big(\|I_1\|_{L^2}+\|I_2\|_{L^2}+\|I_3\|_{L^2}+\|I_4\|_{L^2}+\|I_5\|_{L^2})\notag\\
			\lesssim&\int_{0}^{T}(\|\tilde{P}_1\|^l_{\dot{B}^{\frac{d}{p}-1}_{p,1}}+\|\tilde{\mathbf{u}}_1\|^l_{\dot{B}^{\frac{d}{p}}_{p,1}}
+\|\tilde{c}_1\|^l_{\dot{B}^{\frac{d}{2}-1}_{2,1}}+\|(\tilde{P}_1,\tilde{\mathbf{u}}_1,\tilde{c}_1)\|^h_{\dot{B}^{\frac{d}{2}+1}_{2,1}})\notag\\
			&\ \ \ \ \cdot(\|\delta \tilde{P}\|^l_{\dot{B}^{\frac{d}{p}-1}_{p,1}}+\|\delta\tilde{\mathbf{u}}\|^l_{\dot{B}^{\frac{d}{p}}_{p,1}}
+\|\delta\tilde{c}\|^l_{\dot{B}^{\frac{d}{2}-1}_{2,1}}+\|(\delta\tilde{P},\delta\tilde{\mathbf{u}},\delta\tilde{c})\|^h_{\dot{B}^{\frac{d}{2}}_{2,1}}).\label{4.26}
		\end{align}
		Besides, we analysis some typical terms from $R_{1,2,3}$. Using Lemma \ref{A h} gives that
		\begin{align}
			&\|(h(\tilde{P}_1,\tilde{c}_1)-h(\tilde{P}_1,\tilde{c}_2)){\rm div}\tilde{\mathbf{u}}_2\|^h_{\dot{B}^{\frac{d}{2}}_{2,1}}\notag\\
			\lesssim&\|h(\tilde{P}_1,\tilde{c}_1)-h(\tilde{P}_1,\tilde{c}_2)\|_{\dot{B}^{\frac{d}{p}}_{p,1}}\|{\rm div}\tilde{\mathbf{u}}_2\|^h_{\dot{B}^{\frac{d}{2}}_{2,1}}+\|{\rm div}\tilde{\mathbf{u}}_2\|_{\dot{B}^{\frac{d}{p}}_{p,1}}\|h(\tilde{P}_1,\tilde{c}_1)-h(\tilde{P}_1,\tilde{c}_2)\|^h_{\dot{B}^{\frac{d}{2}}_{2,1}}\notag\\
			&+\|{\rm div}\tilde{\mathbf{u}}_2\|^l_{\dot{B}^{\frac{d}{p}}_{p,1}}\|h(\tilde{P}_1,\tilde{c}_1)-h(\tilde{P}_1,\tilde{c}_2)\|^l_{\dot{B}^{\frac{d}{p}}_{p,1}}
			\!+\!\|{\rm div}\tilde{\mathbf{u}}_2\|_{\dot{B}^{\frac{d}{p}}_{p,1}}\|h(\tilde{P}_1,\tilde{c}_1)-h(\tilde{P}_1,\tilde{c}_2)\|^l_{\dot{B}^{\frac{d}{p}}_{p,1}}.\label{4.27}
		\end{align}
		Thanks to Lemma \ref{A.F(m)} again, we acquire
		\begin{align}
			\|h(\tilde{P}_1,\tilde{c}_1)-h(\tilde{P}_1,\tilde{c}_2)\|_{\dot{B}^{\frac{d}{p}}_{p,1}}\lesssim(1+\|(\tilde{P}_1, \tilde{c}_1,\tilde{c}_2)\|_{\dot{B}^{\frac{d}{p}}_{p,1}})\|\delta\tilde{c}\|_{\dot{B}^{\frac{d}{p}}_{p,1}}.\label{4.28}
		\end{align}
		Besides, in order to bound the nonlinear term in the high frequency region, we introduce mean-value theorem for the derivatives and utilize the Lemma \ref{A h}, Lemma \ref{A.F(m)} again as well as Lemma \ref{A.13} to get
		\begin{align}
			&\|h(\tilde{P}_1,\tilde{c}_1)-h(\tilde{P}_1,\tilde{c}_2)\|^h_{\dot{B}^{\frac{d}{2}}_{2,1}}\notag\\
			=&\|(\partial_{\tilde{c}}h(\tilde{P}_1,\tilde{c}_2+\tau\delta\tilde{c})-\partial_{\tilde{c}}h(0,0)+\partial_{\tilde{c}}h(0,0))\delta\tilde{c}\|^h_{\dot{B}^{\frac{d}{2}}_{2,1}}\notag\\
\lesssim&\|(\partial_{\tilde{c}}h(\tilde{P}_1,\tilde{c}_2+\tau\delta\tilde{c})-\partial_{\tilde{c}}h(0,0))\delta\tilde{c}\|^h_{\dot{B}^{\frac{d}{2}}_{2,1}}+\|\partial_{\tilde{c}}h(0,0)\delta\tilde{c}\|^h_{\dot{B}^{\frac{d}{2}}_{2,1}}\notag\\
			\lesssim&\|\partial_{\tilde{c}}h(\tilde{P}_1,\tilde{c}_2+\tau\delta\tilde{c})-\partial_{\tilde{c}}h(0,0)\|_{\dot{B}^{\frac{d}{p}}_{p,1}}\|\delta\tilde{c}\|^h_{\dot{B}^{\frac{d}{2}}_{2,1}}\notag\\
&+\|\delta\tilde{c}\|_{\dot{B}^{\frac{d}{p}}_{p,1}}\|\partial_{\tilde{c}}h(\tilde{P}_1,\tilde{c}_2+\tau\delta\tilde{c})-\partial_{\tilde{c}}h(0,0)\|^h_{\dot{B}^{\frac{d}{2}}_{2,1}}\notag\\ &+\|\delta\tilde{c}\|^l_{\dot{B}^{\frac{d}{p}}_{p,1}}\|\partial_{\tilde{c}}h(\tilde{P}_1,\tilde{c}_2+\tau\delta\tilde{c})-\partial_{\tilde{c}}h(0,0)\|^l_{\dot{B}^{\frac{d}{p}}_{p,1}}\notag\\ &+\|\delta\tilde{c}\|_{\dot{B}^{\frac{d}{p}}_{p,1}}\|\partial_{\tilde{c}}h(\tilde{P}_1,\tilde{c}_2+\tau\delta\tilde{c})-\partial_{\tilde{c}}h(0,0)\|^l_{\dot{B}^{\frac{d}{p}}_{p,1}}+\|\delta\tilde{c}\|^h_{\dot{B}^{\frac{d}{2}}_{2,1}}\notag\\	
\lesssim&(1+\|\tilde{P}_1\|^l_{\dot{B}^{\frac{d}{p}-1}_{p,1}}+\|(\tilde{c}_1,\tilde{c}_2)\|^l_{\dot{B}^{\frac{d}{2}-1}_{2,1}}+\|(\tilde{P}_1,\tilde{c}_1,\tilde{c}_2)\|^h_{\dot{B}^{\frac{d}{2}+1}_{2,1}})(\|\delta\tilde{c}\|^l_{\dot{B}^{\frac{d}{p}-1}_{p,1}}+\|\delta\tilde{c}\|^h_{\dot{B}^{\frac{d}{2}}_{2,1}}),\label{4.29}
		\end{align}
		where $\tau\in(0,1)$. Hence, plugging (\ref{4.28}) and (\ref{4.29}) into (\ref{4.27}), we acquire
		\begin{align}
			&\|(h(\tilde{P}_1,\tilde{c}_1)-h(\tilde{P}_1,\tilde{c}_2)){\rm div}\tilde{\mathbf{u}}_2\|^h_{\dot{B}^{\frac{d}{2}}_{2,1}}\notag\\
			\lesssim&(1+\|\tilde{P}_1\|^l_{\dot{B}^{\frac{d}{p}-1}_{p,1}}+\|(\tilde{c}_1,\tilde{c}_2)\|^l_{\dot{B}^{\frac{d}{2}-1}_{2,1}}+\|(\tilde{P}_1,\tilde{c}_1,\tilde{c}_2)\|^h_{\dot{B}^{\frac{d}{2}+1}_{2,1}})\notag\\
			&(\|\tilde{\mathbf{u}}_2\|^l_{\dot{B}^{\frac{d}{p}}_{p,1}}+\|\tilde{\mathbf{u}}_2\|^h_{\dot{B}^{\frac{d}{2}+1}_{2,1}})(\|\delta\tilde{c}\|^l_{\dot{B}^{\frac{d}{p}-1}_{p,1}}+\|\delta\tilde{c}\|^h_{\dot{B}^{\frac{d}{2}}_{2,1}}).
		\end{align}
		Estimating the term $\|(h(\tilde{P}_1,\tilde{c}_2)-h(\tilde{P}_2,\tilde{c}_2)){\rm div}\tilde{\mathbf{u}}_2\|^h_{\dot{B}^{\frac{d}{2}}_{2,1}}$ is similar, we omit its proof here. Another term can be bounded by using Lemma \ref{A.F(m)} that
		\begin{align}
			\|\delta \tilde{\mathbf{u}}\cdot\nabla\tilde{P}_2\|^h_{\dot{B}^{\frac{d}{2}}_{2,1}}\lesssim&
			\|\delta \tilde{\mathbf{u}}\|_{\dot{B}^{\frac{d}{p}}_{p,1}}\|\nabla\tilde{P}_2\|^h_{\dot{B}^{\frac{d}{2}}_{2,1}}+\|\nabla\tilde{P}_2\|_{\dot{B}^{\frac{d}{p}}_{p,1}}\|\delta \tilde{\mathbf{u}}\|^h_{\dot{B}^{\frac{d}{2}}_{2,1}}\notag\\
			&+\|\nabla\tilde{P}_2\|^l_{\dot{B}^{\frac{d}{p}}_{p,1}}\|\delta \tilde{\mathbf{u}}\|^l_{\dot{B}^{\frac{d}{p}}_{p,1}}
			+\|\nabla\tilde{P}_2\|_{\dot{B}^{\frac{d}{p}}_{p,1}}\|\delta \tilde{\mathbf{u}}\|^l_{\dot{B}^{\frac{d}{p}}_{p,1}}\notag\\
			\lesssim&(\|\tilde{P}_2\|^l_{\dot{B}^{\frac{d}{p}-1}_{p,1}}+\|\tilde{P}_2\|^h_{\dot{B}^{\frac{d}{p}+1}_{p,1}})(\|\delta\tilde{\mathbf{u}}\|^l_{\dot{B}^{\frac{d}{p}}_{p,1}}+\|\delta\tilde{\mathbf{u}}\|^h_{\dot{B}^{\frac{d}{2}}_{2,1}}).
		\end{align}
		So, it holds that
		\begin{align}
			&\int_0^T\sum_{j\geq j_0}2^{j\frac{d}{2}}\big(\|R_1\|_{L^2}+\|R_2\|_{L^2}+\|R_3\|_{L^2}\big)\notag\\
			\lesssim&\int_0^T(\|(\tilde{P}_1,\tilde{P}_2)\|^l_{\dot{B}^{\frac{d}{p}-1}_{p,1}}+\|\tilde{\mathbf{u}_1}\|^l_{\dot{B}^{\frac{d}{p}}_{p,1}}
+\|(\tilde{c}_1,\tilde{c}_2)\|^l_{\dot{B}^{\frac{d}{2}-1}_{2,1}}+\|(\tilde{P}_1,\tilde{P}_2,\tilde{\mathbf{u}}_1,\tilde{c}_1,\tilde{c}_2)\|^h_{\dot{B}^{\frac{d}{2}}_{2,1}})\notag\\
			&\ \ \cdot(1+\|(\tilde{P}_2,\tilde{\mathbf{u}}_2)\|^h_{\dot{B}^{\frac{d}{2}+1}_{2,1}})(\|\delta\tilde{P}\|^l_{\dot{B}^{\frac{d}{p}-1}_{p,1}}
+\|\delta\tilde{\mathbf{u}}\|^l_{\dot{B}^{\frac{d}{p}}_{p,1}}+\|\delta\tilde{c}\|^l_{\dot{B}^{\frac{d}{2}-1}_{2,1}}
+\|(\delta\tilde{P},\delta\tilde{\mathbf{u}},\delta\tilde{c})\|^h_{\dot{B}^{\frac{d}{2}}_{2,1}}).\label{4.32}
		\end{align}
		Plugging (\ref{4.26}) and (\ref{4.32}) into (\ref{4.23}), one has
		\begin{align}
			&\|(\delta\tilde{P},\delta\tilde{\mathbf{u}},\delta\tilde{c})\|^h_{\tilde{L}_t^\infty(\dot{B}_{2,1}^{\frac{d}{2}})}+\|(\delta\tilde{P},\delta\tilde{\mathbf{u}})\|^h_{\tilde{L}_t^1(\dot{B}_{2,1}^{\frac{d}{2}})}\notag\\
			\lesssim&\int_0^T\Big(1+\|(\tilde{P}_1,\tilde{P}_2)\|^l_{\dot{B}^{\frac{d}{p}-1}_{p,1}}+\|(\tilde{\mathbf{u}}_1,\tilde{\mathbf{u}}_2)\|^l_{\dot{B}^{\frac{d}{p}}_{p,1}}+\|(\tilde{c}_1,\tilde{c}_2)\|^l_{\dot{B}^{\frac{d}{2}-1}_{2,1}}\notag\\
			&\ \ \ \ \ \ \ \ +\|(\tilde{P}_1,\tilde{P}_2,\tilde{\mathbf{u}}_1,\tilde{\mathbf{u}}_2,\tilde{c}_1,\tilde{c}_2)\|^h_{\dot{B}^{\frac{d}{2}}_{2,1}}\Big)^2\notag\\
			&\ \ \ \ \ \cdot\big(\|\delta\tilde{P}\|^l_{\dot{B}^{\frac{d}{p}-1}_{p,1}}+\|\delta\tilde{\mathbf{u}}\|^l_{\dot{B}^{\frac{d}{p}}_{p,1}}+\|\delta\tilde{c}\|^l_{\dot{B}^{\frac{d}{2}-1}_{2,1}}
+\|(\delta\tilde{P},\delta\tilde{\mathbf{u}},\delta\tilde{c})\|^h_{\dot{B}^{\frac{d}{2}}_{2,1}}\big).\label{4.33}
		\end{align}
		Combining (\ref{4.21}) and (\ref{4.33}), we eventually arrive at the conclusion that
		\begin{align}
			&\|\delta\tilde{P}\|^l_{\tilde{L}^\infty_{t}(\dot{B}^{\frac{d}{p}-1}_{p,1})}+\|\delta\tilde{P}\|^l_{L^1_{t}(\dot{B}^{\frac{d}{p}+1}_{p,1})}+\|\delta\tilde{\mathbf{u}}\|^l_{\tilde{L}^\infty_{t}(\dot{B}^{\frac{d}{p}}_{p,1})}+\|\delta\tilde{\mathbf{u}}\|^l_{L^1_{t}(\dot{B}^{\frac{d}{p}}_{p,1})}\notag\\
			&+\|\delta\tilde{c}\|^l_{\tilde{L}^\infty_{t}(\dot{B}^{\frac{d}{2}-1}_{2,1})}\|(\delta\tilde{P},\delta\tilde{\mathbf{u}},\delta\tilde{c})\|^h_{\tilde{L}_t^\infty(\dot{B}_{2,1}^{\frac{d}{2}})}+\|(\delta\tilde{P},\delta\tilde{\mathbf{u}})\|^h_{\tilde{L}_t^1(\dot{B}_{2,1}^{\frac{d}{2}})}\notag\\
			\lesssim&\int_0^T\Big(1+\|(\tilde{P}_1,\tilde{P}_2)\|^l_{\dot{B}^{\frac{d}{p}-1}_{p,1}}+\|(\tilde{\mathbf{u}}_1,\tilde{\mathbf{u}}_2)\|^l_{\dot{B}^{\frac{d}{p}}_{p,1}}
+\|(\tilde{c}_1,\tilde{c}_2)\|^l_{\dot{B}^{\frac{d}{2}-1}_{2,1}}\notag\\
			&\ \ \ \ \ \ \ \ +\|(\tilde{P}_1,\tilde{P}_2,\tilde{\mathbf{u}}_1,\tilde{\mathbf{u}}_2,\tilde{c}_1,\tilde{c}_2)\|^h_{\dot{B}^{\frac{d}{2}}_{2,1}}\Big)^2\notag\\
			&\ \ \ \ \cdot\big(\|\delta\tilde{P}\|^l_{\dot{B}^{\frac{d}{p}-1}_{p,1}}+\|\delta\tilde{\mathbf{u}}\|^l_{\dot{B}^{\frac{d}{p}}_{p,1}}
+\|\delta\tilde{c}\|^l_{\dot{B}^{\frac{d}{2}-1}_{2,1}}+\|(\delta\tilde{P},\delta\tilde{\mathbf{u}},\delta\tilde{c})\|^h_{\dot{B}^{\frac{d}{2}}_{2,1}}\big).
		\end{align}
		Note	
		\begin{align}
			&\|(\tilde{P}_1,\tilde{P}_2)\|^l_{\dot{B}^{\frac{d}{p}-1}_{p,1}}+\|(\tilde{\mathbf{u}}_1,\tilde{\mathbf{u}}_2)\|^l_{\dot{B}^{\frac{d}{p}}_{p,1}}
+\|(\tilde{c}_1,\tilde{c}_2)\|^l_{\dot{B}^{\frac{d}{2}-1}_{2,1}}
+\|(\tilde{P}_1,\tilde{P}_2,\tilde{\mathbf{u}}_1,\tilde{\mathbf{u}}_2,\tilde{c}_1,\tilde{c}_2)\|^h_{\dot{B}^{\frac{d}{2}}_{2,1}}\notag\\
			&<\infty.
		\end{align}
		Thus, Gronwall inequality yields  $(\tilde{P}_1,\tilde{\mathbf{u}}_1,\tilde{c}_1)=(\tilde{P}_2,\tilde{\mathbf{u}}_2,\tilde{c}_2)$ in the space (\ref{4.18}). On the other hand, by using the relation $\dot{B}^{\frac{d}{2}+1}_{2,1}\hookrightarrow\dot{B}^{\frac{d}{2}}_{2,1}$ in the high frequency, the uniqueness of solution is still true in the space
		\begin{align}
			\big\{(&\tilde{P},\tilde{\mathbf{u}},\tilde{c}) : \tilde{P}^l\in \mathcal{C}([0,T_*];\dot{B}^{\frac{d}{p}-1}_{p,1}), \tilde{\mathbf{u}}^l\in \mathcal{C}([0,T_*];\dot{B}^{\frac{d}{p}}_{p,1}), \tilde{c}^l\in \mathcal{C}([0,T_*];\dot{B}^{\frac{d}{2}-1}_{2,1}),\notag\\
			&(\tilde{P}^h, \tilde{\mathbf{u}}^h, \tilde{c}^h)\in \mathcal{C}([0,T_*];\dot{B}^{\frac{d}{2}+1}_{2,1})\big\}.
		\end{align}
\end{enumerate}
		\quad\quad Then, the proof is completed.\end{proof}

\section{Decay rate}
\quad\quad In this section, we shall follow the method used in \cite{Xin} to get the decay rate of the solutions constructed in the previous section. First, we bound the evolution of negative Besov norm, which is the main ingredient in deriving the Lyapunov-type inequality for energy norms.
\subsection{The regularity evolution of negative Besov norm}
\quad\quad In this subsection, we shall derive the following key lemma.
\begin{lemma}\label{5 1.1}
Let $2\leq p\leq4$ if $d=1$ or $2\leq p\leq \min\{4,\frac{2d}{d-2}\}$ if $d\geq 2$, and $(\tilde{P},\tilde{\mathbf{u}},\tilde{c})$ be the global solution to (\ref{1.5}) satisfying (\ref{1.6}). In addition to $\mathcal{X}_{p,0}\leq\eta_1$, assume further that (\ref{1.10}) holds. Then, we have
\begin{align}
\|(\tilde{P},\tilde{\mathbf{u}},\tilde{c})\|^l_{L_t^\infty \dot{B}_{p,\infty}^{-\sigma_1}}\leq \mathcal{D}_{p,0},\ \ (-\frac{d}{p}\leq-\sigma_1<\frac{d}{p}-1)\label{5.0}
\end{align}
with $\mathcal{D}_{p,0}=\|\tilde{P}_0,\tilde{\mathbf{u}}_0,\tilde{c}_0\|^l_{\dot{B}_{p,\infty}^{-\sigma_1}}+\mathcal{X}^2_{p,0}+\mathcal{X}^4_{p,0}.$
\end{lemma}

As the same procedure in the priori estimate, we have
\begin{align}
&\|\tilde{P}\|^l_{L_t^\infty \dot{B}_{p,\infty}^{-\sigma_1}}+\|\tilde{P}\|^l_{L_t^1\dot{B}_{p,\infty}^{-\sigma_1+2}}+\|\mathbf{Z}\|^l_{L_t^\infty \dot{B}_{p,\infty}^{-\sigma_1}}+\|\mathbf{Z}\|^l_{L_t^1\dot{B}_{p,\infty}^{-\sigma_1}}+\|\tilde{c}\|^l_{L_t^\infty \dot{B}_{p,\infty}^{-\sigma_1}}\notag\\
\lesssim &\|(\tilde{P}_0,\mathbf{Z}_0,\tilde{c}_0)\|^l_{\dot{B}_{p,\infty}^{-\sigma_1}}+\|(\tilde{G}_1,\tilde{G}_2)\|^l_{L_t^1\dot{B}_{p,\infty}^{-\sigma_1}}
+\int_0^T\|\nabla\tilde{\mathbf{u}}\|_{L^\infty}\|\tilde{c}\|_{\dot{B}_{2,\infty}^{-\sigma_1}}+\int_0^T\|{\rm div}\tilde{\mathbf{u}}\|_{L^\infty}\|\tilde{c}\|_{\dot{B}_{2,\infty}^{-\sigma_1}}\notag\\
\lesssim &\|(\tilde{P}_0,\mathbf{\mathbf{Z}}_0,\tilde{c}_0)\|^l_{\dot{B}_{p,\infty}^{-\sigma_1}}+\|(\tilde{G}_1,\tilde{G}_2)\|^l_{L_t^1\dot{B}_{p,\infty}^{-\sigma_1}}
+\mathcal{X}_p(t)\|\tilde{c}\|_{L_t^\infty\dot{B}_{2,\infty}^{-\sigma_1}}.\label{5.1}
\end{align}
Then, we shall bound $\|\tilde{G}_1\|^l_{L_t^1\dot{B}_{p,\infty}^{-\sigma_1}}$ and $\|\tilde{G}_2\|^l_{L_t^1\dot{B}_{p,\infty}^{-\sigma_1}}$. In fact,
\begin{align}
\|\tilde{G}_1&\|^l_{L_t^1\dot{B}_{p,\infty}^{-\sigma_1}}\lesssim\|h(\tilde{P},\tilde{c}){\rm div}\tilde{\mathbf{u}}\|^l_{L_t^1\dot{B}_{p,\infty}^{-\sigma_1}}+\|\tilde{\mathbf{u}}\cdot\nabla\tilde{P}\|^l_{L_t^1\dot{B}_{p,\infty}^{-\sigma_1}}\notag\\
&\lesssim\|(\tilde{P},\tilde{c})\|_{L_t^\infty\dot{B}_{p,1}^{\frac{d}{p}}}\|{\rm div}\tilde{\mathbf{u}}\|_{L_t^1\dot{B}_{p,\infty}^{-\sigma_1}}
+\|\tilde{\mathbf{u}}\|_{L_t^1\dot{B}_{p,1}^{\frac{d}{p}}}\|\nabla\tilde{P}\|_{L_t^\infty\dot{B}_{p,\infty}^{-\sigma_1}}\notag\\ &\lesssim\|(\tilde{P},\tilde{c})\|_{L_t^\infty\dot{B}_{p,1}^{\frac{d}{p}}}(\|\tilde{\mathbf{u}}\|^l_{L_t^1\dot{B}_{p,\infty}^{-\sigma_1}}
+\|\tilde{\mathbf{u}}\|^h_{L_t^\infty\dot{B}_{2,1}^{\frac{d}{2}+1}})
+\|\tilde{\mathbf{u}}\|_{L_t^1\dot{B}_{p,1}^{\frac{d}{p}}}(\|\tilde{P}\|^l_{L_t^\infty\dot{B}_{p,\infty}^{-\sigma_1}}
+\|\tilde{P}\|^h_{L_t^\infty\dot{B}_{2,1}^{\frac{d}{2}+1}})\notag\\
&\lesssim\mathcal{X}^2_p(t)+\mathcal{X}_p(t)(\|\tilde{P}\|^l_{L_t^\infty\dot{B}_{p,\infty}^{-\sigma_1}}
+\|\tilde{\mathbf{u}}\|^l_{L_t^\infty\dot{B}_{p,\infty}^{-\sigma_1}}).\label{5.2}
\end{align}
From the relation $\frac{1}{m}-\frac{1}{m_\infty}=M'_{\tilde{P}}(0)\tilde{P}+M'_{\tilde{c}}(0)\tilde{c}+\tilde{M}(\tilde{P},\tilde{c})\tilde{P}\tilde{c}$, we have
\begin{align}
\|\tilde{G}_2\|&^l_{L_t^1\dot{B}_{p,\infty}^{-\sigma_1}}\lesssim\|\tilde{\mathbf{u}}\cdot\nabla\tilde{\mathbf{u}}\|^l_{L_t^1\dot{B}_{p,\infty}^{-\sigma_1}}
+\|(\frac{1}{m}-\frac{1}{m_\infty})\nabla\tilde{P}\|^l_{L_t^1\dot{B}_{p,\infty}^{-\sigma_1}}\notag\\
&\lesssim\|\tilde{\mathbf{u}}\|_{L_t^1\dot{B}_{p,1}^{\frac{d}{p}}}(\|\tilde{\mathbf{u}}\|^l_{L_t^\infty\dot{B}_{p,\infty}^{-\sigma_1}}
+\|\tilde{\mathbf{u}}\|^h_{L_t^\infty\dot{B}_{2,1}^{\frac{d}{2}+1}})+\|(M'_{\tilde{P}}(0)\tilde{P}+M'_{\tilde{c}}(0)\tilde{c}
+\tilde{M}(\tilde{P},\tilde{c})\tilde{P}\tilde{c})\nabla\tilde{P}\|^l_{L_t^1\dot{B}_{p,\infty}^{-\sigma_1}}\notag\\
&\lesssim\|\tilde{\mathbf{u}}\|_{L_t^1\dot{B}_{p,1}^{\frac{d}{p}}}(\|\tilde{\mathbf{u}}\|^l_{L_t^\infty\dot{B}_{p,\infty}^{-\sigma_1}}
+\|\tilde{\mathbf{u}}\|^h_{L_t^\infty\dot{B}_{2,1}^{\frac{d}{2}+1}})+\|\tilde{P}\nabla\tilde{P}\|^l_{L_t^1\dot{B}_{p,\infty}^{-\sigma_1}}
+\|\tilde{c}\nabla\tilde{P}\|^l_{L_t^1\dot{B}_{p,\infty}^{-\sigma_1}}\notag\\
&\quad+\|\tilde{M}(\tilde{P},\tilde{c})\tilde{P}\tilde{c}\nabla\tilde{P}\|^l_{L_t^1\dot{B}_{p,\infty}^{-\sigma_1}}\notag\\
&\lesssim\|\tilde{\mathbf{u}}\|_{L_t^1\dot{B}_{p,1}^{\frac{d}{p}}}(\|\tilde{\mathbf{u}}\|^l_{L_t^\infty\dot{B}_{p,\infty}^{-\sigma_1}}
+\|\tilde{\mathbf{u}}\|^h_{L_t^\infty\dot{B}_{2,1}^{\frac{d}{2}+1}})+\|\tilde{P}\|_{L_t^\infty\dot{B}_{p,\infty}^{-\sigma_1}}\|\nabla\tilde{P}\|_{L_t^1\dot{B}_{p,1}^{\frac{d}{p}}}\notag\\
&\quad+\|\nabla \tilde{P}\|_{L_t^1\dot{B}_{p,1}^{\frac{d}{p}}}\|\tilde{c}\|_{L_t^\infty\dot{B}_{p,\infty}^{-\sigma_1}}
+\|\tilde{M}(\tilde{P},\tilde{c})\|_{L_t^\infty\dot{B}_{p,1}^{\frac{d}{p}}}\|\tilde{P}\|_{L_t^\infty\dot{B}_{p,\infty}^{-\sigma_1}}\|\tilde{c}\|_{L_t^\infty\dot{B}_{p,1}^{\frac{d}{p}}}\|\nabla\tilde{P}\|_{L_t^1\dot{B}_{p,1}^{\frac{d}{p}}}\notag\\
&\lesssim\mathcal{X}^2_p(t)\!+\!\mathcal{X}^4_p(t)\!+\!\mathcal{X}_p(t)(\|\tilde{\mathbf{u}}\|^l_{L_t^\infty\dot{B}_{p,\infty}^{-\sigma_1}}
+\|\tilde{c}\|^l_{L_t^\infty\dot{B}_{p,\infty}^{-\sigma_1}})+(\mathcal{X}_p(t)\!+\!\mathcal{X}^3_p(t))\|\tilde{P}\|^l_{L_t^\infty\dot{B}_{p,\infty}^{-\sigma_1}},\label{5.3}
\end{align}
and
\begin{align}
&\|\tilde{P}\|^l_{L_t^\infty \dot{B}_{p,\infty}^{-\sigma_1}}+\|\tilde{P}\|^l_{L_t^1\dot{B}_{p,\infty}^{-\sigma_1+2}}+\|\mathbf{Z}\|^l_{L_t^\infty \dot{B}_{p,\infty}^{-\sigma_1}}+\|\mathbf{Z}\|^l_{L_t^1\dot{B}_{p,\infty}^{-\sigma_1}}+\|\tilde{c}\|^l_{L_t^\infty \dot{B}_{p,\infty}^{-\sigma_1}}\notag\\
\lesssim& \|(\tilde{P}_0,\mathbf{Z}_0,\tilde{c})\|^l_{\dot{B}_{p,\infty}^{-\sigma_1}}+\mathcal{X}^2_p(t)\!+\!\mathcal{X}^4_p(t)\!+\!\mathcal{X}_p(t)(\|\tilde{\mathbf{u}}\|^l_{L_t^\infty\dot{B}_{p,\infty}^{-\sigma_1}}
+\|\tilde{c}\|^l_{L_t^\infty\dot{B}_{p,\infty}^{-\sigma_1}})\notag\\
&+(\mathcal{X}_p(t)\!+\!\mathcal{X}^3_p(t))\|\tilde{P}\|^l_{L_t^\infty\dot{B}_{p,\infty}^{-\sigma_1}}.\label{5.4}
\end{align}
Noting that $\mathbf{Z}=\kappa_2\nabla \tilde{P}+\alpha\tilde{\mathbf{u}}$, which gives
\begin{align}
&\|\tilde{P}\|^l_{L_t^\infty \dot{B}_{p,\infty}^{-\sigma_1}}+\|\tilde{P}\|^l_{L_t^1\dot{B}_{p,\infty}^{-\sigma_1+2}}+\|\tilde{\mathbf{u}}\|^l_{L_t^\infty \dot{B}_{p,\infty}^{-\sigma_1}}+\|\tilde{\mathbf{u}}\|^l_{L_t^1\dot{B}_{p,\infty}^{-\sigma_1}}+\|\tilde{c}\|^l_{L_t^\infty \dot{B}_{p,\infty}^{-\sigma_1}}\notag\\
\lesssim& \|(\tilde{P}_0,\tilde{\mathbf{u}}_0,\tilde{c})\|^l_{\dot{B}_{p,\infty}^{-\sigma_1}}+\mathcal{X}^2_p(t)\!+\!\mathcal{X}^4_p(t),\label{5.5}
\end{align}
and hence,
\begin{align}
\|(\tilde{P},\tilde{\mathbf{u}},\tilde{c})\|^l_{L_t^\infty \dot{B}_{p,\infty}^{-\sigma_1}}\leq \mathcal{D}_{p,0}.\label{5.6}
\end{align}
Consequently, we complete the proof of Lemma \ref{5 1.1}.
\subsection{The proof of Theorem \ref{1 1.2}}
\quad\quad This section is devoted to proving Theorem \ref{1 1.2}. From Section 2, we can get the following inequality (see the derivation of \eqref{3.11}, \eqref{3.16} and \eqref{3.49} for more details):
\begin{align}
&\frac{d}{dt}(\|\tilde{P}\|_{\dot{B}_{p,1}^{\frac{d}{p}-1}}^l+\|\tilde{\mathbf{u}}\|_{\dot{B}_{p,1}^{\frac{d}{p}}}^l+\|(\tilde{P},\tilde{\mathbf{u}})\|_{\dot{B}_{2,1}^{\frac{d}{2}+1}}^h)
+\|\tilde{P}\|_{\dot{B}_{p,1}^{\frac{d}{p}+1}}^l+\|\tilde{\mathbf{u}}\|_{\dot{B}_{p,1}^{\frac{d}{p}}}^l+\|(\tilde{P},\tilde{\mathbf{u}})\|_{\dot{B}_{2,1}^{\frac{d}{2}+1}}^h\notag\\
&\quad\leq
(1+\|(\tilde{P},\tilde{c})\|_{\dot{B}_{p,1}^{\frac{d}{p}-1}}^l+\|\tilde{\mathbf{u}}\|_{\dot{B}_{p,1}^{\frac{d}{p}}}^l+\|(\tilde{P},\tilde{\mathbf{u}},\tilde{c})\|_{\dot{B}_{2,1}^{\frac{d}{2}+1}}^h)^2\notag\\
&\quad\quad\cdot(\|(\tilde{P},\tilde{c})\|_{\dot{B}_{p,1}^{\frac{d}{p}-1}}^l+\|\tilde{\mathbf{u}}\|_{\dot{B}_{p,1}^{\frac{d}{p}}}^l+\|(\tilde{P},\tilde{\mathbf{u}},\tilde{c})\|_{\dot{B}_{2,1}^{\frac{d}{2}+1}}^h)\notag\\
&\quad\quad\cdot(\|\tilde{P}\|_{\dot{B}_{p,1}^{\frac{d}{p}+1}}^l+\|\tilde{\mathbf{u}}\|_{\dot{B}_{p,1}^{\frac{d}{p}}}^l+\|(\tilde{P},\tilde{\mathbf{u}})\|_{\dot{B}_{2,1}^{\frac{d}{2}+1}}^h),\label{5.7}
\end{align}
By the proof of the global existence of the system (\ref{1.5}), the following estimate holds:
\begin{align}
\|(\tilde{P},\tilde{c})\|_{\dot{B}_{p,1}^{\frac{d}{p}-1}}^l+\|\tilde{\mathbf{u}}\|_{\dot{B}_{p,1}^{\frac{d}{p}}}^l+\|(\tilde{P},\tilde{\mathbf{u}},\tilde{c})\|_{\dot{B}_{2,1}^{\frac{d}{2}+1}}^h\leq c_0\ll 1,\label{5.8}
\end{align}
thus, we can infer from (\ref{5.7}) that
\begin{align}
\frac{d}{dt}(\|\tilde{P}\|_{\dot{B}_{p,1}^{\frac{d}{p}-1}}^l+\|\tilde{\mathbf{u}}\|_{\dot{B}_{p,1}^{\frac{d}{p}}}^l+\|(\tilde{P},\tilde{\mathbf{u}})\|_{\dot{B}_{2,1}^{\frac{d}{2}+1}}^h)
+\bar{c}(\|\tilde{P}\|_{\dot{B}_{p,1}^{\frac{d}{p}+1}}^l+\|\tilde{\mathbf{u}}\|_{\dot{B}_{p,1}^{\frac{d}{p}}}^l+\|(\tilde{P},\tilde{\mathbf{u}})\|_{\dot{B}_{2,1}^{\frac{d}{2}+1}}^h)
\leq 0.\label{5.9}
\end{align}

In what follows, we using interpolation inequalities to get the desired time-decay estimates. On the one hand, owing to $-\sigma_1<\frac{d}{p}-1<\frac{d}{p}+1$, we get
\begin{align}
\|\tilde{P}\|_{\dot{B}_{p,1}^{\frac{d}{p}-1}}^l\lesssim(\|\tilde{P}\|_{\dot{B}_{p,\infty}^{-\sigma_1}}^l)^{\eta_1}(\|\tilde{P}\|_{\dot{B}_{p,1}^{\frac{d}{p}+1}}^l)^{1-\eta_1},\label{5.10}
\end{align}
where $\eta_1=\frac{2}{\frac{d}{p}+1+\sigma_1}\in(0,1)$. By virtue of (\ref{5.0}), one has
\begin{align}
\|\tilde{P}\|_{\dot{B}_{p,1}^{\frac{d}{p}+1}}^l\geq C(\|\tilde{P}\|_{\dot{B}_{p,1}^{\frac{d}{p}-1}}^l)^{\frac{1}{1-\eta_1}}.\label{5.11}
\end{align}
On the other hand, from (\ref{5.8}), it obvious that
\begin{align}
\|\tilde{\mathbf{u}}\|_{\dot{B}_{p,1}^{\frac{d}{p}}}^l\geq C(\|\tilde{\mathbf{u}}\|_{\dot{B}_{p,1}^{\frac{d}{p}}}^l)^{\frac{1}{1-\eta_1}},
\|(\tilde{P},\tilde{\mathbf{u}})\|_{\dot{B}_{2,1}^{\frac{d}{2}+1}}^h\geq C(\|(\tilde{P},\tilde{\mathbf{u}})\|_{\dot{B}_{2,1}^{\frac{d}{2}+1}}^h)^{\frac{1}{1-\eta_1}}.\label{5.12}
\end{align}
Combining (\ref{5.9}), (\ref{5.11}) and (\ref{5.12}), we conclude that there exists a positive constant $\tilde{c}_0$ such that the following Lyapunov-type inequality holds
\begin{align}
&\frac{d}{dt}(\|\tilde{P}\|_{\dot{B}_{p,1}^{\frac{d}{p}-1}}^l+\|\tilde{\mathbf{u}}\|_{\dot{B}_{p,1}^{\frac{d}{p}}}^l+\|(\tilde{P},\tilde{\mathbf{u}})\|_{\dot{B}_{2,1}^{\frac{d}{2}+1}}^h)\notag\\
&+\tilde{c}_0(\|\tilde{P}\|_{\dot{B}_{p,1}^{\frac{d}{p}-1}}^l+\|\tilde{\mathbf{u}}\|_{\dot{B}_{p,1}^{\frac{d}{p}}}^l+\|(\tilde{P},\tilde{\mathbf{u}})\|_{\dot{B}_{2,1}^{\frac{d}{2}+1}}^h)^{1+\frac{2}{\frac{d}{p}-1+\sigma_1}}
\leq 0.
\end{align}
By solving the above differential inequality, we obtain
\begin{align}
\|\tilde{P}\|_{\dot{B}_{p,1}^{\frac{d}{p}-1}}^l+\|\tilde{\mathbf{u}}\|_{\dot{B}_{p,1}^{\frac{d}{p}}}^l+\|(\tilde{P},\tilde{\mathbf{u}})\|_{\dot{B}_{2,1}^{\frac{d}{2}+1}}^h\leq C(1+t)^{-\frac{\frac{d}{p}-1+\sigma_1}{2}}.\label{5.14}
\end{align}
In addition, if $-\sigma_1<\sigma<\frac{d}{p}-1$, then by employing interpolation inequality, we have
\begin{align}
\|\tilde{P}\|_{\dot{B}_{p,1}^\sigma}^l\lesssim(\|\tilde{P}\|_{\dot{B}_{p,\infty}^{-\sigma_1}}^l)^{\eta_2}(\|\tilde{P}\|_{\dot{B}_{p,1}^{\frac{d}{p}-1}}^l)^{1-\eta_2},\ \eta_2=\frac{\frac{d}{p}-1-\sigma}{\frac{d}{p}-1+\sigma_1},\label{5.15}
\end{align}
which together with (\ref{5.0}) gives
\begin{align}
\|\tilde{P}\|_{\dot{B}_{p,1}^\sigma}^l\lesssim(1+t)^{-\frac{\sigma_1+\sigma}{2}}.\label{5.16}
\end{align}
Similarly, if $-\sigma_1<\sigma<\frac{d}{p}$, then employing interpolation inequality gives that
\begin{align}
\|\tilde{\mathbf{u}}\|_{\dot{B}_{p,1}^\sigma}^l\lesssim(\|\tilde{\mathbf{u}}\|_{\dot{B}_{p,\infty}^{-\sigma_1}}^l)^{\eta_3}(\|\tilde{\mathbf{u}}\|_{\dot{B}_{p,1}^{\frac{d}{p}}}^l)^{1-\eta_3},\ \eta_3=\frac{\frac{d}{p}-\sigma}{\frac{d}{p}+\sigma_1}.\label{5.17}
\end{align}
It combines with (\ref{5.0}) implies that
\begin{align}
\|\tilde{\mathbf{u}}\|_{\dot{B}_{p,1}^\sigma}^l\lesssim(1+t)^{-\frac{\frac{d}{p}-1+\sigma_1}{\frac{d}{p}+\sigma_1}\frac{\sigma_1+\sigma}{2}}.\label{5.18}
\end{align}
Next, we shall improve the decay of $\mathbf{u}$. In fact, from $(\ref{1.5})_2$, we can get the following inequality by directly calculate:
\begin{align}
\|\tilde{\mathbf{u}}&\|_{\dot{B}_{p,1}^\sigma}^l\lesssim e^{-\alpha t}\|\tilde{\mathbf{u}}_0\|_{\dot{B}_{p,1}^\sigma}^l+\int_0^t e^{-\alpha(t-\tau)}
(\|\nabla\tilde{P}\|_{\dot{B}_{p,1}^\sigma}^l+\|\tilde{G}_2\|_{\dot{B}_{p,1}^\sigma}^l)d\tau\notag\\
&\lesssim e^{-\alpha t}\|\tilde{\mathbf{u}}_0\|_{\dot{B}_{p,1}^\sigma}^l+\int_0^t e^{-\alpha(t-\tau)}(\|\nabla\tilde{P}\|_{\dot{B}_{p,1}^\sigma}^l+\|\tilde{\mathbf{u}}\|_{\dot{B}_{p,1}^{\frac{d}{p}}}\|\nabla\tilde{\mathbf{u}}\|_{\dot{B}_{p,1}^\sigma}
+\|(\tilde{P},\tilde{c})\|_{\dot{B}_{p,1}^{\frac{d}{p}}}\|\nabla\tilde{P}\|_{\dot{B}_{p,1}^\sigma})\notag\\
&\lesssim e^{-\alpha t}\|\tilde{\mathbf{u}}_0\|_{\dot{B}_{p,1}^\sigma}^l+\int_0^t e^{-\alpha(t-\tau)}\big[\|\tilde{P}\|_{\dot{B}_{p,1}^{\sigma+1}}^l
+(\|\tilde{\mathbf{u}}\|^l_{\dot{B}_{p,1}^{\frac{d}{p}}}+\|\tilde{\mathbf{u}}\|^h_{\dot{B}_{2,1}^{\frac{d}{2}+1}})(\|\tilde{\mathbf{u}}\|_{\dot{B}_{p,1}^{\sigma+1}}^l+\|\tilde{\mathbf{u}}\|_{\dot{B}_{2,1}^{\frac{d}{2}+1}}^h)\notag\\
&\ \ \ \ \ +(\|\tilde{P}\|^l_{\dot{B}_{p,1}^{\frac{d}{p}}}+\|\tilde{P}\|^h_{\dot{B}_{2,1}^{\frac{d}{2}+1}})(\|\tilde{P}\|_{\dot{B}_{p,1}^{\sigma+1}}^l+\|\tilde{P}\|_{\dot{B}_{2,1}^{\frac{d}{2}+1}}^h)
+C(\|\tilde{P}\|_{\dot{B}_{p,1}^{\sigma+1}}^l+\|\tilde{P}\|_{\dot{B}_{2,1}^{\frac{d}{2}+1}}^h)\big].\label{5.19}
\end{align}
To use the corresponding estimate \eqref{5.14} and \eqref{5.16}, we need $\sigma+1\leq\frac{d}{p}-1$. Then we have
\begin{align}
\|\tilde{\mathbf{u}}\|_{\dot{B}_{p,1}^\sigma}^l&\lesssim e^{-\alpha t}\|\tilde{\mathbf{u}}_0\|_{\dot{B}_{p,1}^\sigma}^l+\int_0^t e^{-\alpha(t-\tau)}\bigg[(1+t)^{-\frac{\sigma_1+\sigma+1}{2}}
+(1+t)^{-\frac{\frac{d}{p}-1+\sigma_1}{2}}(1+t)^{-\frac{\frac{d}{p}-1+\sigma_1}{\frac{d}{p}+\sigma_1}\frac{\sigma_1+\sigma+1}{2}}\notag\\
&\quad\quad\quad+(1+t)^{-\frac{\frac{d}{p}-1+\sigma_1}{2}}\big((1+t)^{-\frac{\sigma+\sigma_1+1}{2}}+(1+t)^{-\frac{\frac{d}{p}-1+\sigma_1}{2}}\big)
+(1+t)^{-\frac{\frac{d}{p}-1+\sigma_1}{2}}\bigg],\label{5.20}
\end{align}
which further gives that
\begin{align}
\|\tilde{\mathbf{u}}\|_{\dot{B}_{p,1}^\sigma}^l\lesssim(1+t)^{-\frac{\sigma+\sigma_1+1}{2}},\ \ if\ \ \sigma\leq\frac{d}{p}-2.\label{5.21}
\end{align}
In the case $\sigma+1>\frac{d}{p}-1$, employing $\|\tilde{\mathbf{u}}\|_{\dot{B}_{p,1}^\sigma}^l\lesssim\|\tilde{\mathbf{u}}\|_{\dot{B}_{p,1}^{\frac{d}{p}-2}}^l$ implies that
\begin{align}
\|\tilde{\mathbf{u}}\|_{\dot{B}_{p,1}^\sigma}^l\lesssim\|\tilde{\mathbf{u}}\|_{\dot{B}_{p,1}^{\frac{d}{p}-2}}^l\lesssim(1+t)^{-\frac{\frac{d}{p}-1+\sigma_1}{2}},\ \ if\ \ \frac{d}{p}-2<\sigma\leq\frac{d}{p}.\label{5.22}
\end{align}
In the light of $-\sigma_1<\sigma\leq\frac{d}{p}-1$, we see that
\begin{align}
\|\tilde{P}\|_{\dot{B}_{p,1}^\sigma}^h\lesssim\|\tilde{P}\|_{\dot{B}_{p,1}^{\frac{d}{p}+1}}^h
\lesssim\|\tilde{P}\|_{\dot{B}_{2,1}^{\frac{d}{2}+1}}^h\lesssim(1+t)^{-\frac{\frac{d}{p}-1+\sigma_1}{2}},\label{5.23}
\end{align}
which together with \eqref{5.16} gives
\begin{align}
\|\tilde{P}\|_{\dot{B}_{p,1}^\sigma}&\lesssim\|\tilde{P}\|_{\dot{B}_{p,1}^\sigma}^l+\|\tilde{P}\|_{\dot{B}_{p,1}^\sigma}^h
\lesssim(1+t)^{-\frac{\sigma_1+\sigma}{2}}+(1+t)^{-\frac{\frac{d}{p}-1+\sigma_1}{2}}
\lesssim(1+t)^{-\frac{\sigma_1+\sigma}{2}}.\label{5.24}
\end{align}
When $-\sigma_1<\sigma\leq\frac{d}{p}$, we have
\begin{align}
\|\tilde{\mathbf{u}}\|_{\dot{B}_{p,1}^\sigma}^h\lesssim\|\tilde{\mathbf{u}}\|_{\dot{B}_{p,1}^{\frac{d}{p}+1}}^h
\lesssim\|\tilde{\mathbf{u}}\|_{\dot{B}_{2,1}^{\frac{d}{2}+1}}^h\lesssim(1+t)^{-\frac{\frac{d}{p}-1+\sigma_1}{2}},\label{5.25}
\end{align}
which combines with \eqref{5.21} and \eqref{5.22} gives
\begin{align}
\|\tilde{\mathbf{u}}\|_{\dot{B}_{p,1}^\sigma}&\lesssim\|\tilde{\mathbf{u}}\|_{\dot{B}_{p,1}^\sigma}^l+\|\tilde{\mathbf{u}}\|_{\dot{B}_{p,1}^\sigma}^h
\lesssim(1+t)^{-\frac{\sigma+\sigma_1+1}{2}}+(1+t)^{-\frac{\frac{d}{p}-1+\sigma_1}{2}}\notag\\
&\lesssim(1+t)^{-\frac{\sigma+\sigma_1+1}{2}},\ \ (-\sigma_1<\sigma\leq\frac{d}{p}-2),\label{5.26}
\end{align}
\begin{align}
\|\tilde{\mathbf{u}}\|_{\dot{B}_{p,1}^\sigma}&\lesssim\|\tilde{\mathbf{u}}\|_{\dot{B}_{p,1}^\sigma}^l+\|\tilde{\mathbf{u}}\|_{\dot{B}_{p,1}^\sigma}^h
\lesssim(1+t)^{-\frac{\frac{d}{p}-1+\sigma_1}{2}}+(1+t)^{-\frac{\frac{d}{p}-1+\sigma_1}{2}}\notag\\
&\lesssim(1+t)^{-\frac{\frac{d}{p}-1+\sigma_1}{2}},\ \ (\frac{d}{p}-2<\sigma\leq\frac{d}{p}).\label{5.27}
\end{align}
Hence, thanks to the embedding relation $\dot{B}_{p,1}^0(\mathbb{R}^d)\hookrightarrow L^p(\mathbb{R}^d)$, we deduce that
\begin{align}
\|\Lambda^\sigma\tilde{P}\|_{L^p}\lesssim(1+t)^{-\frac{\sigma_1+\sigma}{2}},\ \ (-\sigma_1<\sigma\leq\frac{d}{p}-1),\label{5.28}\\
\|\Lambda^\sigma\tilde{\mathbf{u}}\|_{L^p}\lesssim(1+t)^{-\frac{\sigma+\sigma_1+1}{2}},\ \ (-\sigma_1<\sigma\leq\frac{d}{p}-2),\label{5.29}\\
\|\Lambda^\sigma\tilde{\mathbf{u}}\|_{L^p}\lesssim(1+t)^{-\frac{\frac{d}{p}-1+\sigma_1}{2}},\ \ (\frac{d}{p}-2<\sigma\leq\frac{d}{p}).\label{5.30}
\end{align}
The proof of Theorem \ref{1 1.2} is completed.
\section{Appendix}
\quad\quad 
First, let $\mathcal {S}(\mathbb{R}^{d})$ be the Schwartz class of rapidly decreasing function. Given $f\in \mathcal
{S}(\mathbb{R}^{d})$, its Fourier transform $\mathcal {F}f=\widehat{f}$ is defined by $\widehat{f}(\xi)=\int_{\mathbb{R}^{n}}e^{-ix\cdot\xi}f(x)dx$. 
Let $(\chi, \varphi)$ be a couple of smooth functions valued in $[0,1]$ such that $\chi$ is supported in the ball $\{\xi\in\mathbb{R}^{d}:  \ |\xi|\leq\frac{4}{3}\}$, $\varphi$ is supported in the shell $\{\xi\in \mathbb{R}^{d}: \ \frac{3}{4}\leq|\xi|\leq\frac{8}{3}\}$,  $\varphi(\xi):=\chi(\xi/2)-\chi(\xi)$
and
$$\chi(\xi)+\sum_{j\geq0}\varphi(2^{-j}\xi)=1~ \mathrm{for}\ \forall \ \xi \in\mathbb{ R}^{d},
~~~\sum_{j\in \mathbb{Z}}\varphi(2^{-j}\xi)=1 ~\mathrm{for} \ \forall \ \xi \in \mathbb{R}^{d}\setminus\{0\}.$$
For $f\in \mathcal{S}'$, the homogeneous frequency localization operators $\dot{\Delta}_j$ and $\dot{S}_j$ are defined by
\begin{equation*}
	\dot{\Delta}_{j}f\triangleq\varphi(2^{-j}D)f=\mathcal{F}^{-1}(\varphi(2^{-j}\xi)\mathcal{F}f)\quad {\rm and}\quad \dot{S}_{j}f\triangleq\chi(2^{-j}D)f=\mathcal{F}^{-1}(\chi(2^{-j}\xi)\mathcal{F}f).
\end{equation*}
We denote the space $\mathcal{S}'_{h}(\mathbb{R}^d)$ by the dual space $\mathcal{S}'(\mathbb{R}^d)=\{f\in\mathcal{S}(\mathbb{R}^d):\,D^\alpha \hat{f}(0)=0\}$, which can also be identified by the quotient space of $\mathcal{S}'(\mathbb{R}^d)/{\mathbb{P}}$ with the polynomial space ${\mathbb{P}}$. The formal equality $$ f=\sum_{j\in\mathbb{Z}}\dot{\Delta}_jf $$ holds true for $f\in\mathcal{S}'_{h}(\mathbb{R}^d)$ and is called the homogeneous Littlewood-Paley decomposition, and then
\begin{equation}
	\dot{S}_jf=\sum_{q\le j-1}\dot{\Delta}_qf. \nonumber
\end{equation}
One easily verifies that with our choice of $\varphi$,
\begin{equation*}
	\dot{\Delta}_j\dot{\Delta}_qf\equiv0\quad \textrm{if}\quad|j-q|\ge
	2\quad \textrm{and} \quad
	\dot{\Delta}_j(\dot{S}_{q-1}f\dot{\Delta}_q f)\equiv0\quad \hbox{if}
	\quad |j-q|\ge 5.
\end{equation*}

\begin{definition}(Homogeneous Besov space)\label{A.1}
	For $s\in \mathbb{R}$ and $1\le p,r\le \infty$, the homogeneous Besov space $\dot{B}^{s}_{p,r}$ is defined by
	\begin{equation}\label{a.1}
		\dot{B}^s_{p,r}\triangleq\left\{f\in \mathcal{S}_h':||f||_{\dot{B}^s_{p,r}}<+\infty\right\},
	\end{equation}
	where
	\begin{equation}\label{a.2}
		||f||_{\dot{B}^s_{p,r}}\triangleq||2^{js}||\dot{\Delta}_jf||_{L^p}||_{\mathit{l}^r(\mathbb{Z})}.
	\end{equation}
\end{definition}

\begin{definition}(Chemin-Lerner spaces \cite{Chemin})\label{A.2}
	Let $T>0$, $s\in\mathbb{R}$, $1<r,p,q\le\infty$. The space $\widetilde{L}^q_{T}(\dot{B}^s_{p,r})$ is defined by
	\begin{equation}\label{a.3}
		\widetilde{L}^q_{T}(\dot{B}^s_{p,r})\triangleq\left\{f\in L^q(0,T;\mathcal{S}'_h):||f||_{\widetilde{L}^q_{T}(\dot{B}^s_{p,r})}<+\infty\right\},
	\end{equation}
	where
	\begin{equation}\label{a.4}
		||f||_{\widetilde{L}^q_{T}(\dot{B}^s_{p,r})}\triangleq||2^{js}||\dot{\Delta}_jf||_{L^q(0,T;L^p)}||_{\mathit{l}^r(\mathbb{Z})}.
	\end{equation}
\end{definition}
\begin{remark}\label{A.3}
	It holds that
	\begin{equation*}
		||f||_{\widetilde{L}^q_{T}(\dot{B}^s_{p,r})}\le||f||_{L^q_{T}(\dot{B}^s_{p,r})}\quad {\rm if}\quad r\ge q;\quad||f||_{\widetilde{L}^q_{T}(\dot{B}^s_{p,r})}\ge||f||_{L^q_{T}(\dot{B}^s_{p,r})}\quad {\rm if}\quad r\le q.
	\end{equation*}
\end{remark}
Restricting the above norms \eqref{a.2}, \eqref{a.4} to the low or high frequencies parts of distributions will be crucial in our approach. For example, let us fix some integer $j_0$ and set
\begin{equation*}
	||f||^l_{\dot{B}^{s}_{p,1}}\triangleq\sum_{j\le j_0}2^{js}||\dot{\Delta}_jf||_{L^p},\quad ||f||^h_{\dot{B}^{s}_{p,1}}\triangleq\sum_{j\ge j_0-1}2^{js}||\dot{\Delta}_jf||_{L^p};
\end{equation*}
\begin{equation*}
    ||f||^l_{\widetilde{L}^\infty_{T}(\dot{B}^s_{p,1})}\triangleq\sum_{j\le j_0}2^{js}||\dot{\Delta}_jf||_{L^\infty_T(L^p)},\ \ ||f||^h_{\widetilde{L}^\infty_{T}(\dot{B}^s_{p,1})}\triangleq\sum_{j\ge j_0-1}2^{js}||\dot{\Delta}_jf||_{L^\infty_T(L^p)}.
\end{equation*}

\begin{lemma}(Bernstein inequalities)\label{A.4}
Let $\mathscr{B}$ be a ball and $\mathscr{C}$ be a ring of $\mathbb{R}^d$.
For $\lambda>0$, integer $k\ge0$, $1\le p\le q\le \infty$ and a smooth homogeneous function $\sigma$ in $\mathbb{R}^d\backslash\{0\}$ of degree $m$, then
	\begin{equation*}
		||\nabla^kf||_{L^q}\le C^{k+1}\lambda^{k+d(\frac{1}{p}-\frac{1}{q})}||f||_{L^p},\quad {\rm whenever\ supp}\widehat{f}\subset\lambda\mathscr{B},
	\end{equation*}
	\begin{equation*}
		C^{-k-1}\lambda^k||f||_{L^q}\le||\nabla^kf||_{L^p}\le C^{k+1}\lambda^k||f||_{L^p},\quad {\rm whenever\ supp}\widehat{f}\subset\lambda\mathscr{C},
	\end{equation*}
	\begin{equation*}
		||\sigma(\nabla)f||_{L^q}\le C_{\sigma,m}\lambda^{m+d(\frac{1}{p}-\frac{1}{q})}||f||_{L^p},\quad {\rm whenever\ supp}\widehat{f}\subset\lambda\mathscr{C}.
	\end{equation*}
\end{lemma}

\begin{proposition}\cite{Bahouri}(Embedding for Besov space on $\mathbb{R}^d$)\label{A.5}
	\begin{itemize}
		\item For any $p\in[1,\infty]$, we have the continuous embedding $\dot{B}^{0}_{p,1}\hookrightarrow L^p\hookrightarrow\dot{B}^{0}_{p,\infty}$.
		\item If $s\in\mathbb{R}$, $1\le p_1\le p_2\le \infty$, and $1\le r_1\le r_2\le \infty$ then  $\dot{B}^{s}_{p_1,r_1}\hookrightarrow\dot{B}^{s-d(\frac{1}{p_1}-\frac{1}{p_2})}_{p_2,r_2}$.
		\item The space $\dot{B}^{\frac{d}{p}}_{p,1}$ is continuously embedded in the set of bounded continuous function (going to zero at infinity if, additionally, $p<\infty$).
	\end{itemize}
\end{proposition}

\begin{lemma}\label{2 2.7}\cite{Bahouri}
Let $s>0$, $1\leq p,r\leq\infty$, then we have
\begin{align}
\|ab\|_{\dot{B}^s_{p,r}}\lesssim \|a\|_\infty\|b\|_{\dot{B}_{p_r}^s}+\|b\|_{L^\infty}\|a\|_{\dot{B}_{p,r}^s}.
\end{align}
For $d\geq 1$ and $-\min\{\frac{d}{p},\frac{d}{p'}\}<s\leq\frac{d}{p}$ for $\frac{1}{p}+\frac{1}{p'}=1$, the following inequality holds:
\begin{align}
\|ab\|_{\dot{B}^s_{p,1}}\lesssim \|a\|_{\dot{B}_{p_1}^{\frac{d}{p}}}\|b\|_{\dot{B}_{p,1}^s}.
\end{align}
Finally, if $d\geq 1$ and $-\min\{\frac{d}{p},\frac{d}{p'}\}\leq s<\frac{d}{p}$ for $\frac{1}{p}+\frac{1}{p'}=1$, then we have
\begin{align}
\|ab\|_{\dot{B}^s_{p,\infty}}\lesssim \|a\|_{\dot{B}_{p_1}^{\frac{d}{p}}}\|b\|_{\dot{B}_{p,\infty}^s}.
\end{align}
\end{lemma}

In the low frequency, we shall use the classical estimates for the following two linear equations. One is the Cauchy problem of the parabolic equations:
\begin{equation}\label{2.1}
\left\{\begin{array}{ll}
u_t-\frac{\kappa_2^2}{\alpha}\Delta u=F,\\
u(0,x)=u_0(x),
\end{array}\right.
\end{equation}
where the unknown is $u=u(x,t)\in\mathbb{R}^n$ with $ t\geq0,\ x\in\mathbb{R}^d.$
\begin{lemma}\label{2 2.1}\cite{Crin-Barat6}
Let $\kappa_2>0,\ \alpha>0,\ d,\ n\geq 1,\ s_1\in\mathbb{R},\ 1\leq\rho_1,p_1\leq\infty$ and $T>0$ be given time. Assume $u_0^l\in\dot{B}_{p_1,1}^{s_1}$ and $F^l\in\tilde{L}_T^{\rho_1}(\dot{B}_{p_1,1}^{s_1-2+\frac{2}{\rho_1}})$. If $u$ is a solution to the Cauchy problem (\ref{2.1}), then for all $\tilde{\rho}_1\in[\rho_1,\infty]$, $u$ satisfies
$$
\|u\|^l_{\tilde{L}_T^{\tilde{\rho}_1}(\dot{B}_{p_1,1}^{s_1+\frac{2}{\tilde{\rho}_1}})}\lesssim \|u_0\|^l_{\dot{B}_{p_1,1}^{s_1}}+\|F\|^l_{\tilde{L}_T^{\rho_1}(\dot{B}_{p_1,1}^{s_1-2+\frac{2}{\rho_1}})}.
$$
\end{lemma}
Another is on the Cauchy problem of the damped equation
\begin{equation}\label{2.2}
\left\{\begin{array}{ll}
\frac{1}{\alpha}u_t+u=F,\\
u(0,x)=u_0(x),
\end{array}\right.
\end{equation}
where the unknown is $u=u(x,t)\in\mathbb{R}^n,\ t>0,\ x\in\mathbb{R}^d.$
\begin{lemma}\label{2 2.2}\cite{Crin-Barat6}
Let $\alpha>0,\ d,n\geq 1,\ s_1\in\mathbb{R},\ 1\leq\rho_1,p_1\leq\infty$ and $T>0$ be given time. Assume $u_0^l\in\dot{B}_{p_1,1}^{s_1}$ and $F^l\in\tilde{L}_T^{\rho_1}(\dot{B}_{p_1,1}^{s_1})$. If $u$ is a solution to the Cauchy problem (\ref{2.2}), then for all $\tilde{\rho}_1\in[\rho_1,\infty]$, $u$ satisfies
$$\|u\|^l_{\tilde{L}_T^{\tilde{\rho}_1}(\dot{B}_{p_1,1}^{s_1})}\lesssim \|u_0\|^l_{\dot{B}_{p_1,1}^{s_1}}+\|F\|^l_{\tilde{L}_T^{\rho_1}(\dot{B}_{p_1,1}^{s_1})}.$$
\end{lemma}

We also use the commutators estimates in low frequency.
\begin{lemma}\label{2 2.6}\cite{Bahouri}
Let $\sigma\in\mathbb{R}$, $1\leq r\leq\infty$, and $1\leq p\leq p_1\leq\infty$. Let $v$ be a vector field over $\mathbb{R}^d$. Assume that
\begin{align}
\sigma>-d\min\{\frac{1}{p_1},\frac{1}{p'}\}\ \ or\ \ \sigma>-1-d\min\{\frac{1}{p_1},\frac{1}{p_1}\}\ \ if\ \ {\rm div} v=0.
\end{align}
Define $R_j\triangleq [v\cdot\nabla,\dot\Delta_j]f$ (or $R_j\triangleq {\rm div}([v,\dot\Delta_j]f)$, if ${\rm div}v=0$). There exists a constant $C$, depending continuously on $p,\ p_1,\ \sigma$ and $d$, such that
\begin{align}
\|(2^{j\sigma}\|R_j\|_{L^p})_j\|_{l^r}\leq C\|\nabla v\|_{\dot{B}_{p_1,\infty}^{\frac{d}{p_1}}\cap L^\infty}\|f\|_{\dot{B}_{p,r}^\sigma},\ if\ \sigma<1+\frac{d}{p_1}.
\end{align}
Further, if $\sigma>0$ (or $\sigma>-1$, if ${\rm div}v=0$) and $\frac{1}{p_2}=\frac{1}{p}-\frac{1}{p_1}$, then
\begin{align}
\|(2^{j\sigma}\|R_j\|_{L^p})_j\|_{l^r}\leq C\|\nabla v\|_{L^\infty}\|f\|_{\dot{B}_{p,r}^\sigma}+\|\nabla f\|_{L^{p_2}}\|\nabla v\|_{\dot{B}_{p_1,r}^{\sigma-1}}.
\end{align}
In the limit case $\sigma=-\min(\frac{d}{p_1},\frac{d}{p'})$ [$or\ \sigma=-1-\min(\frac{d}{p_1},\frac{d}{p'}), if\ {\rm div}v=0$], we have
\begin{align}
\sup\limits_{j\geq -1}2^{j\sigma}\|R_j\|_{L^p}\leq C\|\nabla v\|_{\dot{B}_{p_1,1}^{\frac{d}{p_1}}}\|f\|_{\dot{B}_{p,\infty}^\sigma}.
\end{align}
\end{lemma}
In the high frequency, the next lemma pertains to commutators estimates.
\begin{lemma}\label{2 2.3}\cite{Crin-Barat1,Crin-Barat2,XuJ2}
Let $p\in [2,4]$ and $s>0$. Define $p^*\triangleq 2p/(p-2)$. For $j\in\mathbb{Z}$, denote $R_j\triangleq \dot{S}_{j-1}w\dot\Delta_jz-\dot\Delta_j(wz)$.
There exists a constant $C$ depending only on the threshold number $J_1$ between low and high frequencies and on $s,\ p,\ d$, such that
\begin{equation}\label{2.3}
\begin{array}{ll}
\sum\limits_{j\geq J_1}(2^{js}\|R_j\|_{L^2})\leq C(\|\nabla w\|_{\dot{B}_{p,1}^{\frac{d}{p}}}\|z\|^h_{\dot{B}_{2,1}^{s-1}}
+\|z\|^l_{\dot{B}_{p,1}^{\frac{d}{p}-\frac{d}{p^*}}}\|w\|^l_{\dot{B}_{p,1}^{\sigma_1}}\\[3mm]
+\|z\|_{\dot{B}_{p,1}^{\frac{d}{p}-k}}\|w\|^h_{\dot{B}_{2,1}^{s+k}}
+\|z\|^l_{\dot{B}_{p,1}^{\sigma_2}}\|\nabla w\|^l_{\dot{B}_{p,1}^{\frac{d}{p}-\frac{d}{p^*}}}),
\end{array}
\end{equation}
for any $k\geq 0$, $\sigma\geq s$ and $\sigma_1,\sigma_2\in\mathbb{R}$.
\end{lemma}


The next two lemmas are nonlinear composition estimates.
\begin{lemma}\label{2 2.4}\cite{XuJ2}
Let $f(n)$ be a smooth function such that $f(0)=0$. If $n^{h,J_\epsilon}\in\dot{B}_{2,1}^s(\mathbb{R}^d)$ and $n^{l,J_\epsilon}\in\dot{B}_{p,1}^{\frac{d}{p}}(\mathbb{R}^d)$ for $s>1$ and $s\geq\frac{d}{2}$, $2\leq p\leq 4$, if $d=1$, and $2\leq p\leq min\{4,\frac{2d}{d-2}\}$ if $d\geq 2$, then we have $f(n)^{h,J_\epsilon}\in\dot{B}_{2,1}^s(\mathbb{R}^d)$ and there is a positive constant $C$ independent of $\epsilon$ such that
\begin{equation}\label{2.4}
\begin{array}{ll}
\|f(n)\|_{\dot{B}_{2,1}^s}^{h,J_\epsilon}\leq C(1+\|n\|^{l,J_\epsilon}_{\dot{B}_{p,1}^{\frac{d}{p}}}+2^{(\frac{d}{2}-s)J_\epsilon}\|n\|^{h,J_\epsilon}_{\dot{B}_{2,1}^s})(\|n\|^{l,J_\epsilon}_{\dot{B}_{p,1}^{s+\frac{d}{p}-\frac{d}{2}}}
+\|n\|^{h,J_\epsilon}_{\dot{B}_{2,1}^s}),
\end{array}
\end{equation}
where $J_{\epsilon}$ is the threshold between low and high frequencies.
\end{lemma}

\begin{lemma}\cite{XuJ1}\label{A.13}
Let $m\in\mathbb{N}$ and $s>0$. Let $G$ be a function in $\mathcal{C}^{\infty}(\mathbb{R}^m\times\mathbb{R}^d)$ such that $G(0,\cdots,0)=0$. Then for every real valued functions $f_1,\cdots,f_m\in\dot{B}_{p,r}^s\cap L^\infty$, the function $G(f_1,\cdots,f_m)$ belongs to $\dot{B}_{p,r}^s\cap L^\infty$ and we have
  \begin{align}
		\|G(f_1,\cdots,f_m)\|_{\dot{B}_{p,r}^s}\leq C\|(f_1,\cdots,f_m)\|_{\dot{B}_{p,r}^s}
	\end{align}
with $C$ depending only on $\|f_i\|_{L^\infty}(i=1,\cdots,m)$, $G'_{f_i}$ (and high derivatives), $s,\ p$ and $d$.

In that case $s>-min(\frac{d}{p},\frac{d}{p^\ast})$, then $f_1,\cdots,f_m\in\dot{B}_{p,r}^s\cap\dot{B}_{p,1}^{\frac{d}{p}}$ implies that $G(f_1,\cdots,f_m)\in\dot{B}_{p,r}^s\cap\dot{B}_{p,1}^{\frac{d}{p}}$ and we have
    \begin{align}
		\|G(f_1,\cdots,f_m)\|_{\dot{B}_{p,r}^s}\leq C(1+\|f_1\|_{\dot{B}_{p,1}^{\frac{d}{p}}}+\cdots+\|f_m\|_{\dot{B}_{p,1}^{\frac{d}{p}}})\|(f_1,\cdots,f_m)\|_{\dot{B}_{p,r}^s}.
	\end{align}
\end{lemma}

We also show a product law to handle some nonlinear terms in the proof of the uniqueness.
\begin{lemma}\cite{Crin-Barat6}\label{A h}
	Let $s_1>0$, $2\le p\le 4$. Then, it holds that
	\begin{align}
		\|ab\|^h_{\dot{B}^{s_1}_{2,1}}\lesssim\|a\|_{\dot{B}^{\frac{d}{p}}_{p,1}}\|b\|^h_{\dot{B}^{s_1}_{2,1}}+\|b\|_{\dot{B}^{\frac{d}{p}}_{p,1}}\|a\|^h_{\dot{B}^{s_1}_{2,1}}+\|b\|^l_{\dot{B}^{\frac{d}{p}}_{p,1}}\|a\|^l_{\dot{B}^{\frac{d}{p}}_{p,1}}+\|b\|_{\dot{B}^{\frac{d}{p}}_{p,1}}\|a\|^l_{\dot{B}^{s_1-\frac{d}{2}+\frac{d}{p}}_{p,1}}.
	\end{align}
\end{lemma}

Next, we introduce typical estimates for the composition of functions in the proof of uniqueness.

\begin{lemma}\cite{Crin-Barat6}\label{A.F(m)}
	Assume that $F(m)$ is a smooth function satisfying $F'(0)=0$. Let $1\le p\le \infty$. For any couple $(m_1,m_2)$ of functions in $\dot{B}^{\frac{d}{p}}_{p,1}\cap L^\infty$, there exists a constant $C_{m_1,m_2}>0$ depending on $F''$ and $\|(m_1,m_2)\|_{L^\infty}$ such that
	\begin{itemize}
		\item Let $-\min\{\frac{d}{p},d(1-\frac{1}{p})\}<s\le\frac{d}{p}$ and $1\le r\le \infty$. Then, we have
		\begin{align}
			\|F(m_1)-F(m_2)\|_{\dot{B}^{s}_{p,r}}\le C\|(m_1,m_2)\|_{\dot{B}^{\frac{d}{p}}_{p,1}}\|m_1-m_2\|_{\dot{B}^{\frac{s}{p}}_{p,r}}.
		\end{align}
		\item In the limiting case $r=\infty$, for any $-\min\{\frac{d}{p},d(1-\frac{1}{p})\}\le s<\frac{d}{p}$, it holds that
		\begin{align}
			\|F(m_1)-F(m_2)\|_{\dot{B}^{s}_{p,\infty}}\le C\|(m_1,m_2)\|_{\dot{B}^{\frac{d}{p}}_{p,1}}\|m_1-m_2\|_{\dot{B}^{\frac{s}{p}}_{p,\infty}}.
		\end{align}
	\end{itemize}
\end{lemma}

\bigbreak\bigbreak
\noindent{\bf Funding}: The research was supported by National Natural Science Foundation of China (No. 11971100) and Natural Science Foundation of Shanghai (Grant No. 22ZR1402300).\\
{\bf Conflict of Interest}: The authors declare that they have no conflict of interest.

\bibliographystyle{plain}

\end{document}